\documentclass[a4paper,11pt]{article}
\usepackage{amsmath,amssymb,amsthm}
\usepackage{mathrsfs}
\usepackage{cases}
\usepackage{a4wide}
\usepackage{color,graphicx}
\usepackage[export]{adjustbox}
\usepackage[hang,small,bf]{caption}
\usepackage[subrefformat=parens]{subcaption}
\usepackage{here}
\usepackage{comment}
\PassOptionsToPackage{normalem}{ulem}
\usepackage{ulem}

\numberwithin{equation}{section}
\newtheorem{thm}{Theorem}[section]
\newtheorem{prop}[thm]{Proposition}

\newtheorem{rem}[thm]{Remark}

\newtheorem{asm}{Assumption}

\newcommand{\Section}[1]{Section~\ref{#1}}
\newcommand{\Figure}[1]{Figure~\ref{#1}}

\newcommand{\Proposition}[1]{Proposition~\ref{#1}}

\newcommand{\Assumption}[1]{Assumption~\ref{#1}}

\newcommand{\dL}[1]{\;\mathrm{d}\mathscr{L}^{#1}}
\newcommand{\dH}[1]{\;\mathrm{d}\mathscr{H}^{#1}}

\newcommand{\dd}[1]{\frac{\rm d}{{\rm d}#1}}
\newcommand{\ddt}{\dd{t}}

\newcommand{\bR}{{\mathbb R}}

\newcommand{\pOmega}[0]{\partial\Omega}
\newcommand{\closure}[1]{\overline{#1}}

\newcommand{\Fdd}[1]{#1} 
\newcommand{\Ddd}[1]{#1} 

\newcommand{\curve}[2]{\Gamma^{#1}_{#2}}

\newcommand{\chemicalVec}[2]{\bv{w}^{#1}_{#2}}

\newcommand{\chemicalDiscVec}[2]{\Ddd{\bv{W}^{#1}_{#2}}}
\newcommand{\identity}[0]{\vec{\operatorname{id}}}
\newcommand{\X}[2]{\vec{X}^{#1}_{#2}}
\newcommand{\innerproduct}[4]{\left<#1,#2\right>^{#3}_{#4}}
\newcommand{\innerproductround}[4]{\left(#1,#2\right)^{#3}_{#4}}

\newcommand{\normal}[2]{\Vec{\nu}^{#1}_{#2}}
\newcommand{\Normal}[1]{\Vec{A}\{#1\}}
\newcommand{\normalnoindex}[2]{\vec{\nu}^{#1}_{#2}}
\newcommand{\normalcontainer}[0]{\vec{\nu}_\Omega}

\newcommand{\conormal}[1]{\vec{\mu}_{#1}}

\newcommand{\weightnormal}[3]{\vec{\omega}^{#1}_{#2,#3}}
\newcommand{\weightnormalnoindex}[2]{\vec{\omega}^{#1}_{#2}}
\newcommand{\spacemeshpointsnum}[0]{K^m_{\Omega}}
\newcommand{\sgrad}[0]{\nabla_s} 

\newcommand{\Glinear}[2]{\pi^{h}_{#2}\left[\frac{\X{#1+1}{#2} - \identity}{\tau_{#1}}\cdot\weightnormalnoindex{#1}{#2}\right]}
\newcommand{\Glinearhomo}[2]{\pi^{h}_{#2}\left[\frac{\X{}{#2}}{\tau_{#1}}\cdot\weightnormalnoindex{#1}{#2}\right]}

\newcommand{\basisbulk}[2]{\Psi^{#1}_{#2}}
\newcommand{\metric}[1]{\uuline{#1}}
\newcommand{\zerovec}[0]{\bvzero}

\newcommand{\zerosumspace}[0]{T\Sigma}

\newcommand{\femspacebulkvector}[0]{\bv {S}^m}
\newcommand{\femspacebulkvectorsum}[0]{\bv {S}^m_{\Sigma}}

\newcommand{\femspacegamma}[0]{V^h(\Gamma^m)}
\newcommand{\femspacegammavector}[0]{\underline{V}^h_{\mathcal{T}}(\Gamma^m)}

\newcommand{\basiscurve}[2]{\Phi_{{#1},{#2}}^m}
\newcommand{\unitvector}[1]{\vec{e}_{#1}}

\newcommand{\curveindex}[2]{s^{#1}_{#2}}

\newcommand{\naturalset}[1]{\mathbb{N}_{\leq {#1}}}
\newcommand{\naturalindex}[1]{\mathbb{N}_{\leq {#1}}}
\newcommand{\phasecharacter}[0]{\bv \chi}

\newcommand{\dcmap}[0]{\mathcal O}

\newcommand{\normall}[2]{\Vec{\nu}^{#1}_{#2}}
\newcommand{\mT}{\mathcal{T}}
\newcommand{\mR}{\mathcal{R}}

\newcommand{\jumpPhase}[3]{\left[{#1}\right]^{#3}_{#2}}
\newcommand{\curvature}[1]{\varkappa_{#1}}
\newcommand{\curvatureDisc}[1]{\kappa_{#1}}
\newcommand{\anisotropy}[1]{\gamma_{#1}}
\newcommand{\tension}[1]{\varsigma_{#1}}
\newcommand{\tpindex}[2]{s^{#1}_{#2}}
\newcommand{\subdomains}[1]{\mR_{#1}[\Gamma(t)]}

\newcommand{\mobility}[1]{\beta_{#1}}
\newcommand{\kinetic}[1]{\rho_{#1}}
\newcommand{\tG}[0]{\widetilde{G}}
\newcommand{\volume}[1]{\operatorname{vol}\left(#1\right)}
\newcommand{\vVertex}[2]{\Vec{X}^{#1}_{#2}}
\newcommand{\vVertexSynonym}[2]{\Vec{q}^{#1}_{#2}}

\newcommand{\numSurfVertex}[2]{K^{#1}_{#2}}

\newcommand{\numSimplex}[2]{J^{#1}_{#2}}
\newcommand{\numTjVertex}[1]{Z_{#1}}
\newcommand{\polyhedral}[2]{\vec{\mathfrak{X}}^{#1}_{#2}}
\newcommand{\orderedSeq}[2]{\vec{\varrho}^{#1}_{#2}}
\newcommand{\refsurface}[1]{\Upsilon_{#1}}
\newcommand{\simplex}[2]{\sigma^{#1}_{#2}}
\newcommand{\inprodMassLamped}[4]{\left<#1,\,#2\right>^{#3}_{#4}}
\newcommand{\polygonCurve}[2]{\Gamma^{#1}_{#2}}
\newcommand{\femNormalVertex}[2]{\Vec{\omega}^{#1}_{#2}}
\newcommand{\mH}{\mathscr{H}}
\newcommand{\triangulation}[1]{\mathscr{T}^{#1}}
\newcommand{\femspaceVertex}[2]{V^{#1}_{#2}}
\newcommand{\femspaceVertexVector}[2]{\underline{V}^{#1}_{#2}}

\newcommand{\femspaceChemical}[1]{S^{#1}}
\newcommand{\femspaceBasisCurve}[2]{\Phi^{#1}_{#2}}

\def\bv#1{\mbox{\boldmath{$#1$}}}    

\def\bvzero{{\bf 0}}
\def\bvone{{\bf 1}}

\newcommand{\diag}{\operatorname{diag}}

\newcommand{\fix}[1]{{#1}}

\title{
A Parametric Finite Element Approach for \\ 
an Anisotropic Multi-Phase Mullins--Sekerka Problem \\
with Kinetic Undercooling}
\author{Tokuhiro Eto\thanks{Universit\'{e} Claude Bernard Lyon 1, CNRS, Centrale Lyon, INSA Lyon, Universit\'{e} Jean Monnet, ICJ UMR5208, 69622 Villeurbanne, France. E-mail: eto@math.univ-lyon1.fr}\and Harald Garcke\thanks{Fakult\"{a}t f\"{u}r Mathematik, Universit\"{a}t Regensburg, 93040 Regensburg, Germany. E-mail: \mbox{harald.garcke@ur.de}}\and Robert N\"{u}rnberg\thanks{Dipartimento di Mathematica, Universit\`{a} di Trento, 38123 Trento, Italy. E-mail: robert.nurnberg@unitn.it}}
\date{}

\begin{document}

\maketitle

\begin{abstract}
We consider a sharp interface formulation for an anisotropic multi-phase 
Mullins--Sekerka problem with kinetic undercooling. 
The flow is characterized by a cluster of surfaces 
evolving such that the total surface energy plus a weighted sum of the volumes 
of the enclosed phases  decreases in time. 
Upon deriving a suitable variational formulation, we introduce a fully
discrete unfitted finite element method. In this approach,
the approximations of the moving interfaces are independent of the
triangulations  used for the equations in the bulk.
Our method can be shown to be unconditionally stable.
Several numerical examples demonstrate the capabilities of the 
\fix{proposed} method. In particular, it is demonstrated that the evolution of multiple ice crystals with junctions can be modeled 
using the proposed approach.
\end{abstract}



\section{Introduction}
Crystal growth driven by diffusion and  anisotropic surface energy leads to fascinating pattern formation phenomena in nature.
In addition, the understanding of such crystallization processes is fundamental for many applications in engineering and in particular in the foundry industry and we refer
to \cite{Herlach2008} for more details on phase transformations 
in multi-phase systems.
In this paper, we consider a sharp interface model for an anisotropic 
multi-phase Mullins--Sekerka problem,
which describes the evolution of a system that can exhibit  more than two phases and is governed by a quasi-static diffusion equation.
More precisely, we study the evolution of a 
curve network/surface cluster in a bounded Lipschitz domain $\Omega$ in 
$\mathbb{R}^d$, $d = 2,3$.
The cluster $\Gamma$ is made up of several interfaces $\Gamma_i$\fix{, $i\in\naturalset{I_S}$, $I_S\geq1$}, which can
meet at triple junctions $\mathcal{T}_k$\fix{, $k\in\naturalset{I_T}$, $I_T\geq 0$}, and which separate the 
bounded domain $\Omega\subset\mathbb{R}^d$ into regions $\mathcal{R}_\ell$
belonging to different phases, $\ell \in \naturalset{I_R}$,
$I_R\geq2$.
\fix{Here for $K\in\mathbb N$, we let $\naturalset{K}$ := $\{1,\ldots,K\}$, with the convention that $\naturalset{0} = \emptyset$.}
The model  describes the evolution of the surface cluster such that the
anisotropic interfacial energy is decreased over time. The flow is principally driven
by fluxes across interfaces that are derived from chemical potentials
$\bv w = (w_1,\ldots,w_{I_R})^\top$, which satisfy diffusion equations in the bulk.
We remark that in some applications the vector $\bv w$ also describes concentrations of chemical substances.
For example, in ice crystal growth $\bv w$ is related to vapor number densities. 
In addition, the evolving cluster satisfies an anisotropic Gibbs--Thomson law
featuring effects due to kinetic undercooling. Overall, the model seeks a vector of chemical potentials
$\bv w : \Omega \to \bR^{I_R}$ and an evolving surface cluster $(\curve{}{}(t))_{t\in[0,T]}$ such that
\begin{subequations}\label{eq:AMMS}
\begin{align}
    &\Delta\bv w = \zerovec
        && \mbox{in}\quad\Omega\backslash\curve{}{}(t),\ t\in(0,T], \label{eq:AMMS:a}\\
    &\jumpPhase{\bv w}{\curve{}{i}}{} = \bv 0
        && \mbox{on}\quad\curve{}{i}(t),\ t\in(0,T],\ i\in\naturalset{I_S}, \label{eq:AMMS:a2}\\
    &\bv w\cdot\jumpPhase{\phasecharacter}{\curve{}{i}}{}
        = \curvature{\anisotropy{},i} - \frac{V_i\kinetic{i}}{\mobility{i}(\normall{}{i})}
        && \mbox{on}\quad\curve{}{i}(t),\ t\in(0,T],\ i\in\naturalset{I_S}, \label{eq:AMMS:b}\\
    &\jumpPhase{\nabla\bv w}{\curve{}{i}}{}\,\normall{}{i}
        = -V_i\jumpPhase{\phasecharacter}{\curve{}{i}}{}
        && \mbox{on}\quad\curve{}{i}(t),\ t\in(0,T],\ i\in\naturalset{I_S}, \label{eq:AMMS:c}\\
    &\sum_{\ell=1}^3\left\{\anisotropy{\tpindex{k}{\ell}}(\normall{}{\tpindex{k}{\ell}})\conormal{\tpindex{k}{\ell}} - \left(\anisotropy{\tpindex{k}{\ell}}'(\normall{}{\tpindex{k}{\ell}})\cdot\conormal{\tpindex{k}{\ell}}\right)\normall{}{\tpindex{k}{\ell}}\right\} = \vec 0
        && \mbox{on}\quad\mT_k(t),\ t\in(0,T],\ k\in\naturalset{I_T}, \label{eq:AMMS:f}\\
    &\nabla_{\normalcontainer}\bv w = O
        && \mbox{on}\quad\pOmega_N, \label{eq:AMMS:d}\\
    &\chemicalVec{}{} = {\chemicalVec{}{}}_D
        && \mbox{on}\quad\pOmega_D, \label{eq:AMMS:e}\\
    &\curve{}{}(0) = \Gamma_0, \label{eq:AMMS:g}
\end{align}
\end{subequations}
where $T > 0$ is a final time. Here $\vec\nu_i$, $V_i$ and 
$\curvature{\anisotropy{},i}$ denote a unit normal, the associated normal velocity and the
anisotropic mean curvature of the surfaces making up the cluster
$\curve{}{}(t) = \bigcup_{i=1}^{I_S}\curve{}{i}(t)$.
For a quantity $q$, we define the jump of $q$ across $\Gamma_{i}(t)$ 
in the direction of the unit normal $\vec\nu_i$ by
$\jumpPhase{q}{\curve{}{i}}{}:= \lim_{\varepsilon\downarrow 0}\{q(\cdot + \varepsilon\normall{}{i}) - q(\cdot - \varepsilon\normall{}{i})\}$.
Moreover, $\rho_i \geq0$ are kinetic coefficients and
$\mobility{i}:\mathbb{S}^{d-1}\to\mathbb{R}_{>0}$ describe 
orientation-dependent mobility functions which  are assumed to be smooth, even and
positive functions defined on the unit sphere.
In addition, $\phasecharacter = (\chi_1,\ldots,\chi_{I_R})^\top$ denotes the vector  of the characteristic functions 
$\chi_\ell $
 of the regions
$\mathcal{R}_\ell$.
Hence \eqref{eq:AMMS:b} describes an anisotropic Gibbs--Thomson law with 
kinetic undercooling. 
The equation \eqref{eq:AMMS:c} describes an interfacial mass balance, see \cite{Davis01,GarckeNSt04}.  On  triple junctions the force balance condition 
\eqref{eq:AMMS:f} must hold, which involves the normals and
conormals of the three surfaces meeting at a junction $\mathcal{T}_k$. 
For a geometric and physical interpretation of this force balance condition we refer to \cite{CahnHoffman1972, GarckeNS98} and
\cite[p.~199]{clust3d}. 
The precise definitions and formulations will be stated in 
\Section{sec:MathematicalProperties}.
To close the system, we impose the boundary conditions
\eqref{eq:AMMS:d}, \eqref{eq:AMMS:e} and the initial condition 
\eqref{eq:AMMS:g}. For the former two, we have split the 
boundary $\pOmega$, with outer normal $\vec\nu_\Omega$ into 
the relatively open subsets $\pOmega_D$ and $\pOmega_N$, such that
$\pOmega = \closure{\pOmega_D} \cup \closure{\pOmega_N}$ and
$\pOmega_D \cap \pOmega_N = \emptyset$, 
for the Dirichlet and Neumann boundary conditions that
are imposed on the chemical potential, respectively. By $\nabla_{\normalcontainer}$ we denote the spatial derivative in the direction of the normal $\normalcontainer$. With the help of the Dirichlet data $ {\chemicalVec{}{}}_D$, we can model undercooling and supersaturation at an outer boundary, see, e.g. \cite{jcg} for the two-phase case. Throughout this work, we assume for simplicity that
$ {\chemicalVec{}{}}_D$ is constant. The model \eqref{eq:AMMS} is an anisotropic version of the model studied in \cite{BronsardGarckeStoth1998,EtoGarckeNurnberg2024} which also takes kinetic undercooling into account.  For
the modifications needed to include anisotropy, we refer to \cite{GarckeNS98} and for a discussion  how to model   kinetic undercooling, we refer to \cite{Davis01}.

Other approaches to multi-phase crystal growth use multi-phase field models and we refer to
\cite{GarckeNS98,GarckeNSt04, NestlerGSt05,DANILOV2005e177, Steinbach2009} for a detailed discussion of such models. In \cite{GarckeNS98,GarckeNSt04} formally matched asymptotic expansions are used to relate sharp interface models of the form
\eqref{eq:AMMS} to multi-phase field models.
Let us briefly review the literature on numerical methods for
multi-phase Mullins--Sekerka problems.
To the best of our knowledge, the only front tracking methods for
Mullins--Sekerka flows involving several phases and interfaces meeting at
triple junctions can be found in \cite{EtoGarckeNurnberg2024} 
and \cite{EtoGarckeNurnberg2025}, both by the present authors.
The model approximated in \cite{EtoGarckeNurnberg2024} is based on 
\cite{BronsardGarckeStoth1998} and is closely related to \eqref{eq:AMMS},
whereas the model considered in \cite{EtoGarckeNurnberg2025} 
is based on the formulation proposed 
in
\cite{GarckeSturzenhecker1998}. 
Via formally matched asymptotics, the problem \eqref{eq:AMMS}, as well as the
models studied in \cite{EtoGarckeNurnberg2024,EtoGarckeNurnberg2025}, can be
recovered as sharp interface limits of certain multi-component viscous Cahn--Hilliard 
equations. We refer to \cite{BronsardGarckeStoth1998, GarckeNS98}  for details on how to use the method of formally matched asymptotic expansions for Cahn--Hilliard systems. In particular, in \cite{GarckeNS98} the anisotropic force balance condition \eqref{eq:AMMS:f} has been derived.  For numerical methods for the approximation of Cahn--Hilliard systems 
 we refer to \cite{Eyre1993,BloweyCopettiElliott1996,
BarrettBloweyGarcke2001,Nurnberg09,LiChoiKim2016}.
Later, we will  also discuss some application to snow crystal growth and refer to \cite{Libbrecht05} for an introduction to the physics of snow crystal growth, to 
 \cite{jcg} for sharp interface computations for snow crystal growth and to \cite{DemangeZPB17} for phase field computations of growing snow crystals. 

The rest of the paper is organized as follows.
In \Section{sec:MathematicalProperties}, we give the precise mathematical 
definitions needed to formulate \eqref{eq:AMMS} and prove a dissipation property of the anisotropic surface energy for solutions of \eqref{eq:AMMS}.
In \Section{sec:WeakFormulation}, we introduce a weak formulation of the system and show the energy bound using this formulation.
In \Section{sec:FEM_approx}, we propose a fully discrete finite element scheme for the weak formulation introduced in \Section{sec:WeakFormulation}.
We also prove the existence and uniqueness of solutions, as well as
an unconditional stability bound that mimics the dissipation property 
presented in \Section{sec:MathematicalProperties} on the discrete level.
\Section{sec:MatrixForm} is devoted to discussing solution methods for
the linear systems that arise at each time level.
Finally, we show several numerical simulations in 
\Section{sec:NumericalResults}.

\section{Mathematical formulation}\label{sec:MathematicalProperties}
In this section, we give the precise definitions needed to formulate the moving boundary problem 
\eqref{eq:AMMS}. In addition, we present an important property
of strong solutions to \eqref{eq:AMMS}.

We begin with the description of the surface cluster $\Gamma(t)$, 
following the representation of evolving surface clusters from 
\cite{fluidfbptj} (see also \cite{EtoGarckeNurnberg2024}).
The domain $\Omega$ is split into subdomains $\subdomains{\ell}\ (\ell\in\naturalset{I_R})$, with $I_R\geq 2$,
by the surface cluster $\Gamma(t)= \bigcup_{i=1}^{I_S}\curve{}{i}(t)$,
i.e., $\Omega = \Gamma(t)\cup\bigcup_{\ell=1}^{I_R}\subdomains{\ell}$ with
$\subdomains{\ell_1}\cap\subdomains{\ell_2}=\emptyset$ if $\ell_1\neq \ell_2$.
Each $\Gamma_i(t)\ (i\in\naturalset{I_S})$ is either 
a closed hypersurface without boundary or a hypersurface with boundary in $\Omega$. 
For $k\in\naturalset{I_T}$, $I_T\geq0$, let
$\mT_k(t)$ denote the triple junction at which the three hypersurfaces
$\Gamma_{\tpindex{k}{1}}(t)$, $\Gamma_{\tpindex{k}{2}}(t)$ and $\Gamma_{\tpindex{k}{3}}(t)$
with $1\leq \tpindex{k}{1} < \tpindex{k}{2} < \tpindex{k}{3} \leq I_S$ meet, that is $\mT_k(t) \subset \bigcap_{j=1}^3\partial\Gamma_{\tpindex{k}{j}}(t)$.
In addition, $\conormal{i}$ denotes the conormal, i.e.\ the intrinsic outer 
unit normal to $\partial\Gamma_i$, the boundary of $\Gamma_i$, 
that lies within the tangent plane of $\Gamma_i$. 
See \Figure{fig:3p} for an example of a curve network in the three-phase case.
\begin{figure}[H]
    \centering
    \includegraphics[keepaspectratio, scale=0.25]{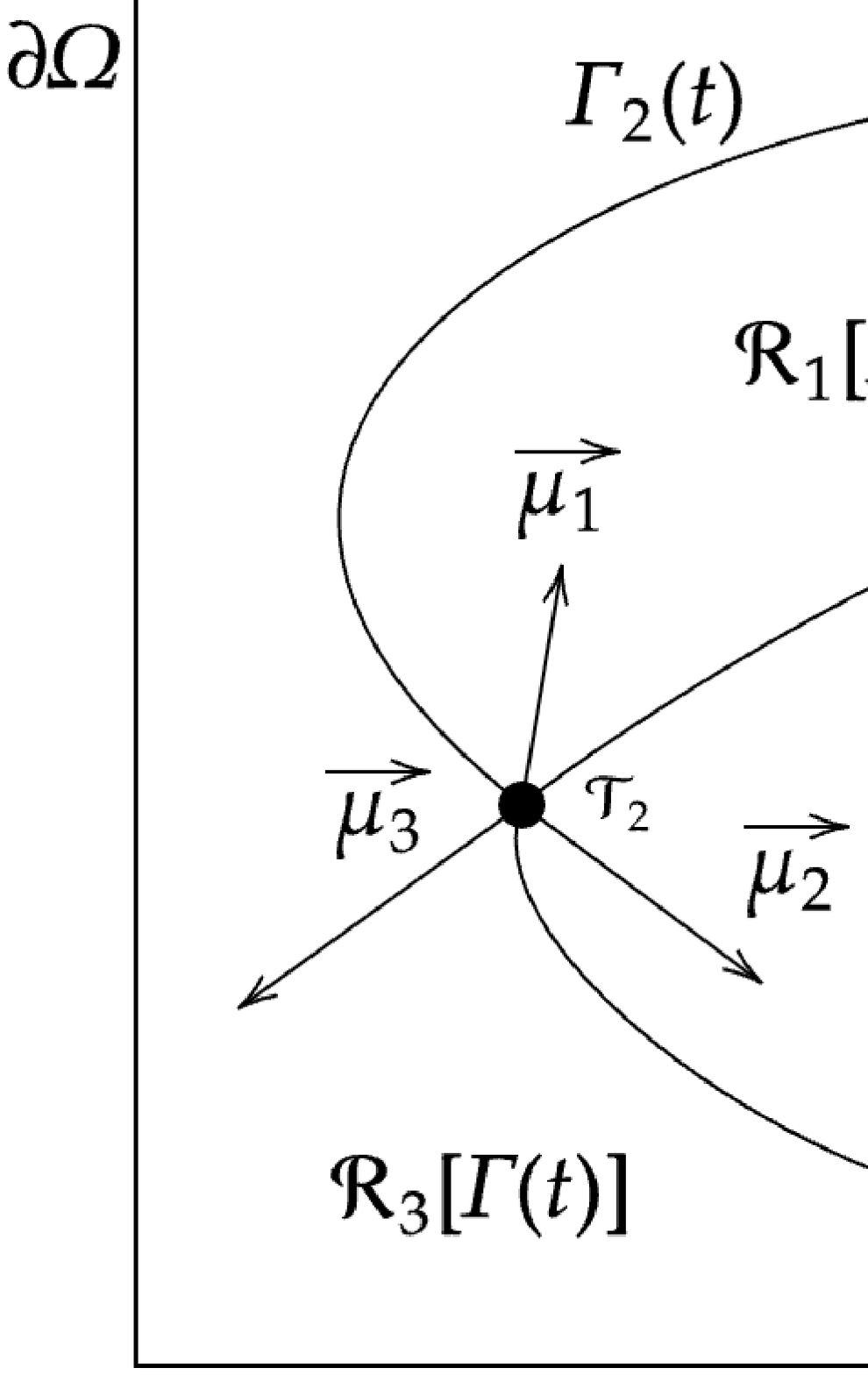}
    \caption{A surface cluster in 2D made up of three open curves and 
two triple junctions \\$(d = 2,\, I_S = 3,\, I_R = 3,\, I_T = 2)$.}\label{fig:3p}
\end{figure}

The surface cluster is endowed with the anisotropic surface energy
\begin{equation} \label{eq:en}
    |\Gamma(t)|_{\anisotropy{}} := \sum_{i=1}^{I_S}|\Gamma_i(t)|_{\anisotropy{i}}\quad\mbox{with}\quad|\Gamma_i(t)|_{\anisotropy{i}} := \int_{\Gamma_i(t)}\anisotropy{i}(\vec{\nu}_i)\dH{d-1},
\end{equation}
where $\mH^{d-1}$ is the $(d-1)$-dimensional Hausdorff measure.
Moreover, $\anisotropy{i}:\mathbb{R}^d\to \mathbb{R}_{\geq0}$ $(i\in\naturalset{I_S})$ denotes the anisotropy function of the surface $\Gamma_i(t)$,
which is assumed to be convex and  absolutely  homogeneous of degree one, i.e.,\ $\anisotropy{i}(\lambda\vec p) = |\lambda|\anisotropy{i}(\vec p)$ for all $\lambda \in\mathbb{R}$ and $\vec p\in\mathbb{R}^d$. 
We assume, that $\anisotropy{i} \in C^1(\mathbb{R}^d\setminus \{0\} , \mathbb{R}_{>0} )$, $i\in\naturalset{I_S}$. Clearly, in the isotropic case
$\anisotropy{i}(\vec p) = |\vec p|$, $i\in\naturalset{I_S}$,
the energy \eqref{eq:en} reduces to the total surface area of the cluster
$\Gamma(t)$.

The anisotropic mean curvature $\curvature{\anisotropy{},i}\,(i\in\naturalset{I_S})$ of $\Gamma_i(t)$ in the direction of $\normall{}{i}$
is then defined through the Cahn--Hoffman vector field (see \cite{CahnHoffman1972}) as follows:
\begin{equation} \label{eq:curvature}
    \curvature{\anisotropy{},i} := - \sgrad\cdot\left(\anisotropy{i}'(\normall{}{i})\right)\qquad\mbox{for}\quad i\in\naturalset{I_S},
\end{equation}
where $\sgrad\cdot$ denotes the surface divergence on $\curve{}{i}$,
and $\gamma_i'$ denotes the spatial gradient of the function 
$\gamma_i:\mathbb{R}^d\to\mathbb{R}$. As the $\gamma_i$ and the $\beta_i$ are even, we observe that the system \eqref{eq:AMMS} does not depend on the choice of the \fix{orientation for the} normals \fix{$\Vec{\nu}_i$} on $\Gamma_i$.

Similarly to \cite{BronsardGarckeStoth1998,EtoGarckeNurnberg2024},
from now on we assume that the chemical potential 
$\bv w$ takes values on the hyperplane 
$T\Sigma := \{\bv u\in\bR^{I_R} \mid \bv u \cdot \bv 1 
= \sum_{\ell=1}^{I_R}u_\ell = 0\}$,
where $\bv 1 = (1,\ldots,1)^\top$. This can be motivated as follows.
Taking the inner products of \eqref{eq:AMMS:a}, \eqref{eq:AMMS:a2}, 
\eqref{eq:AMMS:c}, 
\eqref{eq:AMMS:d} and \eqref{eq:AMMS:e} with $\bv 1$, and noting
that $\bv \chi \cdot \bv 1 \equiv 0$, we obtain, in a suitable weak sense, that
\begin{equation*}
\Delta (\bv w \cdot \bv 1) = 0 \quad \text{in } \Omega,\qquad
\nabla_{\normalcontainer}(\bv w \cdot \bv 1) = 0 \quad \text{on }\pOmega_N, 
\qquad
\bv w \cdot \bv 1 = \bv w_D \cdot \bv 1 \quad \text{on }\pOmega_D.
\end{equation*}
Hence $\bv w \cdot \bv 1$ is equal to a constant in $\Omega$, which equals 
$\bv w_D \cdot \bv 1$ if $\pOmega_D\not=\emptyset$. 
From now on we make the assumption that
$\bv w_D \cdot \bv 1 = 0$, i.e.\ that
$\chemicalVec{}{D} = (w_{D,1},\ldots,w_{D,I_R})^\top \in \zerosumspace$,
which implies that $\bv w(\cdot, t) \in \zerosumspace$ for $t\in[0,T]$.
Furthermore, without loss of generality, we set $\chemicalVec{}{D} = \bv 0$ if
$\pOmega_D = \emptyset$.

We note that in the case of the curve network displayed in Figure~\ref{fig:3p},
and for the isotropic surface energy densities
$\gamma_i(\vec{p}) = \tension{i}|\vec{p}|$ with $\tension{i} > 0$, 
together with $\kinetic{i} = 0$, for $i=1,2,3$, and $\pOmega_N = \pOmega$,
the system \eqref{eq:AMMS} collapses to the model (1.2) in
\cite{EtoGarckeNurnberg2024} which was introduced in
\cite{BronsardGarckeStoth1998}. In that paper, the present authors 
introduced a numerical method based on the
parametric finite element method (PFEM) for the numerical approximation
of \eqref{eq:AMMS} in 2D in this simplified setting. 
Hence the present contribution can be viewed as a generalization of
that numerical method, and its analysis, to general surface clusters in 3D
with anisotropic surface energies and the presence of kinetic undercooling as
well as undercooling imposed at the boundary.

\begin{prop}\label{prop:dissp}
    Assume that $(\bv w(\cdot, t),\,\curve{}{}(t))_{t\in[0,T]}$
    is a classical solution to $\eqref{eq:AMMS}$.
    Then, we have that
    \begin{equation}\label{eq:dissipation}
        \ddt \left\{|\Gamma(t)|_{\anisotropy{}} - \sum_{\ell = 1}^{I_R}w_{D,\ell}\volume{\subdomains{\ell}}\right\} + \|\nabla\chemicalVec{}{}\|^2_{L^2(\Omega)} + \left\lVert \sqrt{\frac{\kinetic{}}{\mobility{}(\normall{}{})}}V\right\rVert_{L^2(\curve{}{}(t))}^2 = 0,
    \end{equation}
recall \eqref{eq:en}. 
    In addition, it holds that
\begin{equation} \label{eq:dtvol}
\ddt\volume{\subdomains{\ell}}= \int_{\pOmega_D}\nabla w_\ell\cdot\normall{}{\Omega}\dH{d-1} \qquad\mbox{for}\quad \ell\in\naturalset{I_R}.
\end{equation}
\end{prop}
\begin{proof}
First, we recall from \cite[Lemma 3.1 and 
 Eq.(3.4)]{clust3d} that
the balance law \eqref{eq:AMMS:f} on the triple junctions implies
\begin{equation}\label{eq:ts}
\ddt \sum_{i=1}^{I_S} |\curve{}{i}(t)|_{\anisotropy{i}} = -\sum_{i=1}^{I_S} \int_{\curve{}{i}(t)} \curvature{\anisotropy{},i}V_i\dH{d-1},
\end{equation}
see also \cite[p.~298]{BarrettGarckeNurnberg2008IMA}.
Using \eqref{eq:AMMS:b} and \eqref{eq:AMMS:c}, we then obtain that
\begin{align}\label{eq:dissp2}
\ddt \sum_{i=1}^{I_S} |\curve{}{i}(t)|_{\anisotropy{i}} &= -\sum_{i=1}^{I_S} \int_{\curve{}{i}(t)}\left( \chemicalVec{}{}\cdot\jumpPhase{\phasecharacter}{\curve{}{i}}{} + \frac{\kinetic{i}V_i}{\mobility{i}(\normall{}{i})}\right)V_i\dH{d-1}\nonumber\\
&= - \sum_{i=1}^{I_S}\int_{\curve{}{i}(t)} \chemicalVec{}{}\cdot (V_i\jumpPhase{\phasecharacter}{\curve{}{i}}{})\dH{d-1} - \sum_{i=1}^{I_S}\int_{\curve{}{i}(t)}\frac{\kinetic{i}}{\mobility{i}(\normall{}{i})}V_i^2\dH{d-1}\nonumber\\
&= \sum_{i=1}^{I_S}\int_{\curve{}{i}(t)} \chemicalVec{}{}\cdot(\jumpPhase{\nabla\chemicalVec{}{}}{\curve{}{i}}{}\normall{}{i})\dH{d-1} - \sum_{i=1}^{I_S}\int_{\curve{}{i}(t)}\frac{\kinetic{i}}{\mobility{i}(\normall{}{i})}V_i^2\dH{d-1}.
\end{align}
Meanwhile, for each $\ell\in \naturalset{I_R}$,
it follows from integration by parts, \eqref{eq:AMMS:a}, \eqref{eq:AMMS:d},
and \eqref{eq:AMMS:e} that
\begin{align}\label{eq:dissp3}
\int_{\Omega} |\nabla w_\ell|^2\,\dL{d} &= 
\int_{\Omega\setminus\Gamma(t)} |\nabla w_\ell|^2\,\dL{d}
\nonumber \\
&    =  \int_{\pOmega_D} w_{D,\ell}\nabla w_\ell\cdot\normall{}{\Omega}\dH{d-1} - \sum_{i=1}^{I_S}\int_{\curve{}{i}(t)} w_\ell\, \jumpPhase{\nabla w_\ell}{\curve{}{i}}{}\cdot\normall{}{i}\dH{d-1}.
\end{align}
Summing \eqref{eq:dissp3} for $\ell\in\naturalset{I_R}$, we have
\begin{equation}\label{eq:dissp4}
\|\nabla\chemicalVec{}{}\|_{L^2(\Omega)}^2 = \int_{\pOmega_D} \chemicalVec{}{D}\cdot(\nabla \chemicalVec{}{}\normall{}{\Omega})\dH{d-1}  - \sum_{i=1}^{I_S}\int_{\curve{}{i}(t)} \chemicalVec{}{}\cdot (\jumpPhase{\nabla \chemicalVec{}{}}{\curve{}{i}}{}\normall{}{i})\dH{d-1}.
\end{equation}
Similarly to \eqref{eq:dissp3}, we obtain from \eqref{eq:AMMS:a},
integration by parts and \eqref{eq:AMMS:c}, \eqref{eq:AMMS:e} that
\begin{align}\label{eq:dissp9}
0 &= \int_{\Omega\setminus\Gamma(t)}\Delta w_\ell\,\dL{d} 
= \int_{\pOmega_D}\nabla w_\ell\cdot\normall{}{\Omega}\dH{d-1}
- \sum_{i=1}^{I_S} \int_{\curve{}{i}(t)} \jumpPhase{\nabla w_\ell}{\curve{}{i}}{}\cdot\normall{}{i}\dH{d-1} \nonumber\\
&= \int_{\pOmega_D}\nabla w_\ell\cdot\normall{}{\Omega}\dH{d-1} 
+ \sum_{i=1}^{I_S} \int_{\curve{}{i}(t)} [\chi_\ell]_{\Gamma_i} V_i \dH{d-1}
\nonumber\\
&= \int_{\pOmega_D}\nabla w_\ell\cdot\normall{}{\Omega}\dH{d-1} - \ddt\volume{\subdomains{\ell}},
\end{align}
where we recall that 
$\chi_\ell$
 denotes the characteristic function
of the region $\subdomains\ell$, so that $[\chi_\ell]_{\Gamma_i} = 1$
if $\Gamma_i(t) \subset \partial\subdomains\ell$ and $\vec\nu_i$ points into
the region.
This proves \eqref{eq:dtvol}. Moreover, 
multiplying \eqref{eq:dissp9} with $w_{D,\ell}$ and summing over 
$\ell\in\naturalset{I_R}$ gives
\begin{equation}\label{eq:dissp5}
\ddt\sum_{\ell=1}^{I_R} w_{D,\ell}\volume{\subdomains{\ell}} = \int_{\pOmega_D}\chemicalVec{}{D}\cdot(\nabla \chemicalVec{}{}\normall{}{\Omega})\dH{d-1}.
\end{equation}
Combining \eqref{eq:dissp2}, \eqref{eq:dissp4}, and \eqref{eq:dissp5} yields the desired result \eqref{eq:dissipation}. 
\end{proof}

\begin{rem}
We note that \eqref{eq:dissipation} gives a dissipation result for the quantity
\begin{equation} \label{eq:quantity}
|\Gamma(t)|_{\anisotropy{}} 
- \sum_{\ell = 1}^{I_R}w_{D,\ell}\volume{\subdomains{\ell}},
\end{equation}
\fix{which can be viewed as a free energy for the system \eqref{eq:AMMS}.}
In the case $\pOmega_D = \emptyset$ this reduces to the surface energy \eqref{eq:en}. Moreover, if $\pOmega_D = \emptyset$, then \eqref{eq:dtvol} 
ensures that the volume of each region is preserved.
\fix{In the case $\pOmega_D\neq \emptyset$, on the other hand, \eqref{eq:dtvol} describes the mass in-flux from the boundary $\pOmega$.}
Compare also with Propositions~2.1 and 2.2 in \cite{EtoGarckeNurnberg2024}.
\end{rem}

\section{Weak formulation}\label{sec:WeakFormulation}
Let us derive a weak formulation for $\eqref{eq:AMMS}$.
First, we introduce the following function spaces for the bulk trial and test
functions:
\begin{align*}
    &S_0(\Omega) := \left\{u\in H^1(\Omega) ~|~ u = 0\quad\mbox{on}\quad \pOmega_D\right\},\qquad \bv S_0(\Omega) := [S_0(\Omega)]^{I_R},\\
    &\bv S_D(\Omega) := \left\{\bv u\in [H^1(\Omega)]^{I_R} ~|~ \bv u = \chemicalVec{}{D}\quad\mbox{on}\quad \pOmega_D\right\},\\
    &\bv S_\Sigma(\Omega) := \left\{\bv u\in [H^1(\Omega)]^{I_R} ~|~ \bv u(x) \in \zerosumspace\quad\forall x\in\Omega\right\}.
\end{align*}
We suppose that $(\chemicalVec{}{}(\cdot, t), \curve{}{}(t))_{t\in[0,T]}$
is a solution to $\eqref{eq:AMMS}$ and
$\chemicalVec{}{}(\cdot, t)$ belongs to $\mathbf{S}_D(\Omega)\cap \mathbf{S}_\Sigma(\Omega)$.
Then, similarly to \eqref{eq:dissp9}, we obtain on  
testing $\eqref{eq:AMMS:a}$ with $\bv \varphi\in\bv S_0(\Omega)$ 
and performing integration by parts, for
$\ell\in\naturalset{I_R}$ that
\begin{align*}
    0 &=
\int_{\Omega\backslash\Gamma(t)}\Delta w_\ell\varphi_\ell\dL{d} 
     = -\sum_{i=1}^{I_S} \int_{\curve{}{i}(t)}\jumpPhase{\nabla w_\ell}{\curve{}{i}}{}\cdot\normall{}{i}\varphi_\ell\dH{d-1} - \int_\Omega \nabla w_\ell\cdot\nabla\varphi_\ell\dL{d}\\&
     = \sum_{i=1}^{I_S} \int_{\curve{}{i}(t)} \jumpPhase{\chi_\ell}{\curve{}{i}}{} V_i\varphi_\ell\dH{d-1} - \int_\Omega \nabla w_\ell\cdot\nabla\varphi_\ell\dL{d},
\end{align*}
where we have used that $\varphi_\ell\in S_0(\Omega)$ and the conditions \eqref{eq:AMMS:d} and \eqref{eq:AMMS:c}.
Summing over $\ell=1,\ldots,I_R$ gives
\begin{equation}\label{eq:WeakForm1}
    \int_\Omega \nabla \bv w : \nabla \bv \varphi\,\dL{d} -
\sum_{\ell=1}^{I_R}
\sum_{i=1}^{I_S} \int_{\curve{}{i}(t)} \jumpPhase{\chi_\ell}{\curve{}{i}}{} V_i\varphi_\ell\dH{d-1} = 0
 \qquad
    \forall \bv \varphi \in \bv S_0(\Omega).
\end{equation}
Multiplying \eqref{eq:AMMS:b} with a test function
$\xi \in L^2(\Gamma(t))$ and integrating over $\Gamma(t)$ yields
\begin{equation}\label{eq:WeakForm2}
\sum_{i=1}^{I_S} \int_{\curve{}{i}(t)} \curvature{\anisotropy{},i}\xi_i\dH{d-1} 
-
\sum_{\ell=1}^{I_R}
\sum_{i=1}^{I_S} \int_{\curve{}{i}(t)} \jumpPhase{\chi_\ell}{\curve{}{i}}{} w_\ell\xi_i \dH{d-1}
- \int_{\curve{}{i}}\frac{\kinetic{i}}{\mobility{i}(\normall{}{i})}V_i\xi_i\dH{d-1} = 0.
\end{equation}

For general anisotropies it is not straightforward to come up with a weak
formulation of the anisotropic mean curvature vector 
that is suitable for a parametric finite element approximation
based on linear elements. Hence from now on we follow
\cite{BarrettGarckeNurnberg2008IMA,BarrettGarckeNurnberg2008NM,clust3d,ejam3d}
and choose a special class of anisotropies $\anisotropy{i}$ that are of the 
form
\begin{equation*}
    \anisotropy{i}(\vec p) = \sum_{\ell = 1}^{L_i}\anisotropy{i}^{(\ell)}(\vec p),\qquad\mbox{with}\quad \anisotropy{i}^{(\ell)}(\vec p) := \sqrt{\vec{p}\cdot G_i^{(\ell)}\vec{p}},
\end{equation*}
where $G_i^{(\ell)}\in\mathbb{R}^{d\times d}$ is a positive definite matrix for each $\ell\in\naturalset{L_i}$, and $L_i\geq 1$.
\fix{We remark that most energies of relevance, including crystalline 
surface energies, can be approximated either by 
the above class of energies, or by a nonlinear generalization of it, where the
sum is replaced by an $\ell^r$ norm, for $r \in [1, \infty)$.
We refer to Chapter~6 in the review article \cite{BarrettGarckeRobertBook2020}
for further details.}

We now present a weak formulation of the anisotropic mean curvature 
vector, which goes back to \cite{BarrettGarckeNurnberg2008NM}, see
also \cite{clust3d} for the extension to surface clusters. We also refer to \cite{LiBao21} for an alternative weak formulation.
For a symmetric positive matrix $G$, we set $\tG  = [{\rm det}\,G]^{\frac{1}{d-1}}\,[G]^{-1}$ and define the $\tG$-inner product
\[
\bigl(\vec\eta,~\vec\zeta\bigr)_{\tG}= \vec\eta\cdot\tG\vec\zeta,\qquad\forall \vec\eta,~\vec\zeta\in\bR^d.
\]
For a smooth scalar field $g$ over $\Gamma_i(t)$, we define the anisotropic surface gradient
\begin{align*}
\nabla_s^{\tG}g = \sum_{j=1}^{d-1}\partial_{\vec t_{j}}g\,\vec t_{j}=\sum_{j=1}^{d-1}(\nabla_sg\cdot\vec t_{j})\,\vec t_{j},
\end{align*}
where $\partial_{\vec t_{j}}g = \nabla_sg\cdot\vec t_{j}$ is the directional derivative, $\nabla_s$ is the usual surface gradient operator, and $\{\vec t_{j}\}_{j=1}^{d-1}$ forms an orthonormal basis with respect to the $\tG$-inner product for the tangent plane of $\Gamma_i(t)$ at the point of interest, i.e., 
\begin{equation*}
\vec t_{j}\cdot\vec\nu_i = 0,\qquad \left(\vec t_{j}, ~\vec t_{k}\right)_{\tG} = \delta_{jk}, \quad 1\leq j,k\leq d-1.
\end{equation*}
Moreover, the anisotropic surface divergence and gradient of a smooth vector field $\vec g$ are given by
\begin{align*}
\nabla_s^{\tG}\cdot\vec g = \sum_{j=1}^{d-1} (\partial_{\vec t_{j}}\vec g)\cdot(\tG\vec t_{j}),\qquad 
\nabla_s^{\tG}\vec g =  \sum_{j=1}^{d-1}(\partial_{\vec t_{j}}\vec g)\otimes(\tG\vec t_{j}),
\end{align*}
where $\otimes$ is the standard tensor product for two vectors in $\bR^d$.
We also define the inner product
\begin{equation} \label{eq:Gtildeprod}
	(\nabla_s^{\tG}\,\vec{u}, \nabla_s^{\tG}\,\vec{v})_{\tG }:=
	\sum_{i=1}^{d-1} (\partial_{\vec{t}_i}\,\vec{u} ,
	\partial_{\vec{t}_i}\,\vec{v})_{\tG}\,
\end{equation}
for smooth $\vec{u},\vec{v}:\Gamma \to {\mathbb{R}}^d$.
Then, we define for smooth functions $\vec\eta$ and $\vec\zeta$ 
the inner product
\begin{equation} \label{eq:aniip}
\big<\nabla_s^{\tG}\vec \eta, ~\nabla_s^{\tG}\vec\zeta\big>_{\gamma, \Gamma(t)}
=\sum_{i=1}^{I_S}\sum_{\ell=1}^{L_i} \int_{\Gamma_i(t)}
\left(\nabla_s^{\tG_i^{(\ell)}}\vec\eta,~\nabla_s^{\tG_i^{(\ell)}}\vec\zeta\right)_{\tG_i^{(\ell)}}\gamma_\ell(\vec \nu_i)\dH{d-1}.
\end{equation}
In \cite{clust3d} it was shown that for a surface cluster satisfying
\eqref{eq:AMMS:f} the anisotropic mean curvature 
\eqref{eq:curvature} satisfies the identity
\begin{equation*}
   \sum_{i=1}^{I_S} \int_{\curve{}{i}(t)}{\curvature{\anisotropy{},i}\normalnoindex{}{i}}\cdot{\Vec{\eta}_i}{} \dH{d-1}
+ \innerproduct{\sgrad^{\tG}\identity}{\sgrad^{\tG}\Vec{\eta}}{}{\anisotropy{},\curve{}{}(t)} = 0
\end{equation*}
for every $\vec{\eta}\in [H^1(\curve{}{}(t))]^d$ with
$\vec{\eta}_{\tpindex{k}{1}} = \vec{\eta}_{\tpindex{k}{2}} = \vec{\eta}_{\tpindex{k}{3}}$ 
on $\mT_k$, for all $k\in\naturalset{I_T}$.

Let us summarize the weak formulation of the system $\eqref{eq:AMMS}$ as 
follows. Here for convenience we let $\langle\cdot,\cdot\rangle_\Omega$ 
denote the $L^2$--inner product over $\Omega$, and similarly 
for $\Gamma_i(t)$ as well as 
$\langle\cdot,\cdot\rangle_{\Gamma(t)} = \sum_{i=1}^{I_S} 
\langle\cdot,\cdot\rangle_{\Gamma_i(t)}$.
Find $(\bv{w},\Gamma(t))_{t\in[0,T]}$ such that 
$\Gamma(0) = \Gamma_0$ and for all $t \in (0,T]$ it holds that
$\bv w(t) \in \bv S_D(\Omega)\cap \bv S_\Sigma(\Omega)$ such that the following
identities hold.
\newline\newline\noindent
\begin{subequations} \label{eq:WFM}
\textbf{[Motion law]} 
For all $\bv\varphi\in \bv S_0(\Omega)\cap \bv S_\Sigma(\Omega)$,
\begin{equation}\label{eq:WFM-ML}
    \innerproduct{\nabla\bv w}{\nabla\bv \varphi}{}{\Omega} 
-
\sum_{\ell=1}^{I_R}
\sum_{i=1}^{I_S} \innerproduct{\jumpPhase{\chi_\ell}{\curve{}{i}}{} V_i}{\varphi_\ell}{}{\curve{}{i}(t)}
=0 .
\end{equation}
\textbf{[Gibbs--Thomson law with kinetic undercooling]} For all $\xi\in L^2(\curve{}{}(t))$,
\begin{equation}\label{eq:WFM-GTL}
    \innerproduct{\curvature{\anisotropy{}}}{\xi}{}{\curve{}{}(t)} 
-
\sum_{\ell=1}^{I_R}
\sum_{i=1}^{I_S} \innerproduct{\jumpPhase{\chi_\ell}{\curve{}{i}}{} w_\ell}{\xi_i}{}{\curve{}{i}(t)}
    -\innerproduct{\frac{\kinetic{}}{\mobility{}(\normall{}{})}V}{\xi}{}{\curve{}{}(t)}= 0.
\end{equation}
\textbf{[Curvature vector]} For all $\vec{\eta}\in [H^1(\curve{}{}(t))]^d$ such that
$\vec{\eta}_{\tpindex{k}{1}} = \vec{\eta}_{\tpindex{k}{2}} = \vec{\eta}_{\tpindex{k}{3}}$ 
on $\mT_k(t)$ for every $k\in\naturalset{I_T}$,
\begin{equation}\label{eq:WFM-CV}
    \innerproduct{\curvature{\anisotropy{}}\normalnoindex{}{}}{\Vec{\eta}}{}{\curve{}{}} + \innerproduct{\sgrad^{\tG}\identity}{\sgrad^{\tG}\Vec{\eta}}{}{\anisotropy{},\curve{}{}(t)} = 0.
\end{equation}
\end{subequations}

\fix{
\begin{rem}
We note that it is enough to choose test functions $\bv \varphi \in \bv S_0(\Omega) \cap \bv S_\Sigma(\Omega)$ in \eqref{eq:WFM-ML}.
In fact, choosing in \eqref{eq:WeakForm1} a test function $\bv \varphi = \zeta \bv 1$ 
yields for the first term
\begin{equation*}
\int_\Omega \sum_{\ell=1}^{I_R} \nabla  w_\ell  \cdot \nabla  \zeta \,\dL{d}= 	\int_\Omega  \nabla \left(\sum_{\ell=1}^{I_R} w_\ell  \right)\cdot \nabla  \zeta \,\dL{d}=0,
\end{equation*}
where the last identity holds because 
$\chemicalVec{}{}(\cdot, t)$ belongs to $ \mathbf{S}_\Sigma(\Omega)$.
In addition, we also have
\begin{equation*}
\sum_{\ell=1}^{I_R}
\sum_{i=1}^{I_S} \int_{\curve{}{i}(t)} \jumpPhase{\chi_\ell}{\curve{}{i}}{} V_i\zeta\dH{d-1}= \sum_{i=1}^{I_S} \int_{\curve{}{i}(t)} \jumpPhase{\sum_{\ell=1}^{I_R}\chi_\ell}{\curve{}{i}}{} V_i\zeta\dH{d-1}=0,
\end{equation*}
where the last identity holds because $\sum_{\ell=1}^{I_R}\chi_\ell=1$.
This shows that it  is enough  to require \eqref{eq:WFM-ML} just for test functions 
$\bv \varphi \in \bv S_0(\Omega) \cap \bv S_\Sigma(\Omega)$.
\end{rem}
}

\section{Finite element approximation}\label{sec:FEM_approx}
Let the time interval $[0,T]$ be split into $M$ sub-intervals $[t_{m-1}, t_m]$ 
for each $m = 1,\ldots,M$, whose length are equal to $\tau_m$. 
Then, given a multiplet of polygonal surfaces $\curve{0}{} = (\curve{0}{1},\cdots,\curve{0}{I_S})$, our aim is to find time discrete multiplets $\curve{1}{},\cdots,\curve{M}{}$ governed by discrete analogues of \eqref{eq:WFM}. 

Let us first define the finite element spaces from where we will seek the approximate solutions.
Let $\triangulation{m}$ be a triangulation of $\closure{\Omega}$. 
\begin{align*}
    &\femspaceChemical{m} := \left\{v\in C(\closure{\Omega})~|~ v\!\mid_o\ \mbox{is affine}\ \forall o\in\triangulation{m}\right\},\qquad
    \femspacebulkvector{} := [\femspaceChemical{m}]^{I_R},\\
    &\femspacebulkvector_0 := \left\{\bv v \in \femspacebulkvector{}~|~ \bv v = \zerovec\ \mbox{on}\ \pOmega_D\right\},\qquad
    \femspacebulkvector_D := \left\{\bv v \in \femspacebulkvector{}~|~ \bv v = \chemicalVec{}{D}\ \mbox{on}\ \pOmega_D\right\},\\
    &\femspacebulkvectorsum{} := \left\{\bv v\in \femspacebulkvector{}~|~ \bv v(x)\in \zerosumspace\quad\forall x\in\closure{\Omega}\right\}.
\end{align*}
Using these notations, the discrete chemical potential $\chemicalDiscVec{m+1}{}$ will be sought in $\femspacebulkvector_D\cap\femspacebulkvectorsum{}$.

For the following presentation of the discrete surface clusters we closely
follow the presentation in \cite{clust3d}, see also 
\cite{BaoGarckeNuernbergZhao2022}
and \cite{EtoGarckeNurnberg2025}.
In order to describe $\Gamma^m$ for $m\geq 0$, and the discrete matching 
conditions that have to hold on the triple junction, we 
let $\refsurface{i}^h\,(i\in \naturalindex{I_S})$ be polyhedral reference surfaces with $\closure{\refsurface{i}^h} = \bigcup_{j=1}^{\numSimplex{}{i}}\closure{\simplex{}{i,j}}$, where $\{\simplex{}{i,j}\}_{j=1}^{J_i}$ is a family of
 mutually disjoint open $(d-1)$-simplices with vertices
$\{\vVertexSynonym{}{i,k}\}_{k=1}^{K_i}$ vertices.

Moreover, we assume that each boundary $\partial\refsurface{i}^h$ is split into $I_P^i$ sub-boundaries $\partial_p\refsurface{i}^h\,(p\in\naturalindex{I_P^i})$,
and each sub-boundary $\partial_p\refsurface{i}^h$ corresponds to the 
parameterization of a triple junction.
In particular, we will let 
$\Gamma^m_i = \Vec{\mathfrak{X}}^m_i(\refsurface{i}^h)$, so that the
triple junction $\mathcal{T}_k$, where the surfaces
$\Gamma_{s^k_1}(t)$, $\Gamma_{s^k_2}(t)$, and $\Gamma_{s^k_3}(t)$ meet, 
is approximated by the images of $\Vec{\mathfrak{X}}^m$ on
$\partial_{p_1^k} \refsurface{s_1^k}^h$, 
$\partial_{p_2^k} \refsurface{s_2^k}^h$,
and $\partial_{p_3^k} \refsurface{s_3^k}^h$.
To this end, we have to ensure that these sub-boundaries perfectly match up on 
the triple junctions, and in particular contain the same number of vertices.
Hence, we assume that for every $k\in\naturalindex{I_T}$, it holds that
\begin{equation}\label{eq:q}
    \numTjVertex{k} := \#Q_{s_1^k,p^k_1} = \#Q_{s_2^k,p^k_2} = 
\#Q_{s_3^k, p^k_3},
\end{equation}
where
$Q_{s,p} := \left\{\vVertexSynonym{}{s,\ell}\right\}_{\ell=1}^{\numSurfVertex{}{s}}\cap \partial_p\refsurface{s}^h$
denotes the set of vertices belonging to the boundary patch
$\partial_p\refsurface{s}^h$.
Then we assume in addition that there exist bijections $\orderedSeq{k}{r}:\,\naturalindex{\numTjVertex{k}}\to Q_{s_r^k,p^k_r}\,(r=1,2,3)$
such that $(\orderedSeq{k}{r}(1),\ldots,\orderedSeq{k}{r}(\numTjVertex{k}))\,(r=1,2,3)$ are ordered sequences of the vertices.

Let
\begin{multline*}
    \underline{V}^h(\refsurface{}^h) := \bigg\{(\polyhedral{}{1},\cdots,\polyhedral{}{I_S})\in\bigotimes_{i=1}^{I_S}[C(\closure{\refsurface{i}^h})]^d~|~\polyhedral{}{i}\!\mid_{\simplex{}{i,j}}\mbox{ is affine}\quad\forall j\in \naturalindex{\numSimplex{}{i}},\,\forall i\in\naturalindex{I_S},\\
    \mbox{and}\qquad\polyhedral{}{\tpindex{k}{1}}(\orderedSeq{k}{1}(z)) = \polyhedral{}{\tpindex{k}{2}}(\orderedSeq{k}{2}(z)) = \polyhedral{}{\tpindex{k}{3}}(\orderedSeq{k}{3}(z))\qquad\forall z\in\naturalindex{\numTjVertex{k}},\,\forall  k\in\naturalindex{I_T}\bigg\}.
\end{multline*}
Then, for each $m\geq 0$ and $\Vec{\mathfrak{X}}^m\in V^h(\refsurface{}^h)$,
we define $\Gamma^m := \Vec{\mathfrak{X}}^m(\refsurface{}^h)$ with
$\curve{m}{i} = \Vec{\mathfrak{X}}^m_i(\refsurface{i}^h)$,
$\simplex{m}{i,j} := \polyhedral{m}{i}(\simplex{}{i,j})$ and
$\vVertexSynonym{m}{i,k} := \polyhedral{m}{i}(\vVertexSynonym{}{i,k})$.
The discrete triple junctions $\mT^m_k\,(k\in\naturalindex{I_T})$
are defined by
$\mT^m_k := \left\{ \polyhedral{m}{\tpindex{k}{1}}(\orderedSeq{k}{1}(z)) ~|~ z\in \naturalindex{\numTjVertex{k}}\right\}$.
On the polyhedral surface $\Gamma^m$, we introduce finite element spaces defined by
\begin{align*}
    V^h(\polygonCurve{m}{i}) &:= \left\{v\in C(\polygonCurve{m}{i})~|~v\!\mid_{\simplex{m}{i,j}}\mbox{ is affine}\quad \forall j\in\naturalindex{\numSimplex{}{i}}\right\},\quad
    \underline{V}^h(\polygonCurve{m}{i}) := [V^h(\polygonCurve{m}{i})]^d,
\quad i\in\naturalindex{I_S}.
\end{align*}
For later use, we also define 
\begin{equation} \label{eq:defV0}
    V^h_0(\polygonCurve{m}{i}) := \left\{v\in V^h(\polygonCurve{m}{i})~|~\, v = 0\quad\mbox{on}\quad \partial\polygonCurve{m}{i}\right\},\quad
    \underline{V}^h_0(\polygonCurve{m}{i}) := [V^h_0(\polygonCurve{m}{i})]^d,
\quad i\in\naturalindex{I_S}
\end{equation}
and let $\{\femspaceBasisCurve{m,0}{i,k}\}_{k=1}^{K^0_i}$ be the standard basis of $V^h_0(\polygonCurve{m}{i})$, so that 
$\femspaceBasisCurve{m,0}{i,k}(\vec q^m_{i,\ell}) = \delta_{k\ell}$, 
$k,\ell \in \naturalindex{K^0_i}$.

Then, the approximate solutions $\vVertex{m+1}{}$ and $\kappa_{\anisotropy{}}^{m+1}$ are respectively sought in the finite element spaces defined by
\begin{align*}
    \femspaceVertexVector{h}{\mathcal{T}}(\Gamma^m) &:= \left\{(\vVertex{}{1},\cdots,\vVertex{}{I_S})\in\bigotimes_{i=1}^{I_S}\underline{V}^h(\polygonCurve{m}{i})~|~\vVertex{}{\tpindex{k}{1}} = \vVertex{}{\tpindex{k}{2}} = \vVertex{}{\tpindex{k}{3}}\quad\mbox{on}\quad\mT_k^m,\ \forall k\in\naturalindex{I_T}\right\},\\
    \femspaceVertex{h}{}(\Gamma^m) &:= \bigotimes_{i=1}^{I_S}V^h(\polygonCurve{m}{i}).
\end{align*}
We now define the normal vector of each simplex $\simplex{m}{i,j}$.
To this end, let
$\left\{\vec q_{i,j,\ell}^{m}\right\}_{\ell=1}^{d}$ be the vertices of 
$\sigma_{i,j}^{m}$, and ordered with the same orientation for all 
$\sigma_{i,j}^{m}$, $j\in\naturalindex{J_i}$. Then we define
\begin{equation*}
    \normal{m}{i,j} := \frac{\Normal{\simplex{m}{i,j}}}{|\Normal{\simplex{m}{i,j}}|}\quad\mbox{with}\quad\Normal{\simplex{m}{i,j}} := \begin{cases}
       (\vVertexSynonym{m}{i,j,2} - \vVertexSynonym{m}{i,j,1})^\perp\quad\mbox{if}\quad d = 2,\\
       (\vVertexSynonym{m}{i,j,2} - \vVertexSynonym{m}{i,j,1})\wedge(\vVertexSynonym{m}{i,j,3} - \vVertexSynonym{m}{i,j,1})\quad\mbox{if}\quad d = 3,
    \end{cases}
\end{equation*}
where $|\cdot| = \mH^{d-1}(\cdot)$;
the symbol $\wedge$ denotes the wedge product, and $v^\perp := (-v_2,v_1)^\top$ for $v = (v_1,v_2)^\top\in\mathbb{R}^2$.
Let $\normal{m}{i}$ be the normal vector on $\polygonCurve{m}{i}$ which equals $\normal{m}{i,j}$ on $\simplex{m}{i,j}$.

Let us define the mass lumped inner product of two piecewise continuous functions $u$ and $v$ on $\polygonCurve{m}{i}$ by
\begin{equation*}
    \inprodMassLamped{u}{v}{h}{\polygonCurve{m}{i}} := \frac{1}{d}\sum_{j=1}^{\numSimplex{}{i}}|\simplex{m}{i,j}|\sum_{k=1}^d\lim_{\substack{\simplex{m}{i,j}}\ni\vec{q}\to\vVertexSynonym{m}{i,j,k}}(uv)(\vec{q}), \qquad i\in\naturalindex{I_S}.
\end{equation*}
Using this, we define the mass lumped inner product on $\Gamma^m$ by
\begin{equation*}
    \inprodMassLamped{u}{v}{h}{\Gamma^m} := \sum_{i=1}^{I_S}\inprodMassLamped{u_i}{v_i}{h}{\polygonCurve{m}{i}}.
\end{equation*}
Meanwhile, we will write the natural $L^2$--inner product as follows:
\begin{equation*}
    \inprodMassLamped{u}{v}{}{\polygonCurve{m}{}} = \sum_{i=1}^{I_S}\inprodMassLamped{u_i}{v_i}{}{\polygonCurve{m}{i}} =\sum_{i=1}^{I_S} \int_{\polygonCurve{m}{i}} u_i v_i\dH{d-1}.
\end{equation*}
The notion of these inner products can be extended
for two vector- and tensor-valued functions.
The vertex normal $\femNormalVertex{m}{i}\in \underline V^h(\polygonCurve{m}{i})$ on $\polygonCurve{m}{i}$ 
is defined in terms of the $L^2$--projection as follows (see \cite{BarrettGarckeRobertBook2020}):
\begin{equation*}
    \inprodMassLamped{\femNormalVertex{m}{i}}{\Vec{\xi}}{h}{\polygonCurve{m}{i}} = \inprodMassLamped{\normal{m}{i}}{\Vec{\xi}}{}{\polygonCurve{m}{i}},\quad\forall\Vec{\xi}\in \underline V^h(\polygonCurve{m}{i}),\qquad i\in\naturalindex{I_S}.
\end{equation*}

Our finite element approximation of \eqref{eq:WFM} is now given as follows. 
Let $\Gamma^0$ be given. Then, for $m\geq0$, find
$(\chemicalDiscVec{m+1}{}, \curvatureDisc{\anisotropy{}}^{m+1}, \X{m+1}{})\in (\femspacebulkvector_D\cap\femspacebulkvectorsum)\times \femspacegamma \times \femspacegammavector$, and set $\Gamma^{m+1} = \vec X^{m+1}(\Gamma^m)$,
such that the following conditions hold:
\newline\newline\noindent
\begin{subequations} \label{eq:FEA}
\textbf{[Motion law]} For all $\bv\varphi\in\femspacebulkvector_0\cap\femspacebulkvectorsum{}$,
\begin{equation}\label{eq:FE-ML-linear}
\innerproduct{\nabla \chemicalDiscVec{m+1}{}}{\nabla\bv\varphi}{}{\Omega} -
\sum_{\ell=1}^{I_R}
\sum_{i=1}^{I_S} \innerproduct{\jumpPhase{\chi_\ell}{\curve{m}{i}}{}
{\Glinear{m}{i}}}{\varphi_\ell}{(h)}{\curve{m}{i}}
= 0.
\end{equation}
\textbf{[Gibbs--Thomson law with kinetic undercooling]} For all $\xi\in \femspacegamma$,
\begin{multline}\label{eq:FE-GTL-linear}
\innerproduct{\curvatureDisc{\anisotropy{}}^{m+1}}{\xi}{h}{\curve{m}{}} -
\sum_{\ell=1}^{I_R}
\sum_{i=1}^{I_S} \innerproduct{\jumpPhase{\chi_\ell}{\curve{m}{i}}{}
W^{m+1}_\ell}{\xi_i}{(h)}{\curve{m}{i}}
- \innerproduct{\frac{\kinetic{}}{\mobility{}(\vec\nu^m)}\frac{\vec X^{m+1}-\identity}{\tau_m}}{\xi\vec\omega^m}{h}{\curve{m}{}}= 0.
\end{multline}
\textbf{[Curvature vector]} For all $\vec\eta\in\femspacegammavector$,
\begin{equation}\label{eq:FE-CV-linear}
    \innerproduct{\curvatureDisc{\anisotropy{}}^{m+1}\weightnormalnoindex{m}{}}{\vec{\eta}}{h}{\curve{m}{}} + \innerproduct{\sgrad^{\tG}\X{m+1}{}}{\sgrad^{\tG}\vec{\eta}}{}{\anisotropy{},\curve{m}{}} = 0,
\end{equation}
\end{subequations}
where, analogously to \eqref{eq:aniip}, we have defined the discrete inner 
product 
\begin{equation*}
\innerproduct{\sgrad^{\tG}\vec\zeta}{\sgrad^{\tG}\vec{\eta}}{}{\anisotropy{},\curve{m}{}}
:= \sum_{i=1}^{I_S} \sum_{\ell=1}^{L_i} \int_{\curve{m}{i}}\innerproductround{\sgrad^{\tG_i^{(\ell)}}\vec\zeta_i}{\sgrad^{\tG_i^{(\ell)}}\vec{\eta}_i}{}{\tG_i^{(\ell)}}\,\anisotropy{i}^{(\ell)}(\normal{m}{i})\dH{d-1}.
\end{equation*}
Observe that here and throughout, 
the notation $\cdot^{(h)}$ means an expression with or
without the superscript $h$.
That is, the scheme \eqref{eq:FEA} represents two different numerical methods:
one with mass lumping in the bulk-interface cross terms in
\eqref{eq:FE-ML-linear} and \eqref{eq:FE-GTL-linear}, 
and one with exact integration.
This follows similar approaches in 
\cite{BarrettGarckeRobert2010,Nurnberg202203,EtoGarckeNurnberg2024,EtoGarckeNurnberg2025}.

Our aim is to prove the well-posedness and unconditional stability 
of the introduced scheme \eqref{eq:FEA}. For the former we make a mild
assumption on the cluster $\Gamma^m$, as well as on the compatibility between
the bulk triangulation $\triangulation{m}$ and $\Gamma^m$.
This follows analogous assumptions being made in 
\cite[Assumptions 64 and 108]{BarrettGarckeRobertBook2020},
see also \cite{EtoGarckeNurnberg2024,EtoGarckeNurnberg2025}.
\begin{asm}\label{asm:1}
    For every $i\in\naturalindex{I_S}$, it holds that
    \begin{equation*}
        \operatorname{span}\left\{\femNormalVertex{m}{i}(\vec q^m_{i,k}) ~|~ k\in\naturalindex{K^0_i}\right\} \neq \{\vec 0\}.
    \end{equation*}
    Moreover, we assume that
    \begin{equation}\label{eq:AssumptionVertexNormal}
        \operatorname{span}\left\{
\sum_{\ell=1}^{I_R}
\sum_{i=1}^{I_S} \innerproduct{\jumpPhase{\chi_\ell}{\curve{m}{i}}{}
{\vec\omega^m_i}}{\varphi_\ell}{(h)}{\curve{m}{i}}
 ~|~\ \bv \varphi\in\femspacebulkvector_0\cap\femspacebulkvectorsum{}\right\} = \bR^d.
    \end{equation}
\end{asm}
\fix{%
We remark that the first condition in \Assumption{asm:1} basically means that each of the $I_S$ surfaces has at least one nonzero inner vertex normal, something that can only be
violated in very pathological cases.
The condition \eqref{eq:AssumptionVertexNormal}, on the other hand,
is a very mild constraint on the interaction between bulk and interface meshes.
In fact, it can only be violated if all the vectors in the set are linearly
dependent, which happens, for example, if all the surfaces are flat and 
lie on top of each other. In all our numerical simulations \Assumption{asm:1} was never violated.}

We now establish the well-posedness of the finite element approximation \eqref{eq:FEA}.
\begin{thm}\label{thm:eu}
Let $\triangulation{m}$ and $\Gamma^m$ satisfy Assumption~\ref{asm:1}.
Then there exists a unique solution 
$$(\chemicalDiscVec{m+1}{},\curvatureDisc{\anisotropy{}}^{m+1},\X{m+1}{})
\in(\femspacebulkvector_D\cap\femspacebulkvectorsum)\times \femspacegamma\times \femspacegammavector$$
to \eqref{eq:FEA}.
\end{thm}
\begin{proof}
Since the system \eqref{eq:FEA} is linear in the unknowns, with as many
unknowns as equations, it suffices to show that
$(\chemicalDiscVec{m+1}{}, \curvatureDisc{\anisotropy{}}^{m+1}, \X{m+1}{})\equiv (\bv 0, 0, \vec{0})$ is the only solution to the homogeneous system of \eqref{eq:FEA}.
Therefore, we assume that
$(\chemicalDiscVec{}{}, \curvatureDisc{\anisotropy{}}, \X{}{})\in 
(\femspacebulkvector_0\cap\femspacebulkvectorsum)\times \femspacegamma{}\times \femspacegammavector$
is a solution to 
\begin{subequations} \label{eq:homo}
\begin{align}\label{eq:homo1}
\innerproduct{\nabla \chemicalDiscVec{}{}}{\nabla\bv\varphi}{}{\Omega} -
\sum_{\ell=1}^{I_R}
\sum_{i=1}^{I_S} \innerproduct{\jumpPhase{\chi_\ell}{\curve{m}{i}}{}
{\Glinearhomo{m}{i}}}{\varphi_\ell}{(h)}{\curve{m}{i}}
= 0
\qquad \forall\bv\varphi\in\femspacebulkvector_0\cap\femspacebulkvectorsum{},
\\
\innerproduct{\curvatureDisc{\anisotropy{}}}{\xi}{h}{\curve{m}{}} -
\sum_{\ell=1}^{I_R}
\sum_{i=1}^{I_S} \innerproduct{\jumpPhase{\chi_\ell}{\curve{m}{i}}{}
W_\ell}{\xi_i}{(h)}{\curve{m}{i}}
- \innerproduct{\frac{\kinetic{}}{\mobility{}(\vec\nu^m)}\frac{\vec X}{\tau_m}}{\xi\vec\omega^m}{h}{\curve{m}{}}= 0
\qquad \forall\xi\in \femspacegamma,
\label{eq:homo2}\\
    \innerproduct{\curvatureDisc{\anisotropy{}}\weightnormalnoindex{m}{}}{\vec{\eta}}{h}{\curve{m}{}} + \innerproduct{\sgrad^{\tG}\X{}{}}{\sgrad^{\tG}\vec{\eta}}{}{\anisotropy{},\curve{m}{}} = 0
\qquad \forall\vec\eta\in\femspacegammavector.
\label{eq:homo3}
\end{align}
\end{subequations}
Choosing $\bv\varphi = \chemicalDiscVec{}{}$ in \eqref{eq:homo1},
$\xi = \pi^h\left[\X{}{}\cdot \weightnormalnoindex{m}{}\right]$ in 
\eqref{eq:homo2},
and $\vec{\eta} = \X{}{}$ in \eqref{eq:homo3}, we obtain
\begin{equation*}
\innerproduct{\sgrad^{\tG}\X{}{}}{\sgrad^{\tG}\X{}{}}{}{\gamma,\curve{m}{}} 
+ \tau_m\|\nabla\chemicalDiscVec{}{}\|^2_{L^2(\Omega)} 
+ \frac1{\tau_m}\innerproduct{\frac{\kinetic{}}{\mobility{}(\normal{m}{})}\X{}{}\cdot\weightnormalnoindex{m}{}}{\X{}{}\cdot\weightnormalnoindex{m}{}}{h}{\curve{m}{}} = 0.
\end{equation*}
Since all the terms on the left-hand side of the above equation are 
non-negative, we can use the fact that all $\tG_i^{(\ell)}$ are positive definite and the definition \eqref{eq:Gtildeprod} to conclude that 
$\chemicalDiscVec{}{}\equiv \bv C \in \zerosumspace$
and $\X{}{}\equiv \X{}{c}$ are constant functions.
If $\pOmega_D \not=\emptyset$ we immediately get $\bv C=\bv 0$.
Meanwhile, we deduce from \eqref{eq:homo1} that for any $\bv\varphi\in \femspacebulkvector_0\cap\femspacebulkvectorsum$,
\begin{equation*}
    \X{}{c}\cdot
\sum_{\ell=1}^{I_R}
\sum_{i=1}^{I_S} \innerproduct{\jumpPhase{\chi_\ell}{\curve{m}{i}}{}
{\vec\omega^m_i}}{\varphi_\ell}{(h)}{\curve{m}{i}}
=0.
\end{equation*}
We see from \Assumption{asm:1} that the above equation holds if and only if $\X{}{c} = \vec{0}$.
Hence it follows from $\eqref{eq:homo2}$ that 
\begin{equation} \label{eq:kappaconstant}
\kappa_{\gamma,i} = 
\sum_{\ell=1}^{I_R} \jumpPhase{\chi_\ell}{\curve{m}{i}}{} C_\ell
\end{equation}
is also equal to a constant, for $i \in\naturalset{I_S}$. 
With these constants, we now define
\begin{equation} \label{eq:vecz}
    \vec\eta_i := \curvatureDisc{\anisotropy{},i} \sum_{j=1}^{K_i^0}
  \weightnormalnoindex{m}{i}(\vec q^m_{i,\ell})  {\basiscurve{i}{j}}\in [V^h_0(\polygonCurve{m}{i})]^d \qquad\mbox{for}\ i \in\naturalset{I_S}.
\end{equation}
On recalling \eqref{eq:defV0} we observe
that $\vec\eta=(\vec{\eta}_1,\ldots,\vec{\eta}_{I_S})\in\femspacegammavector$.
Thus, on choosing this $\vec\eta$ in \eqref{eq:homo3} we obtain 
\begin{equation*}
    0 = \innerproduct{\curvatureDisc{\anisotropy{}}\weightnormalnoindex{m}{}}{\vec{\eta}}{h}{\curve{m}{}} = \sum_{i=1}^{I_S}\curvatureDisc{\anisotropy{},i}\innerproduct{\weightnormalnoindex{m}{i}}{\vec{\eta}_i}{h}{\curve{m}{i}}
    = \sum_{i=1}^{I_S}(\curvatureDisc{\anisotropy{},i})^2\sum_{j=1}^{K_i^0} |\weightnormalnoindex{m}{i}(\vec q^m_{i,\ell})|^2 \innerproduct
    	{\basiscurve{i}{j}} {\basiscurve{i}{j}}{h}{\curve{m}{i}}
  .
\end{equation*}
\Assumption{asm:1} now immediately implies
that $\curvatureDisc{\anisotropy{},i}= 0$ for 
$i \in\naturalset{I_S}$, and hence $\curvatureDisc{\anisotropy{}}\equiv 0$.

Finally, \fix{%
    for every $i\in\naturalset{I_S}$, there exist $\ell_i^+$ and $\ell_i^-$ with $\ell_i^+ \neq \ell_i^-$ such that $\jumpPhase{\chi_{\ell_i^\pm}}{\curve{m}{i}}{} = \pm 1$, and hence \eqref{eq:kappaconstant} yields
    \[
    0 = \sum_{\ell=1}^{I_R} \jumpPhase{\chi_\ell}{\curve{m}{i}}{} C_\ell = C_{\ell_i^+} - C_{\ell_i^-}.
    \]
    Since $\Omega$ is connected, we see that $C_1 = \ldots = C_{I_R}$, and so
    on recalling $\bv C \in \zerosumspace$, we obtain that $\bv C = \bv 0$.}
This concludes the proof.
\end{proof}

We now show a discrete analogue of the energy dissipation law in
\Proposition{prop:dissp}.

\begin{thm}[Discrete energy dissipation]\label{thm:ded}
Let $(\chemicalDiscVec{m+1}{},\curvatureDisc{\anisotropy{}}^{m+1},\X{m+1}{})\in(\femspacebulkvector_D\cap\femspacebulkvectorsum)\times\femspacegamma{}\times\femspacegammavector{}$
be the solution to \eqref{eq:FEA}. Then, the following inequality holds:
\begin{align} \label{eq:ded}
& |\curve{m+1}{}|_{\anisotropy{}} +
\tau_m \sum_{\ell=1}^{I_R} w_{D,\ell}
\sum_{i=1}^{I_S} \innerproduct{\jumpPhase{\chi_\ell}{\curve{m}{i}}{}
\frac{\X{m+1}{} - \identity}{\tau_m}}{\weightnormalnoindex{m}{i}}{(h)}{\curve{m}{i}}
\nonumber \\ & \quad
+ \tau_m\|\nabla\chemicalDiscVec{m+1}{}\|^2_{L^2(\Omega)} 
+ \tau_m\innerproduct{\frac{\kinetic{}}{\mobility{}(\normal{m}{})}}{\left|\frac{\X{m+1}{} - \identity}{\tau_m}\cdot\weightnormalnoindex{m}{}\right|^2}{h}{\curve{m}{}}
\leq |\curve{m}{}|_{\anisotropy{}}.
\end{align}
\end{thm}
\begin{proof}
Choosing $\bv\varphi = \chemicalDiscVec{m+1}{} - \chemicalVec{}{D}\in \femspacebulkvector_0\cap\femspacebulkvectorsum$ in \eqref{eq:FE-ML-linear}, 
$\xi = \pi^h\left[(\X{m+1}{} - \identity\!\mid_{\Gamma^m})\cdot\weightnormalnoindex{m}{}\right]$
in \eqref{eq:FE-GTL-linear} 
and $\vec{\eta} = \X{m+1}{} - \identity\!\mid_{\Gamma^m}$
in \eqref{eq:FE-CV-linear} yields that
\begin{subequations} 
\begin{align}\label{eq:ded-1}
&\|\nabla\chemicalDiscVec{m+1}{}\|^2_{L^2(\Omega)} -
\sum_{\ell=1}^{I_R}
\sum_{i=1}^{I_S} \innerproduct{\jumpPhase{\chi_\ell}{\curve{m}{i}}{}
{\Glinear{m}{i}}}{W^{m+1}_\ell - w_{D,\ell}}{(h)}{\curve{m}{i}}
= 0,\\
\label{eq:ded-2}
&\innerproduct{\curvatureDisc{\anisotropy{}}^{m+1}\weightnormalnoindex{m}{}}{\X{m+1}{} - \identity}{h}{\curve{m}{}} -
\sum_{\ell=1}^{I_R}
\sum_{i=1}^{I_S} \innerproduct{\jumpPhase{\chi_\ell}{\curve{m}{i}}{}
W^{m+1}_\ell}{\pi^h_i\left[(\X{m+1}{} - \identity)\cdot\weightnormalnoindex{m}{i}\right]}{(h)}{\curve{m}{i}} \nonumber\\&\qquad\qquad
- \innerproduct{\frac{\kinetic{}}{\mobility{}(\normal{m}{})}\frac{\X{m+1}{} - \identity}{\tau_m}\cdot\weightnormalnoindex{m}{}}{(\X{m+1}{} - \identity)\cdot\weightnormalnoindex{m}{}}{h}{\curve{m}{}} = 0, \\
\label{eq:ded-3}
&\innerproduct{\curvatureDisc{\anisotropy{}}^{m+1}\weightnormalnoindex{m}{}}{\X{m+1}{} - \identity}{h}{\curve{m}{}}
+ \innerproduct{\sgrad^{\tG}\X{m+1}{}}{\sgrad^{\tG}(\X{m+1}{} - \identity)}{}{\anisotropy{},\curve{m}{}} = 0.
\end{align}
\end{subequations}
Combining \eqref{eq:ded-1}, \eqref{eq:ded-2} and \eqref{eq:ded-3} yields that
\begin{align} \label{eq:ded-4}
&\tau_m \|\nabla\chemicalDiscVec{m+1}{}\|^2_{L^2(\Omega)} +
\sum_{\ell=1}^{I_R}
\sum_{i=1}^{I_S} \innerproduct{\jumpPhase{\chi_\ell}{\curve{m}{i}}{}
{\pi^h_i\left[(\X{m+1}{} - \identity)\cdot\weightnormalnoindex{m}{i}\right]}}{w_{D,\ell}}{(h)}{\curve{m}{i}} \nonumber\\ & \quad
+ \frac1{\tau_m}
\innerproduct{\frac{\kinetic{}}{\mobility{}(\normal{m}{})}}{|(\X{m+1}{} - \identity)\cdot\weightnormalnoindex{m}{}|^2}{h}{\curve{m}{}}
+
\innerproduct{\sgrad^{\tG}\X{m+1}{}}{\sgrad^{\tG}(\X{m+1}{} - \identity)}{}{\anisotropy{},\curve{m}{}} =0.
\end{align} 
We now recall the following discrete anisotropic energy estimate from 
\cite[Lemma~102]{BarrettGarckeRobertBook2020}:
\begin{equation} \label{eq:ani3d}
\innerproduct{\sgrad^{\tG}\X{m+1}{}}{\sgrad^{\tG}(\X{m+1}{} - \identity)}{}{\anisotropy{},\curve{m}{}}
\geq |\curve{m+1}{}|_{\anisotropy{}} - |\curve{m}{}|_{\anisotropy{}},
\end{equation}
see also \cite{BarrettGarckeNurnberg2008IMA,BarrettGarckeNurnberg2008NM}. 
The desired result \eqref{eq:ded} directly follows from \eqref{eq:ded-4} and
\eqref{eq:ani3d}.
\end{proof}

\begin{rem}
It is not difficult to see that \eqref{eq:ded} is a discrete analogue of the
dissipation property \eqref{eq:dissipation}. In fact, we only need to
recall from \eqref{eq:dissp9} that
\begin{align*}
-\ddt \sum_{\ell = 1}^{I_R}w_{D,\ell}\volume{\subdomains{\ell}}
= \sum_{\ell = 1}^{I_R}w_{D,\ell}\sum_{i=1}^{I_S} \int_{\curve{}{i}(t)} [\chi_\ell]_{\Gamma_i} V_i \dH{d-1} .
\end{align*}
Moreover, we note that \eqref{eq:ded} is the natural generalization of
the discrete energy estimate \cite[Theorem~3.1]{BarrettGarckeRobert2010} 
to the multi-phase problem considered here.
\end{rem}

\section{Solution methods}\label{sec:MatrixForm}
In this section, we discuss solution methods for the system of
linear equations arising from \eqref{eq:FEA} at each time level.
To this end, we make use of ideas from 
\cite{BarrettGarckeRobert2007,BarrettGarckeRobert2010}, see also
\cite{EtoGarckeNurnberg2024,EtoGarckeNurnberg2025}.
Here the crucial idea is to avoid having to work with the 
trial and test spaces $\femspacegammavector$ and
$\femspacebulkvectorsum$ directly, and rather employ
a technique that is similar to a standard treatment of 
periodic boundary conditions for ODEs and PDEs.

We introduce the orthogonal projections 
$\mathcal{P}:[V^h(\curve{m}{})]^d \to \femspacegammavector$ 
and $\mathcal{Q}:\femspacebulkvector\to \femspacebulkvectorsum$,
where as inner product in each case we consider the mass lumped 
$L^2$--inner product. Firstly, 
it is easy to see that for 
$\bv W \in\femspacebulkvector$ it holds that
$\mathcal{Q}\chemicalDiscVec{}{} = \chemicalDiscVec{}{} -
\dfrac{\chemicalDiscVec{}{}\cdot\bvone}{\bvone\cdot\bvone}\, \bvone$
point-wise in $\overline\Omega$.

Now, given $\X{m}{}:=\identity\!\mid_{\Gamma^m} \in \femspacegammavector$, 
let $(\chemicalDiscVec{m+1}{},\kappa^{m+1}_\gamma,\X{m}{}+\delta\X{m+1}{})\in(\femspacebulkvector_D\cap\femspacebulkvectorsum)\times V^h(\curve{m}{})\times \femspacegammavector$ be the unique solution to \eqref{eq:FEA} whose existence has been proven in Theorem~\ref{thm:eu}. 
{From} now on, as no confusion can arise, we identify
$(\chemicalDiscVec{m+1}{},\kappa^{m+1}_\gamma,\delta\X{m+1}{})$ with their
vectors of coefficients with respect to the bases $\{\basisbulk{m}{i}\}_{1\leq i\leq\spacemeshpointsnum}$ and $\{\{\basiscurve{i}{j}\}_{1\leq j\leq N_i}\}_{i=1}^{I_S}$ of the unconstrained spaces $\femspacebulkvector$ and 
$V^h(\curve{m}{})$. 
Let $N=\sum_{i=1}^{I_S} N_i$ and $K= I_R K_\Omega^m$. Then, in addition, we let 
$\metric{P} : (\bR^d)^N \to \mathbb X \subset (\bR^d)^N$ be the Euclidean
space equivalent of $\mathcal{P}$, and similarly for the Euclidean equivalent
$Q : \bR^K \to \mathbb{W} \subset \bR^K$ of $\mathcal{Q}$.

To simplify the next part of the presentation, we first consider the 
three-phase case in two space dimensions, whose 
setting is shown in \Figure{fig:3p}, i.e.\ $d=2$, $I_S=I_R=3$, $I_T=2$.
In addition, we assume that $\partial\Omega_D = \emptyset$.
Then the solution to \eqref{eq:FEA} can be written as $(Q\chemicalDiscVec{m+1}{},\curvatureDisc{\anisotropy{}}^{m+1},\X{m}{} + \metric{P}\delta\X{m+1}{})$
for any solution of the linear system
\begin{equation}\label{eq:MatrixRep}
\begin{pmatrix}
QA_\Omega Q & O & Q \vec{N}_{\Omega,\Gamma}^\top\metric{P} \\
B_{\Omega,\Gamma} Q & C_\Gamma & -\vec{D}_\Gamma^{(\mobility{})}\metric{P} \\
O & \metric{P}\vec{D}_{\Gamma} & \metric{P}\metric{E_\Gamma^{(\anisotropy{})}}\metric{P}
\end{pmatrix}
\begin{pmatrix}
\chemicalDiscVec{m+1}{} \\
\curvatureDisc{\anisotropy{}}^{m+1} \\
\delta\X{m+1}{}
\end{pmatrix} = \begin{pmatrix}
O \\
O \\
-\metric{P}\metric{E_\Gamma^{(\anisotropy{})}}\metric{P}\X{m}{}
\end{pmatrix},
\end{equation}
where $A_\Omega\in\mathbb{R}^{K\times K},\vec{N}_{\Omega,\Gamma}\in(\mathbb{R}^d)^{N\times K}$,
$B_{\Omega,\Gamma}\in\mathbb{R}^{N\times K}$,
$C_\Gamma\in\mathbb{R}^{N\times N}$,
$\vec{D}_\Gamma\in(\mathbb{R}^d)^{N\times N}$,
$\vec{D}_\Gamma^{(\beta)}\in(\mathbb{R}^d)^{N\times N}$, and
$\metric{{E}_\Gamma^{(\anisotropy{})}}\in(\mathbb{R}^{d\times d})^{N\times N}$ are defined by
\begin{align*}
&A_\Omega := 
\begin{pmatrix}
A & O & O  \\
O & A & O \\
O & O & A
\end{pmatrix},\ 
\vec{N}_{\Omega,\Gamma} :=
\begin{pmatrix}
O & \vec{N}_1 & -\vec{N}_1 \\
-\vec{N}_2 & O & \vec{N}_2 \\
\vec{N}_3 & -\vec{N}_3 & O
\end{pmatrix},\ 
B_{\Omega,\Gamma} := 
\begin{pmatrix}
O & B_1 & -B_1 \\
-B_2 & O & B_2 \\
B_3 & -B_3 & O \\
\end{pmatrix},\\ 
&C_\Gamma := 
\begin{pmatrix}
C_1 & O & O \\
O & C_2 & O \\
O & O & C_3 \\
\end{pmatrix},\ 
\vec{D}_\Gamma :=
\begin{pmatrix}
\vec{D}_1 & O & O \\
O & \vec{D}_2 & O \\
O & O & \vec{D}_3
\end{pmatrix},\\ 
&\vec{D}_\Gamma^{(\beta)} :=
\begin{pmatrix}
\vec{D}_1^{(\beta)} & O & O \\
O & \vec{D}_2^{(\beta)} & O \\
O & O & \vec{D}_3^{(\beta)}
\end{pmatrix},\ 
\metric{{E}_\Gamma^{(\anisotropy{})}} :=
\begin{pmatrix}
\metric{E_1^{(\anisotropy{})}} & O & O \\
O & \metric{E_2^{(\anisotropy{})}} & O \\
O & O & \metric{E_3^{(\anisotropy{})}} \\
\end{pmatrix},
\end{align*}
with
\begin{equation*}
\begin{array}{ll}
\left[A\right]_{i,j} := \innerproduct{\nabla\Psi^m_j}{\nabla\Psi^m_i}{}{\Omega},  & \left[\vec{N}_c\right]_{l,i} := \frac1{\tau_m}\innerproduct{\Phi^m_{c,l}}{\Psi^m_i}{\Fdd{(h)}}{\curve{m}{c}}\weightnormal{m}{c}{l},\\
\left[B_c\right]_{k,j} := \innerproduct{\Psi^m_j}{\Phi^m_{c,k}}{\Fdd{(h)}}{\curve{m}{c}}, & \left[C_c\right]_{k,l} := \innerproduct{\Phi^m_{c,l}}{\Phi^m_{c,k}}{h}{\curve{m}{c}}, \\
\left[\vec{D}_c\right]_{k,l} := \innerproduct{\Phi^m_{c,l}}{\Phi^m_{c,k}}{h}{\curve{m}{c}}\weightnormal{m}{c}{l}, & \left[\vec{D}_c^{(\beta)}\right]_{k,l} := \frac1{\tau_m}\innerproduct{\frac{\kinetic{c}}{\beta_c(\normal{m}{c}{})}\Phi^m_{c,l}}{\Phi^m_{c,k}}{h}{\curve{m}{c}}\weightnormal{m}{c}{l}, \\
\left[\metric{{E}_c^{(\anisotropy{})}}\right]_{k,l} := \left(\innerproduct{\sgrad^{\tG}(\Phi^m_{c,l}\unitvector{j})}{\sgrad^{\tG}(\Phi^m_{c,k}\unitvector{i})}{}{\anisotropy{},\curve{m}{c}}
\right)_{i,j=1}^d, & \\
\end{array}
\end{equation*}
for each $c\in\naturalset{I_S}$, where $\{\unitvector{j}\}_{j\in\naturalset{d}}$
denotes the Euclidean standard basis in $\bR^d$.

The advantage of the system \eqref{eq:MatrixRep} over a naive implementation
of \eqref{eq:FEA} is that complications due to nonstandard finite element
spaces are completely avoided. A disadvantage is, however, that the system
\eqref{eq:MatrixRep} is highly singular, in that due to the presence 
of the projections the dimension of its kernel is larger than the dimension of
the scalar bulk finite element space $S^m$.
This
makes it difficult to solve \eqref{eq:MatrixRep} in practice. A more practical
formulation can be obtained by eliminating one of the components of
$\chemicalDiscVec{m+1}{}$. 
In particular, on recalling that $\chemicalDiscVec{m+1}{} \cdot \bvone = 0$, we can reduce the unknown variables $\chemicalDiscVec{m+1}{}\in\mathbb{R}^K$ to $(W^{m+1}_1,W^{m+1}_2)\in\mathbb{R}^{K-\spacemeshpointsnum}$ 
by introducing the linear map $\widehat{Q}:\mathbb{R}^{K-\spacemeshpointsnum}\to \mathbb{W}\subset\mathbb{R}^K$ defined by
\begin{equation*}
\widehat{Q} := \begin{pmatrix}
I_{\spacemeshpointsnum} & O \\
O & I_{\spacemeshpointsnum}\\
-I_{\spacemeshpointsnum} & -I_{\spacemeshpointsnum}\\
\end{pmatrix},
\end{equation*}
where $I_M$ denotes the identity matrix of size $M$ for $M\in\mathbb{N}$. 
Then the solution to \eqref{eq:FEA} can be written as
$(\widehat Q \widehat{\bv W}^{m+1}, \kappa^{m+1}_\gamma,
 \X{m}{} + \metric{P}\delta\X{m+1}{})$ for any solution of the reduced
linear system 
\begin{equation}\label{eq:MatrixRepReduced}
\begin{pmatrix}
 \widehat{A}_\Omega & O & \widehat{N}_{\Omega,\Gamma}^\top\metric{P} \\
 \widehat{B}_{\Omega,\Gamma} & C_\Gamma & -\vec{D}_\Gamma^{(\mobility{})}\metric{P} \\
O & \metric{P}\vec{D}_{\Gamma} & \metric{P}\metric{E_\Gamma^{(\anisotropy{})}}\metric{P}
\end{pmatrix}
\begin{pmatrix}
\widehat{\chemicalDiscVec{}{}}^{m+1} \\
\curvatureDisc{\anisotropy{}}^{m+1} \\
\delta\X{m+1}{}
\end{pmatrix} = \begin{pmatrix}
O \\
O \\
-\metric{P}\metric{E_\Gamma^{(\anisotropy{})}}\metric{P}\X{m}{}
\end{pmatrix},
\end{equation}
where
\begin{equation*}
\widehat{A}_\Omega = 
\begin{pmatrix}
A & O  \\
O & A 
\end{pmatrix},\quad
\widehat{B}_{\Omega,\Gamma} :=  B_{\Omega,\Gamma}\widehat{Q} =
\begin{pmatrix}
B_1 & 2B_1 \\
-2B_2 & -B_2 \\
B_3 & -B_3 \\
\end{pmatrix},\quad
\widehat{N}_{\Omega,\Gamma} :=
\begin{pmatrix}
    O & \vec{N}_1 \\
    -\vec{N}_2 & O \\
    \vec{N}_3 & -\vec{N}_3
\end{pmatrix}.
\end{equation*}
In contrast to \eqref{eq:MatrixRep}, the kernel of 
\eqref{eq:MatrixRepReduced} is small, as it only involves the projections
concerning the triple junction attachment conditions.
Hence, iterative solution methods, combined with good
preconditioners, work very well to solve \eqref{eq:MatrixRepReduced} in
practice.

\begin{rem}
For illustrative purposes, we presented the matrix formulations 
\eqref{eq:MatrixRep} and \eqref{eq:MatrixRepReduced} for the simple
curve network shown in Figure~\ref{fig:3p}, and for $\partial\Omega_D = \emptyset$.
Extending these matrix formulations to general surface clusters and situations with $\partial\Omega_D \not= \emptyset$ is straightforward.
For example, we obtain for the block matrices
in \eqref{eq:MatrixRep} that
$A_\Omega := \operatorname{diag}(A)_{\ell=1,\ldots,I_R}$, 
$\vec{N}_{\Omega,\Gamma} := (\jumpPhase{\chi_\ell}{\curve{}{c}}{}\vec{N}_c)_{c=1,\ldots,I_S,\ell=1,\ldots,I_R}$,
$B_{\Omega,\Gamma} := (\jumpPhase{\chi_\ell}{\curve{}{c}}{}B_c)_{c=1,\ldots,I_S,\ell=1,\ldots,I_R}$, 
$C_{\curve{}{}} := \operatorname{diag}(C_i)_{i=1,\ldots,I_S}$,
$\vec{D}_{\curve{}{}} := \operatorname{diag}(\vec{D}_i)_{i=1,\ldots,I_S}$, 
$\vec{D}^{(\mobility{})}_{\curve{}{}} := \operatorname{diag}(\vec{D}_i^{(\mobility{})})_{i=1,\ldots,I_S}$ and
$\metric{E_{\curve{}{}}^{(\anisotropy{})}} := \operatorname{diag}(\metric{E_i^{(\anisotropy{})}})_{i=1,\ldots,I_S}$.
Once again, the generalized system corresponding to \eqref{eq:MatrixRep} can 
be reduced by eliminating the final component $W^{m+1}_{I_R}$ from 
$\chemicalDiscVec{m+1}{}$.
We obtain the same block structure as in \eqref{eq:MatrixRepReduced},
with the new entries now given by
$\widehat A_\Omega = \operatorname{diag}{(A)_{\ell=1,\ldots,I_R-1}}$,
$\widehat {B}_{\Omega,\Gamma} = 
(([\chi_\ell]_{\Gamma_i} - [\chi_{I_R-1}]_{\Gamma_i})
B_i)_{i=1,\ldots,I_S,\ell=1,\ldots,I_R-1}$ and
$\widehat {N}_{\Omega,\Gamma} =
(-[\chi_\ell]_{\Gamma_i}\vec{N}_i)_{i=1,\ldots,I_S,\ell=1,\ldots,I_R-1}$.
\end{rem}

\section{Numerical results}\label{sec:NumericalResults}

We implemented the fully discrete finite element approximation
\eqref{eq:FEA} within the finite element toolbox ALBERTA, see \cite{Alberta}. 
The arising linear systems of the form \eqref{eq:MatrixRepReduced} 
are solved with a GMRes iterative solver with the following preconditioners.
In 2D, we take as preconditioner a least squares solution of the block matrix 
in \eqref{eq:MatrixRepReduced} without $\metric{P}$. Due to the large memory
requirements of this preconditioner in 3D, as an alternative in 3D we take
as preconditioner a least squares solution of the lower triangular block 
of the matrix in \eqref{eq:MatrixRepReduced} without $\metric{P}$.
For the computation of the least squares solution we employ
the sparse factorization package SPQR, see \cite{Davis11},
while for the inversion of nonsingular blocks we use the sparse
factorization package UMFPACK, see \cite{Davis04}.

We employ an unfitted finite element discretization. The precise description
of the adaptively refined and coarsened bulk triangulations $\triangulation{m}$ 
can be found in \cite[\S5.1]{BarrettGarckeRobert2010}, see also the recent
works \cite{EtoGarckeNurnberg2024,EtoGarckeNurnberg2025,fluidfbptj}.
We stress that due to the unfitted nature of our finite element approximations,
special quadrature rules need to be employed in order to assemble terms that
feature both bulk and surface finite element functions. For all the
computations presented in this section, we use true integration for these
terms, and we refer to
\cite{BarrettGarckeRobert2010,Nurnberg202203} for details on the practical
implementation.
We note that our theoretical framework does not allow for changes of 
topology, e.g., the vanishing of an interface. Hence, following our
previous work in \cite{EtoGarckeNurnberg2024}, in our 2D computations we 
perform heuristic surgeries whenever a curve becomes too short.
Here a closed curve is simply discarded, while a curve that was part of a 
network is removed. This will leave two triple junctions, where only two 
curves meet, and the involved curves can be glued together so that the
simulation can continue.
Throughout this section we use uniform time steps, in that
$\tau_m=\tau$ for $m=0,\ldots, M-1$.
Unless otherwise stated we choose $\rho_i=1$ and $\beta_i \equiv 1$
for $i\in\naturalset{I_S}$.
We always choose $\Omega=(-4,4)^d$.

In analogue to \eqref{eq:quantity}, we define the discrete energy as
\[
\mathcal{E}^m = |\Gamma^m|_\gamma - \sum_{\ell=1}^{I_R} w_{D,\ell} 
\volume{\mathcal{R}_\ell^m},
\]
where $\mathcal{R}_\ell^m$ are the natural discrete analogues of the subdomains
$\mathcal{R}_\ell[\Gamma(t_m)]$, $\ell=1,\ldots,I_R$.
Moreover, we recall the convention that $\bv w_D = \bv 0$ when
$\pOmega_N=\partial\Omega$.

In order to describe the orientations of the surface clusters with respect to
the bulk regions they enclose, we follow the notation from
\cite{EtoGarckeNurnberg2024} and define the matrices
$\mathcal{O} \in \{-1,0,1\}^{I_R \times I_S}$ with entries
\[
\mathcal{O}_{\ell i} = - [\chi_\ell]_{\Gamma_i}.
\]

\subsection{Simulations in 2D with $\pOmega_N = \pOmega$}

In this subsection we consider some numerical simulations for $d=2$ and
$\pOmega_N = \pOmega$. We also let $\rho=0$. For the anisotropy we define
\begin{equation} \label{eq:hex2d}
\gamma_{\rm hex}(\vec p)
:= \sum_{\ell=1}^3 \sqrt{ [(R(\tfrac\pi{3})^\ell]^\top D(\delta) 
(R(\tfrac\pi{3}))^\ell \vec p \cdot \vec p}, \quad \delta = 0.1,
\end{equation}
where 
$R(\theta)=\binom{\phantom{-}\cos\theta\ \sin\theta}{-\sin\theta\ \cos\theta}$
and $D(\delta) = \diag(1,\delta^2)$. Observe that the Wulff shape of
\eqref{eq:hex2d} is given by a smoothed hexagon, see
\cite{BarrettGarckeNurnberg2008IMA} for details.
Unless otherwise stated, we choose $\gamma_i = \gamma_{\rm hex}$, 
$i\in\naturalset{I_S}$.

\vspace{0.3cm}
\noindent
{\bf Example 1}:
We investigate how a standard double bubble and
a disk evolve, when the first  phase is made up of the left part of the double  bubble, and the second 
phase is made up of the right part of the double  bubble and the disk.
We have $I_S = 4$, $I_R = 3$, $I_T = 2$, 
$(\curveindex{1}{1},\curveindex{1}{2},\curveindex{1}{3}) = (\curveindex{2}{1},\curveindex{2}{2},\curveindex{2}{3}) = (1,2,3)$
and
$\dcmap = \begin{pmatrix}0 &-1&1&0 \\ 1 & 0 &-1&-1 \\ -1 &1&0&1 \end{pmatrix}$.
The two bubbles of the double bubble enclose an area of about $3.139$ each,
while the disk has an initial radius of $\frac58$, meaning it initially
encloses an area of $\frac{25\pi}{64} \approx 1.227$. During the evolution the
disk vanishes, and the right bubble grows correspondingly, see 
Figure~\ref{fig:2dhex_db_plus_one}.
\begin{figure}
\center
\includegraphics[angle=-90,width=0.18\textwidth]{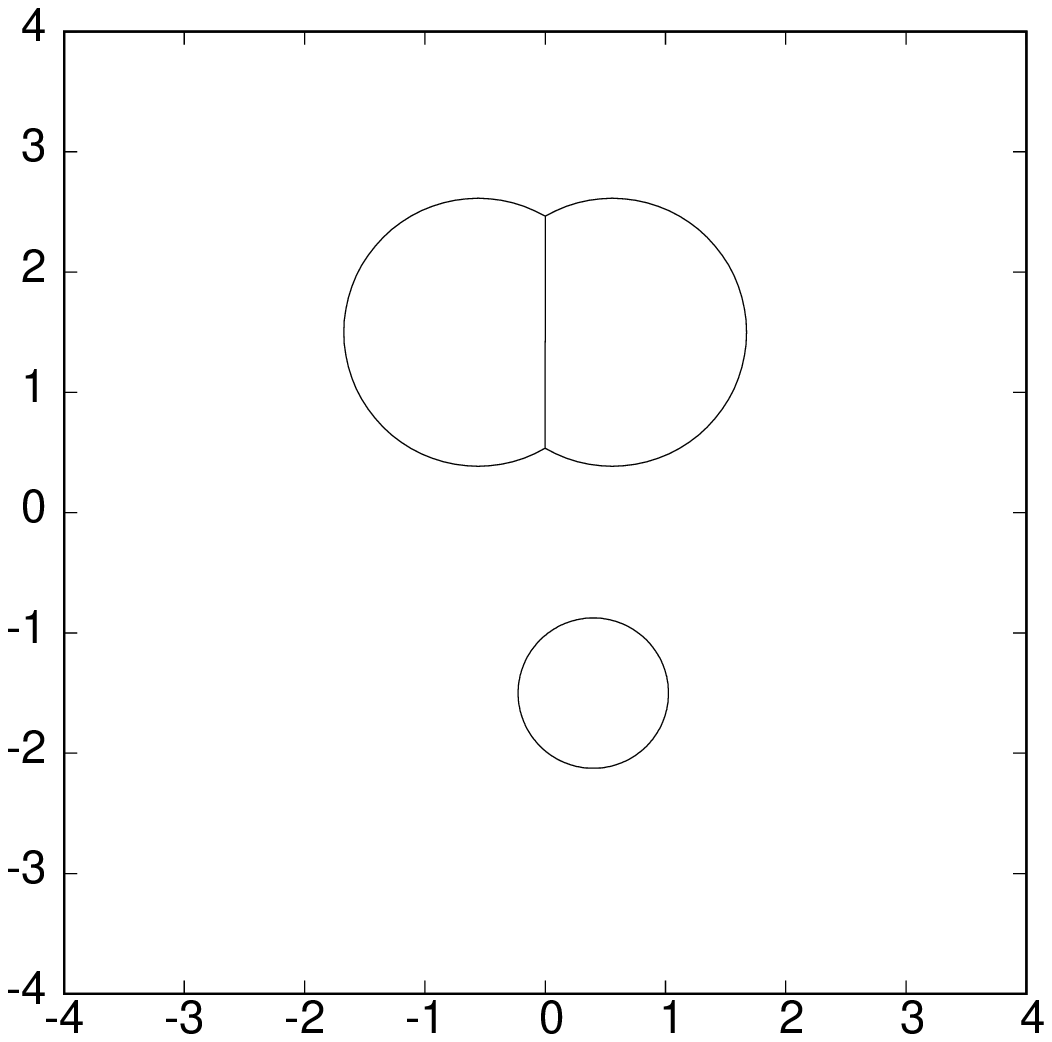}
\includegraphics[angle=-90,width=0.18\textwidth]{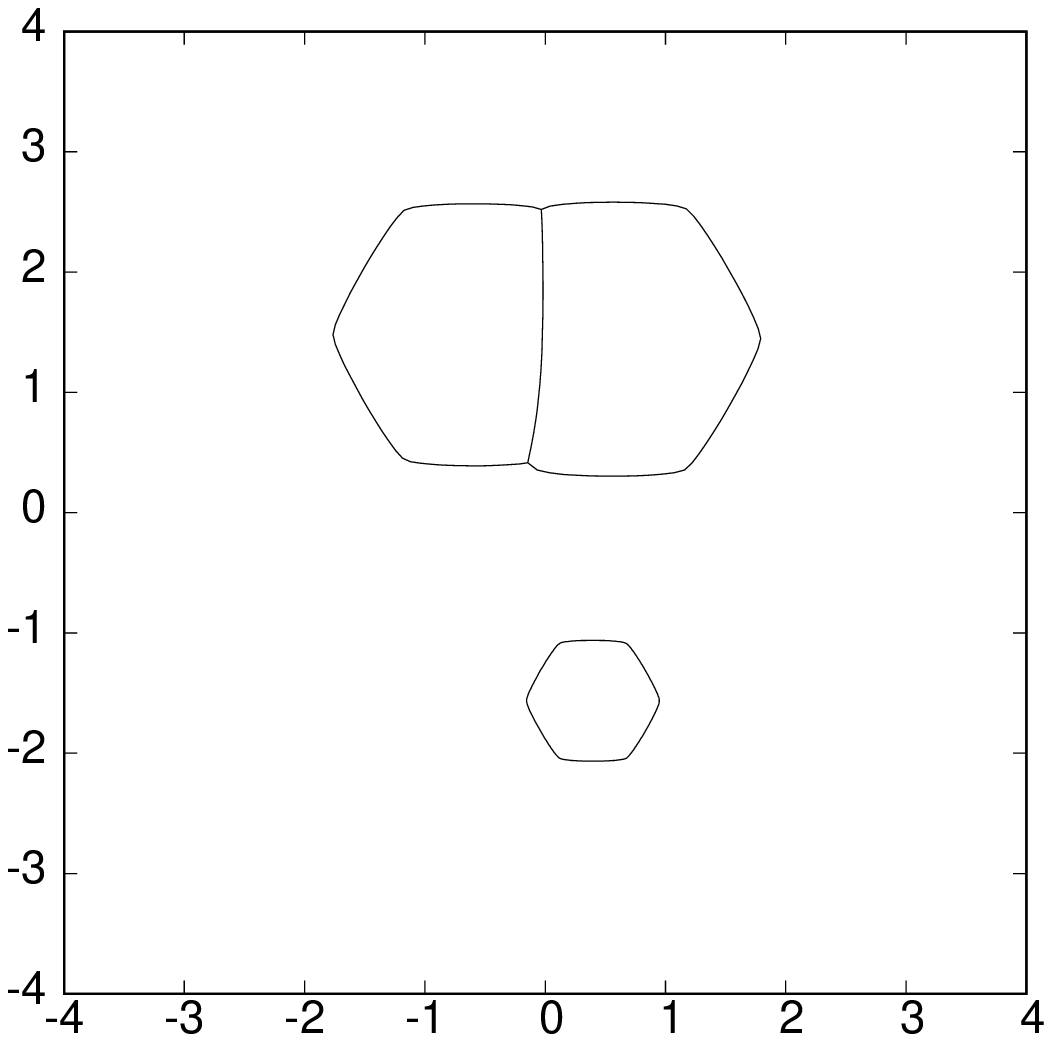}
\includegraphics[angle=-90,width=0.18\textwidth]{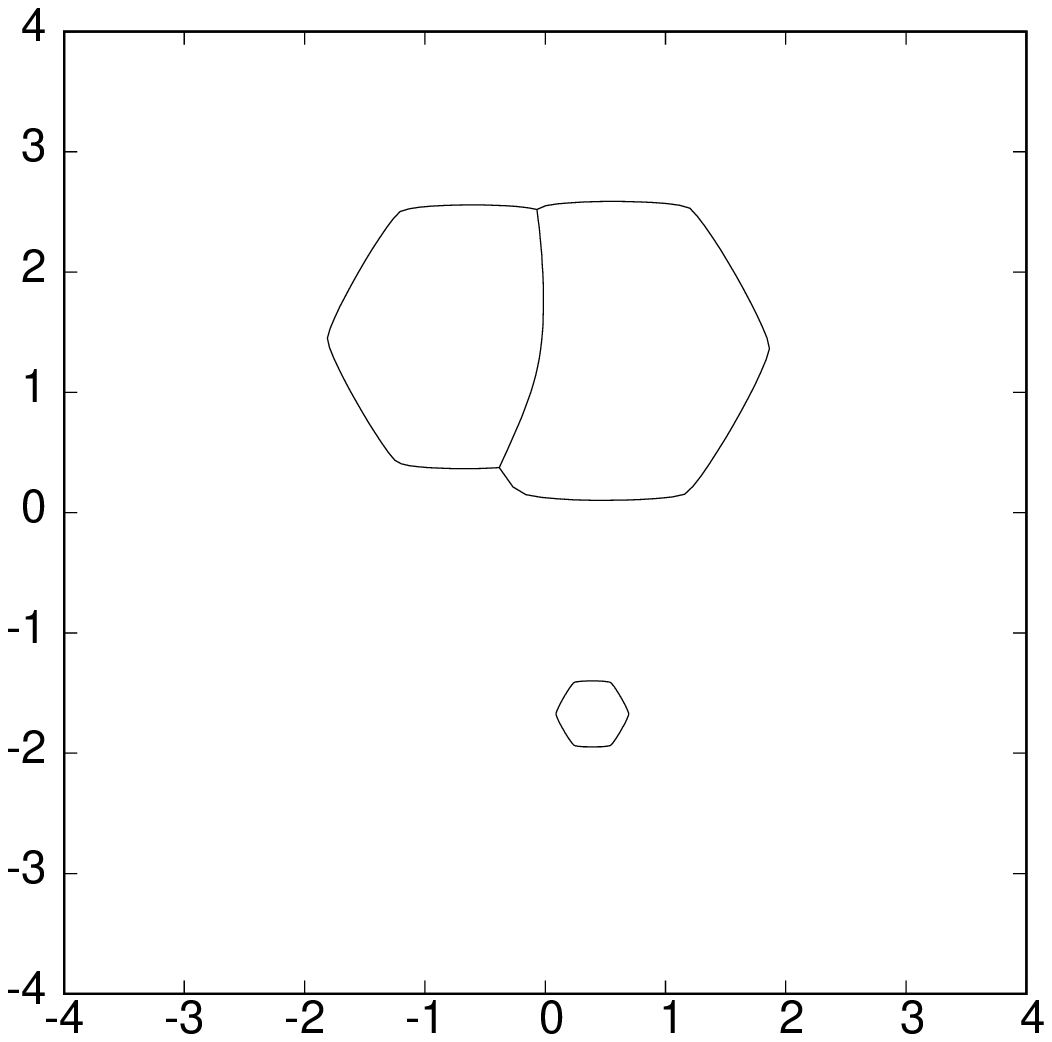}
\includegraphics[angle=-90,width=0.18\textwidth]{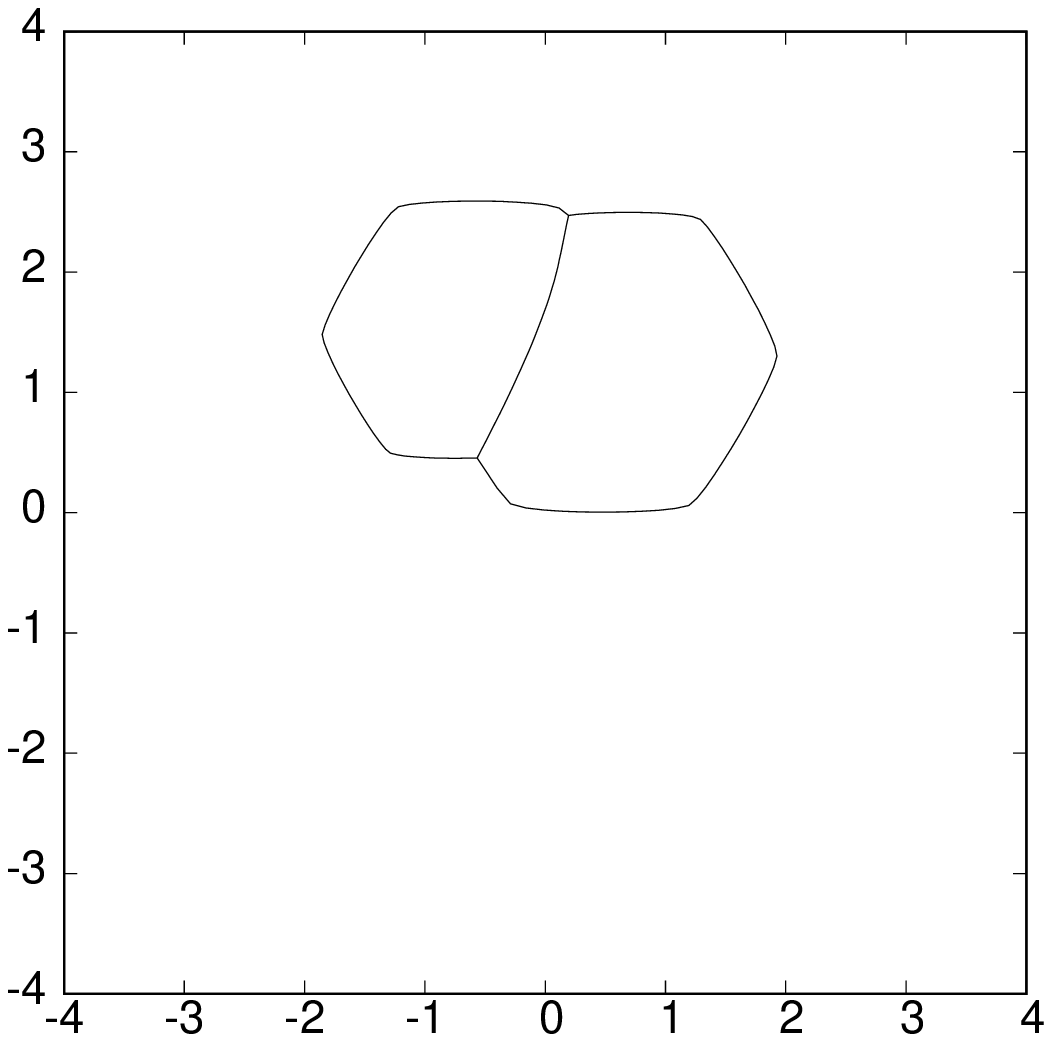}
\includegraphics[angle=-90,width=0.25\textwidth]{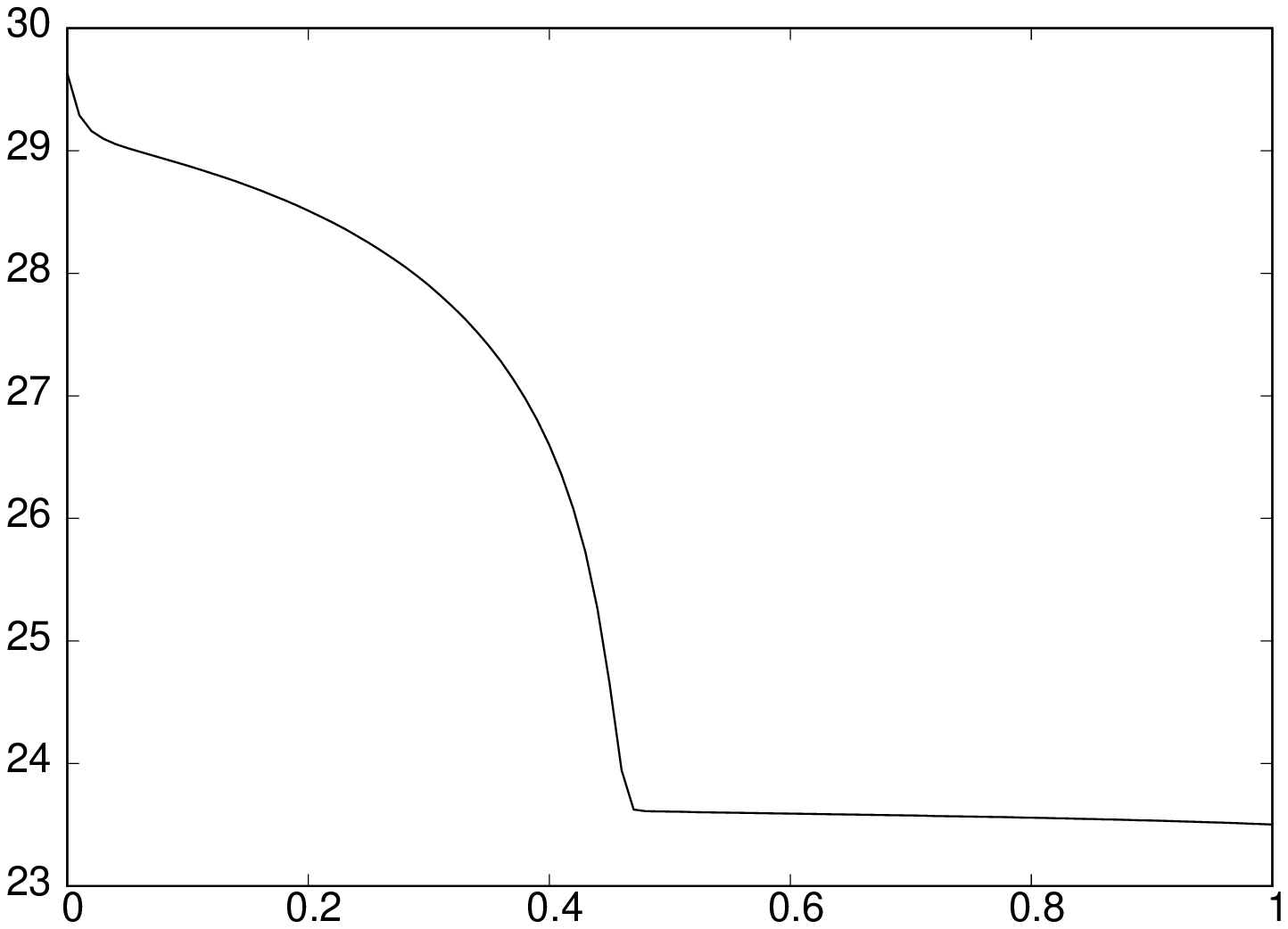} 
\\[5mm]
\includegraphics[angle=-0,width=0.18\textwidth]{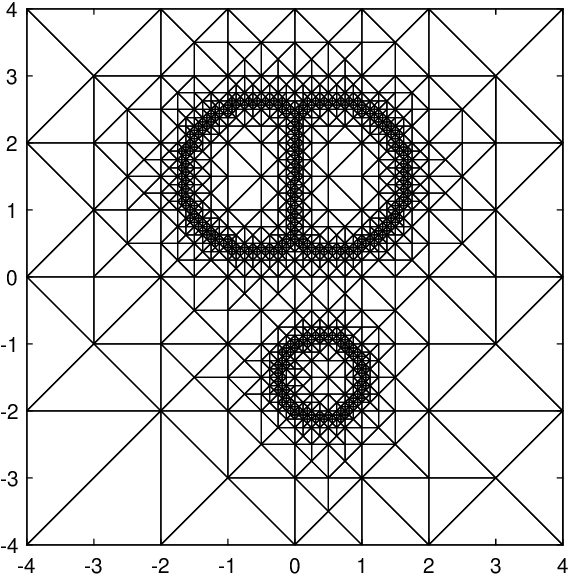}
\includegraphics[angle=-0,width=0.18\textwidth]{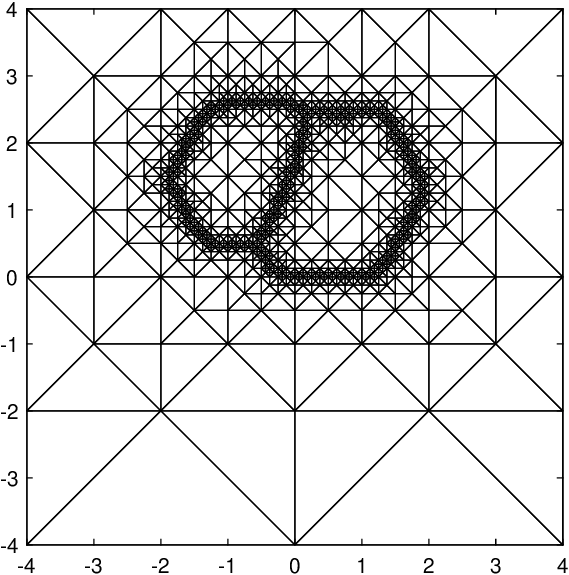}
\caption{The solution at times $t=0, 0.2, 0.4, 1$, and a plot of the discrete 
energy over time. Below we show the adaptive bulk mesh at times $t=0$ and
$t=1$.
}
\label{fig:2dhex_db_plus_one}
\end{figure}%
Repeating the simulation with a bigger initial disk gives the results in
Figure~\ref{fig:2dhex_db_plus_bigone}. Here the radius is $\frac54$, so that the
enclosed area is $4.909$. Now the disk grows at the expense of the right
bubble, so that eventually two separate phases remain.
\begin{figure}
\center
\includegraphics[angle=-90,width=0.18\textwidth]{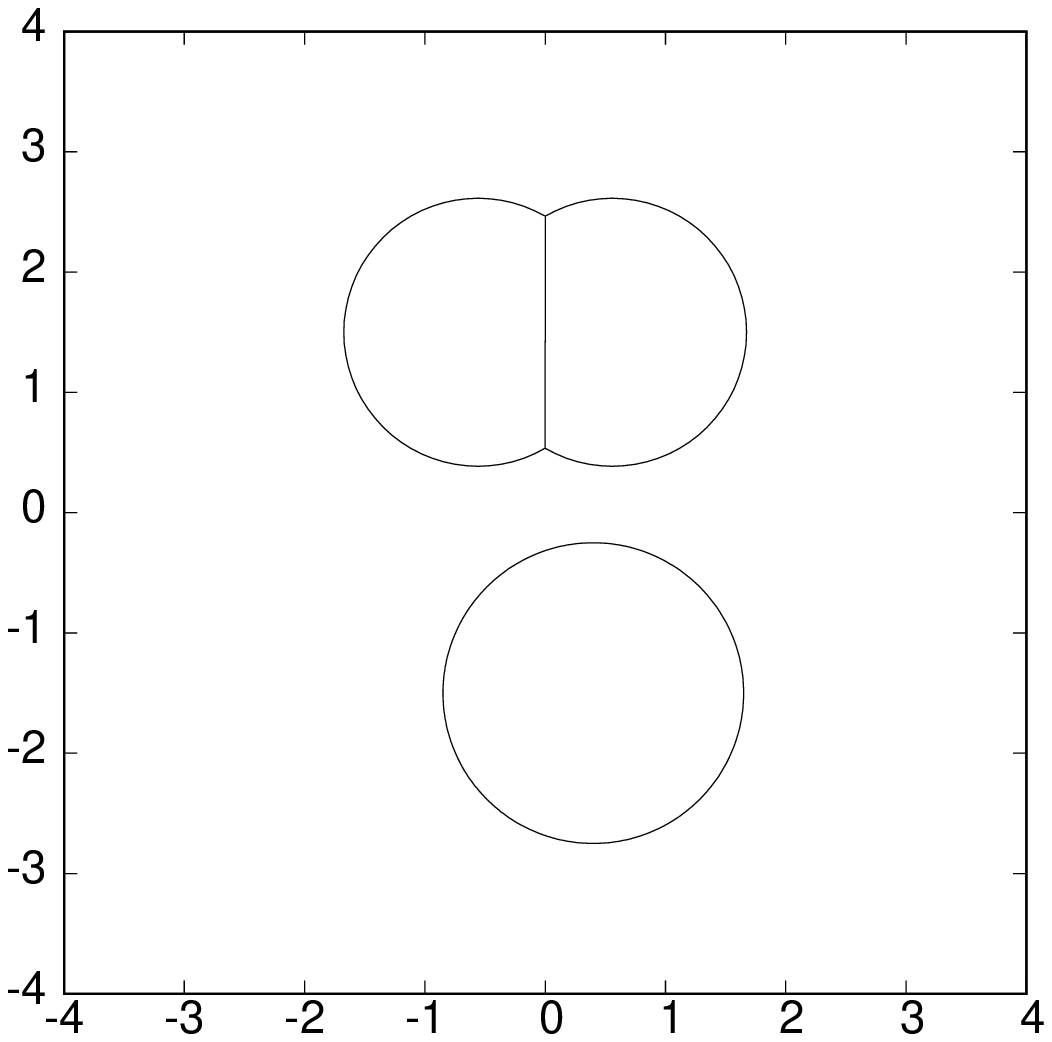}
\includegraphics[angle=-90,width=0.18\textwidth]{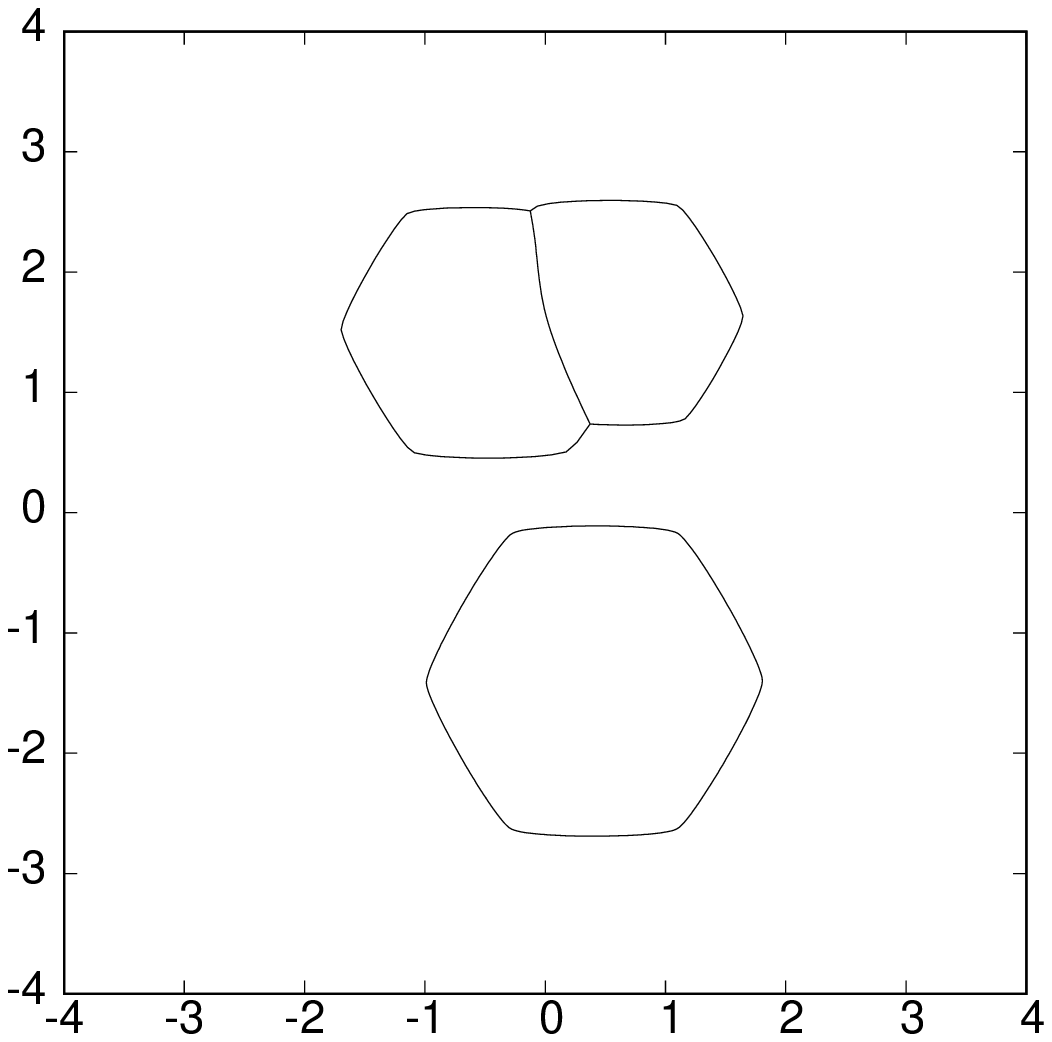}
\includegraphics[angle=-90,width=0.18\textwidth]{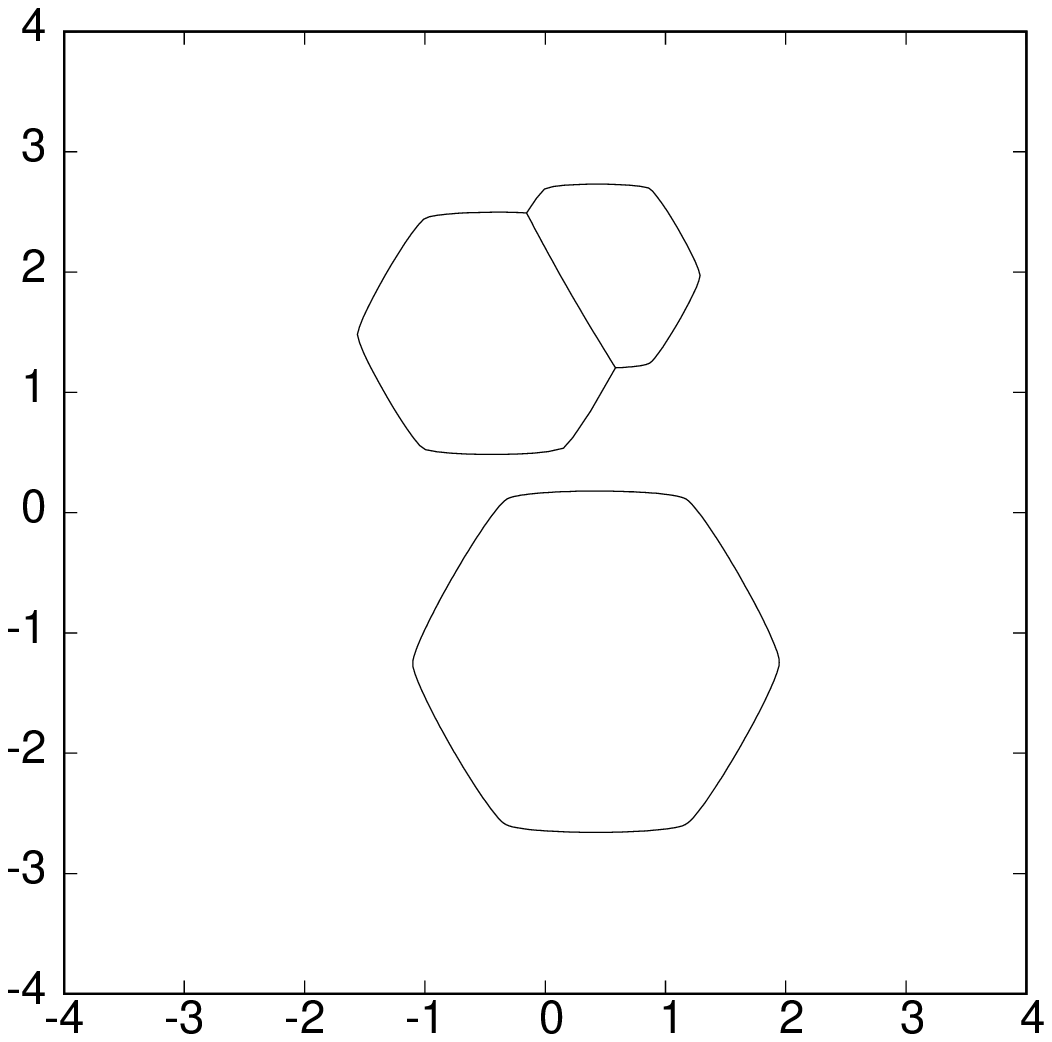}
\includegraphics[angle=-90,width=0.18\textwidth]{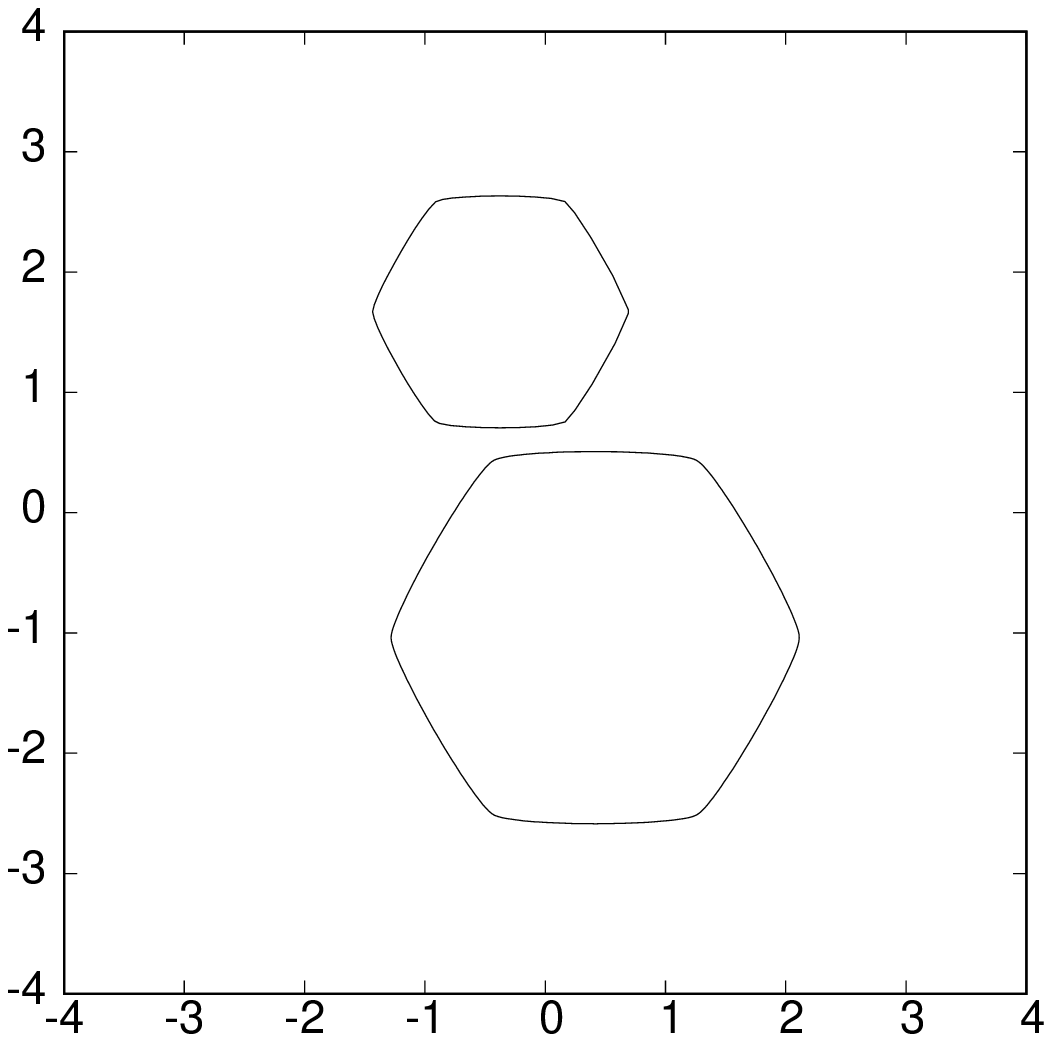}
\includegraphics[angle=-90,width=0.25\textwidth]{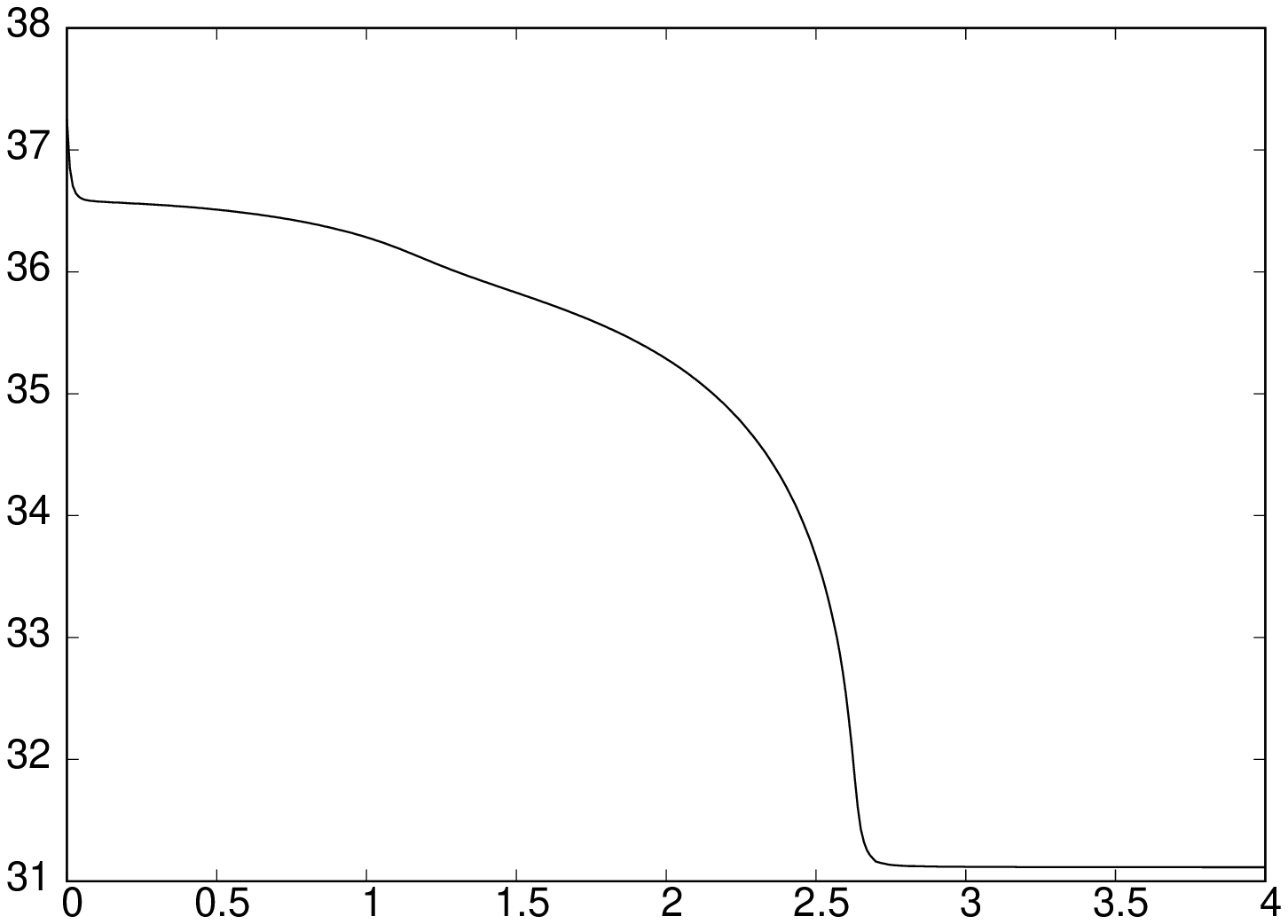}
\caption{The solution at times $t=0, 1, 2, 4$, and a plot of the discrete 
energy over time.
}
\label{fig:2dhex_db_plus_bigone}
\end{figure}%

\vspace{0.3cm}
\noindent
{\bf Example 2}:
We consider the evolution of two double bubbles. In particular,
we have $I_S = 6$, $I_R = 3$, $I_T = 4$, 
$(\curveindex{1}{1},\curveindex{1}{2},\curveindex{1}{3}) = 
(\curveindex{2}{1},\curveindex{2}{2},\curveindex{2}{3}) = (1,2,3)$,
$(\curveindex{3}{1},\curveindex{3}{2},\curveindex{3}{3}) = 
(\curveindex{4}{1},\curveindex{4}{2},\curveindex{4}{3}) = (4,5,6)$
and
\begin{equation*}
\dcmap = \begin{pmatrix} 
0 & -1 & 1 & 0 & -1 & 1\\
1 & 0 & -1 & 1 & 0 & -1\\
-1 & 1 & 0 & -1 & 1 & 0
\end{pmatrix}.
\end{equation*}
The first bubble is chosen with enclosing areas $3.14$ and $6.48$, 
while the second double bubbles encloses two areas of size $3.64$. In each
case, the left bubble is assigned to phase 1, while the right bubbles are
assigned to phase 2. In this way, the lower double bubble holds the larger
portion of phase 1, while the upper double bubble holds the larger portion of
phase 2. Consequently, each double bubble evolves to a single disk that
contains just one phase. See Figure~\ref{fig:2dhex_db_times2}.
\begin{figure}
\center
\includegraphics[angle=-90,width=0.18\textwidth]{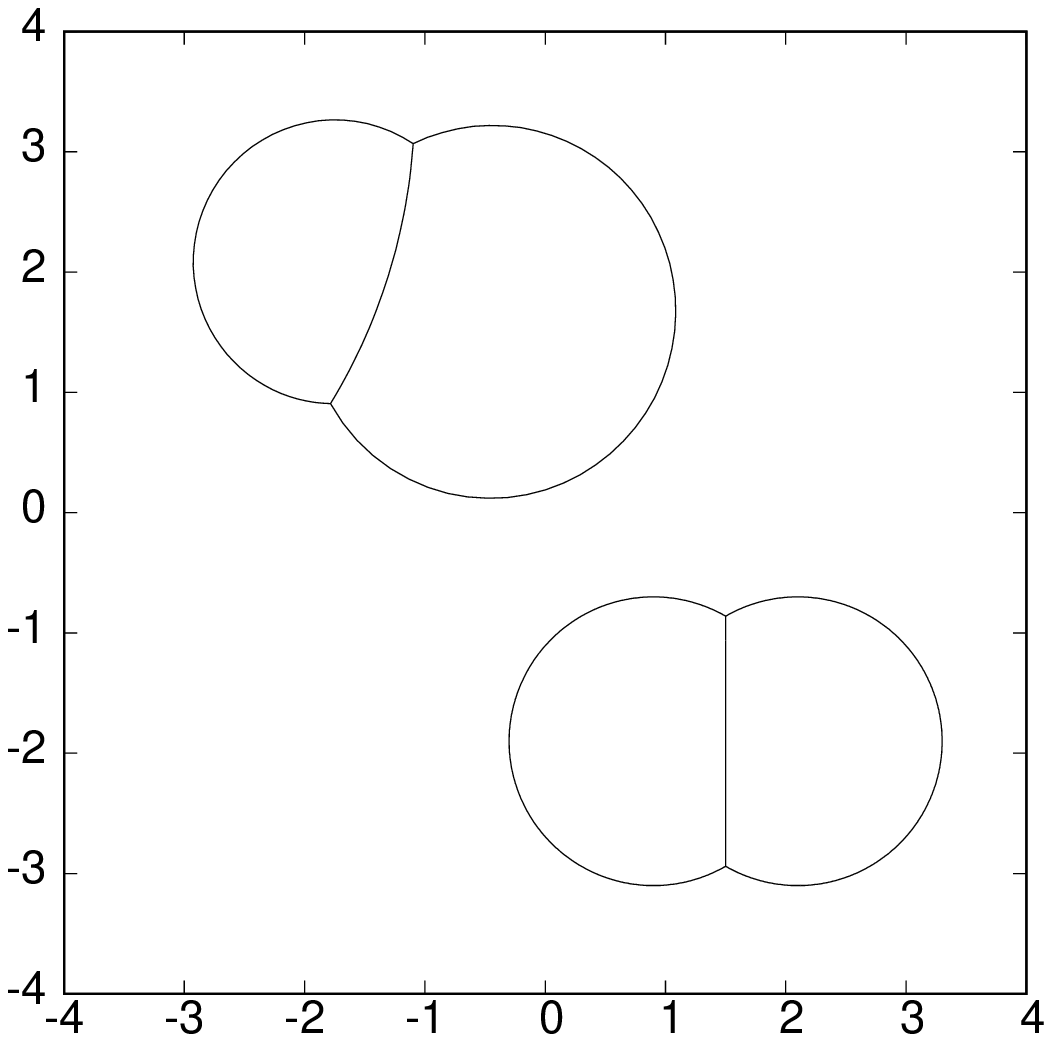}
\includegraphics[angle=-90,width=0.18\textwidth]{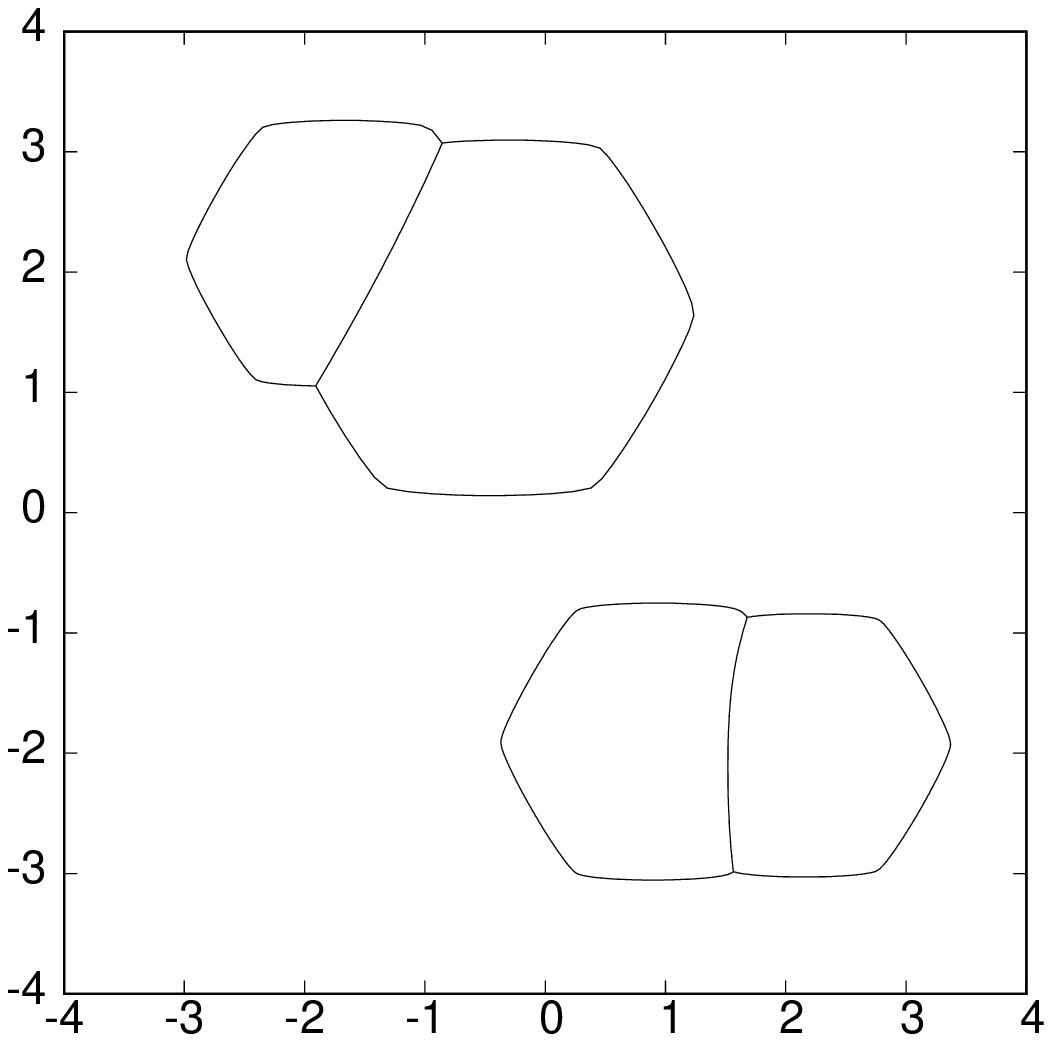}
\includegraphics[angle=-90,width=0.18\textwidth]{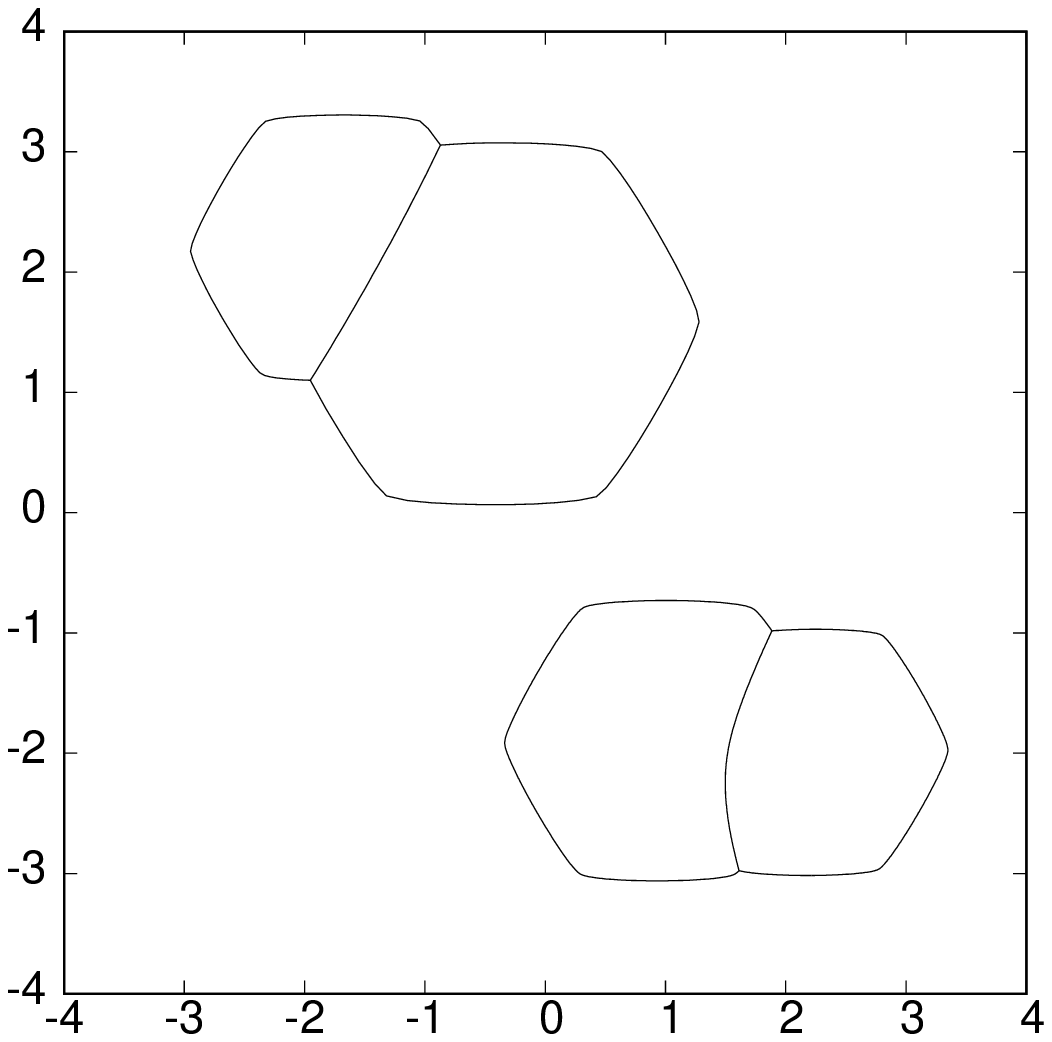}
\includegraphics[angle=-90,width=0.18\textwidth]{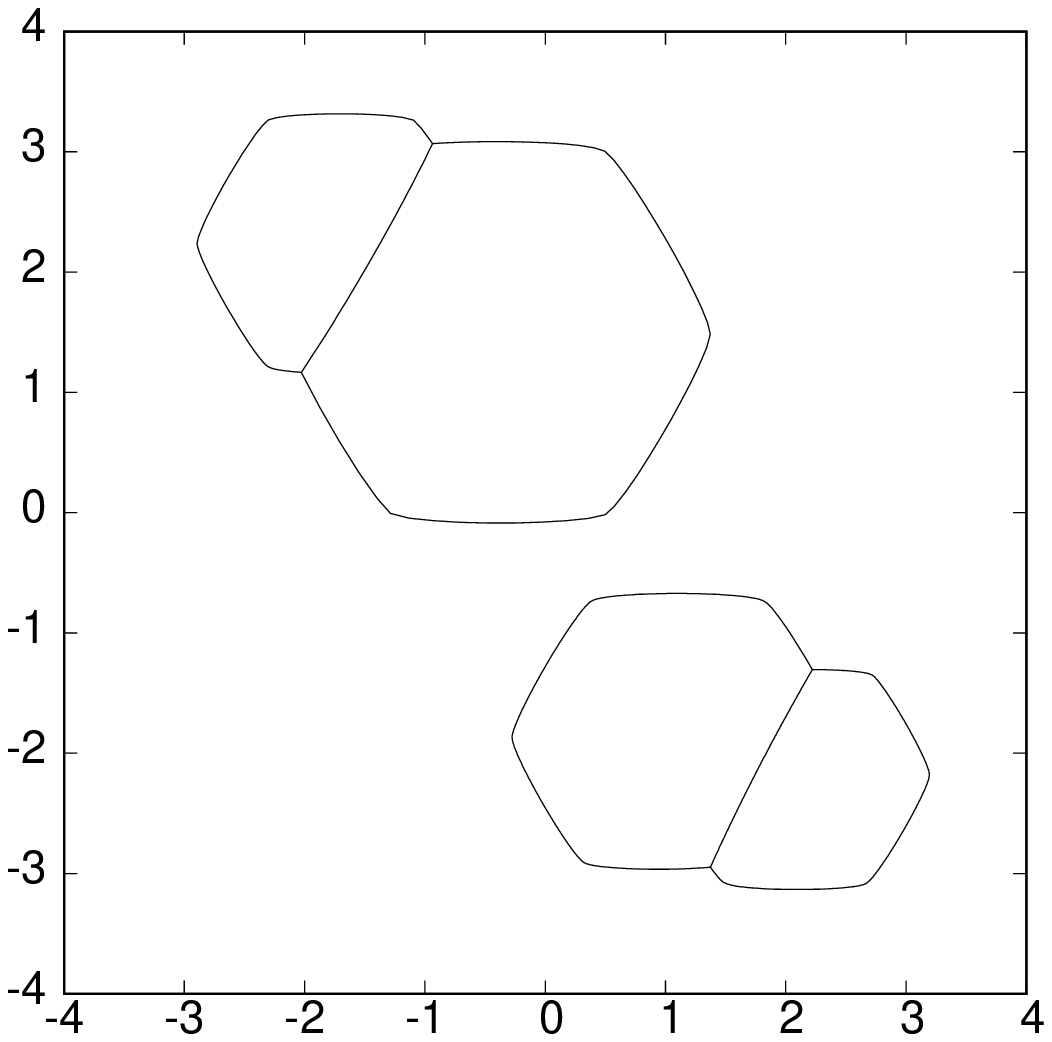}
\includegraphics[angle=-90,width=0.18\textwidth]{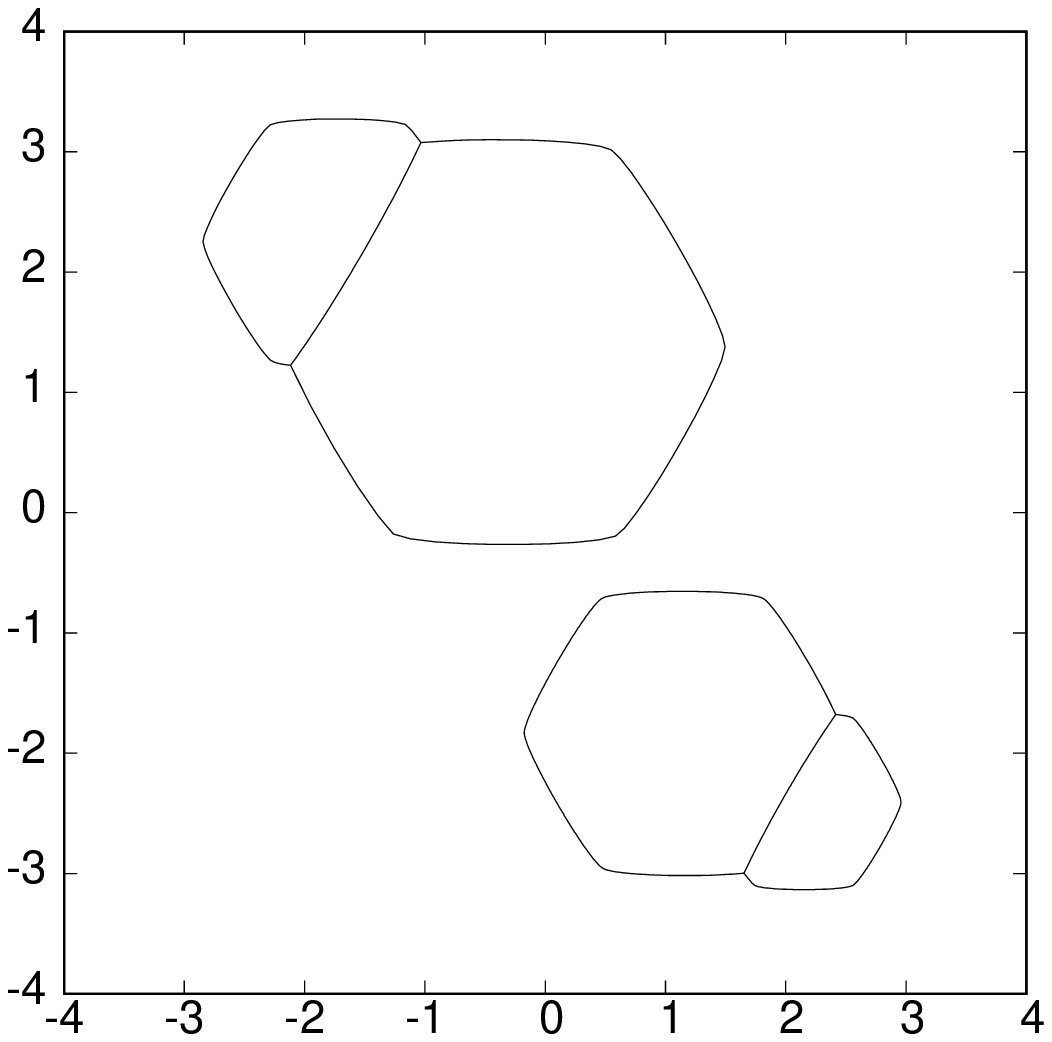}
\includegraphics[angle=-90,width=0.18\textwidth]{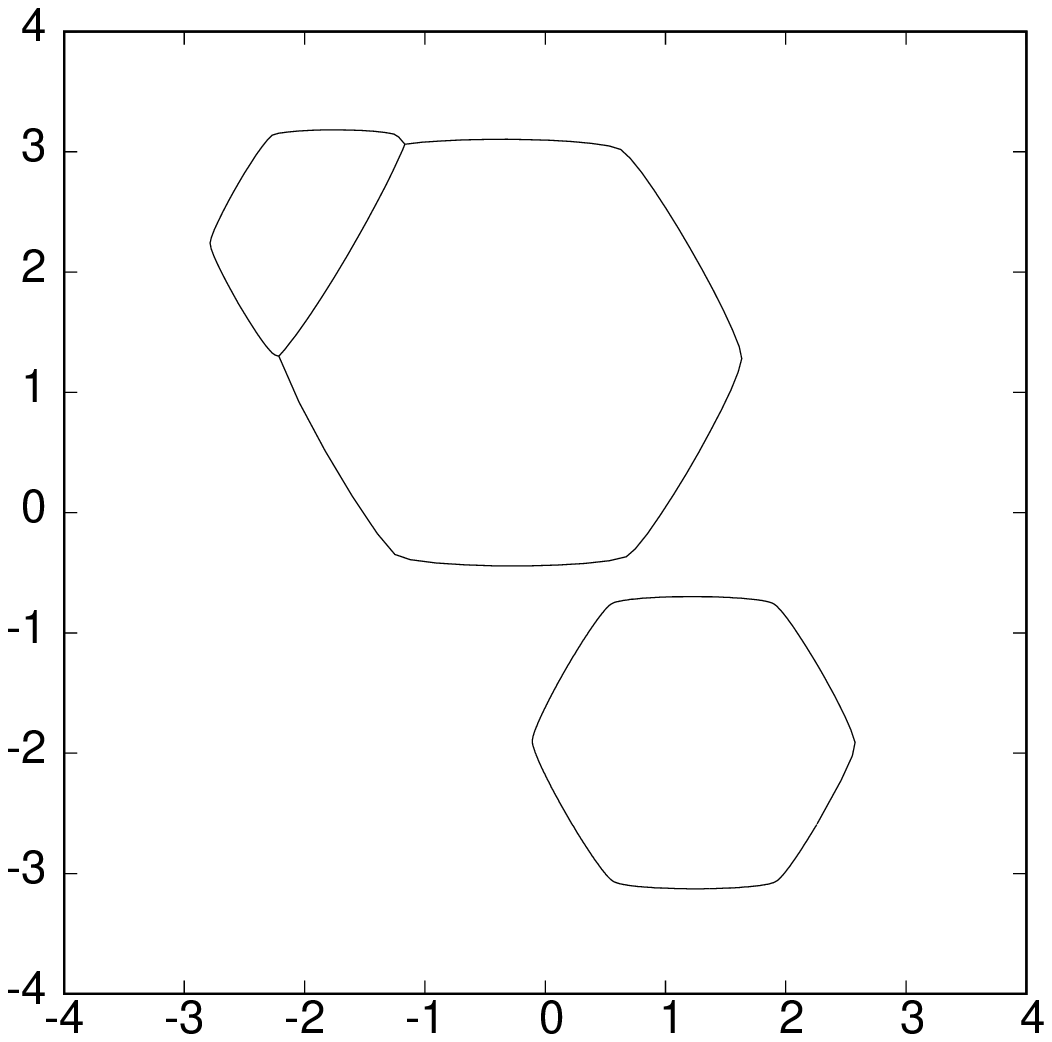}
\includegraphics[angle=-90,width=0.18\textwidth]{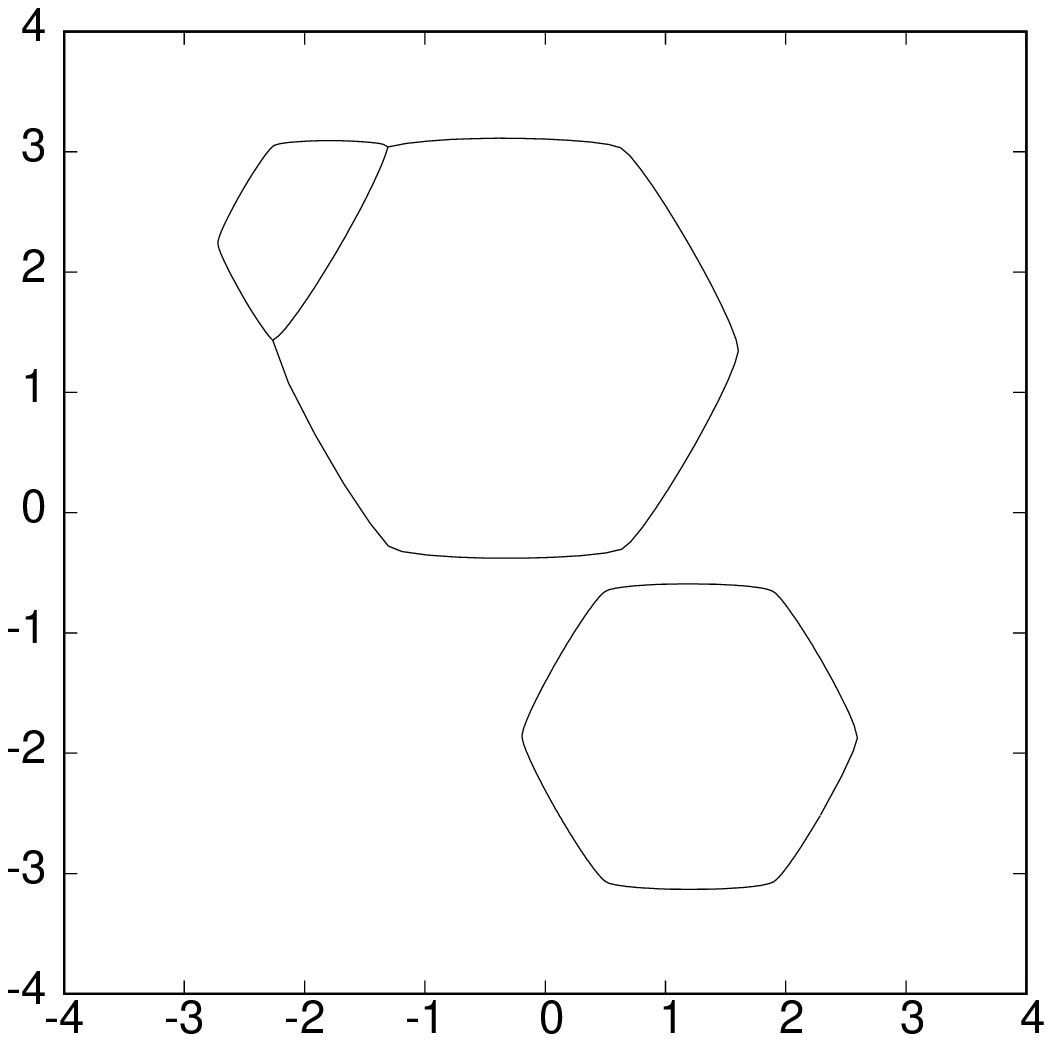}
\includegraphics[angle=-90,width=0.18\textwidth]{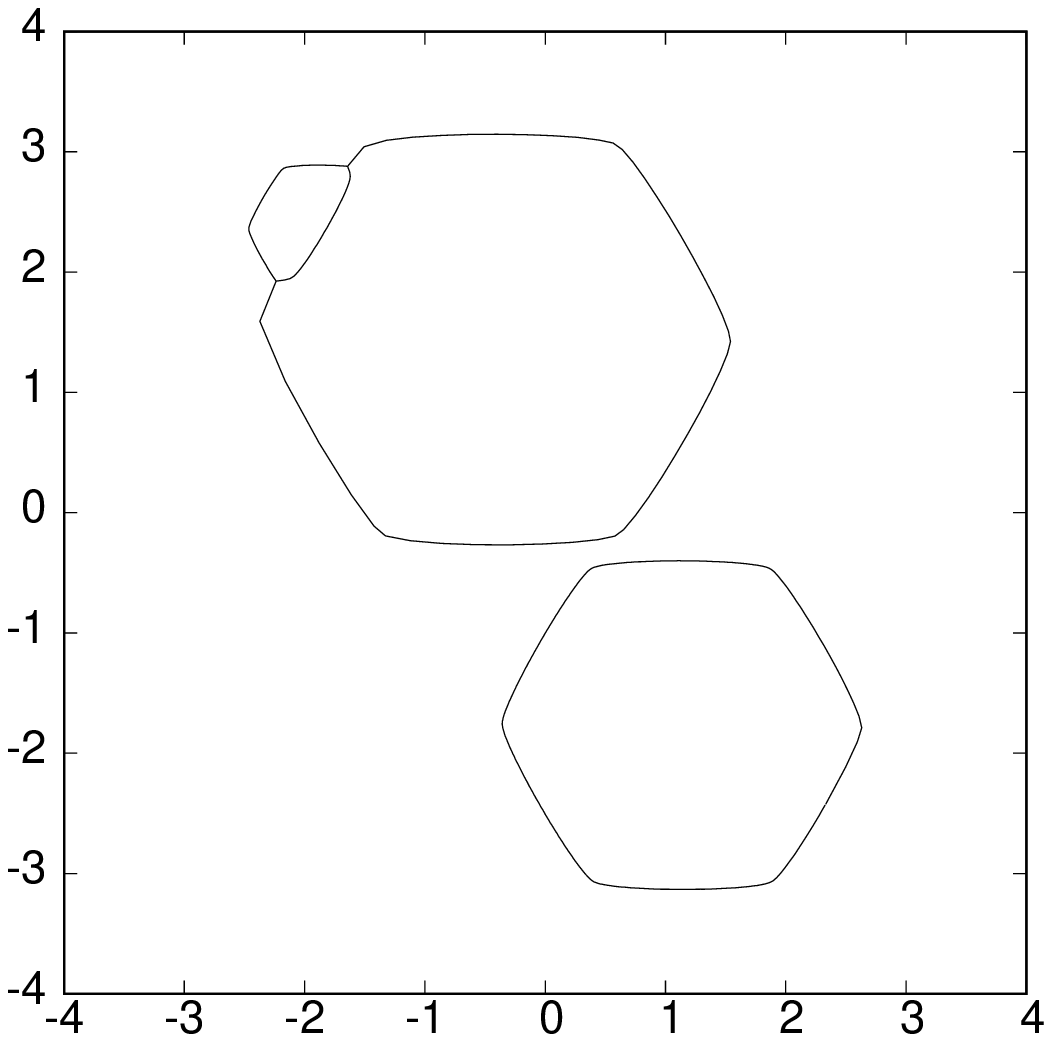}
\includegraphics[angle=-90,width=0.18\textwidth]{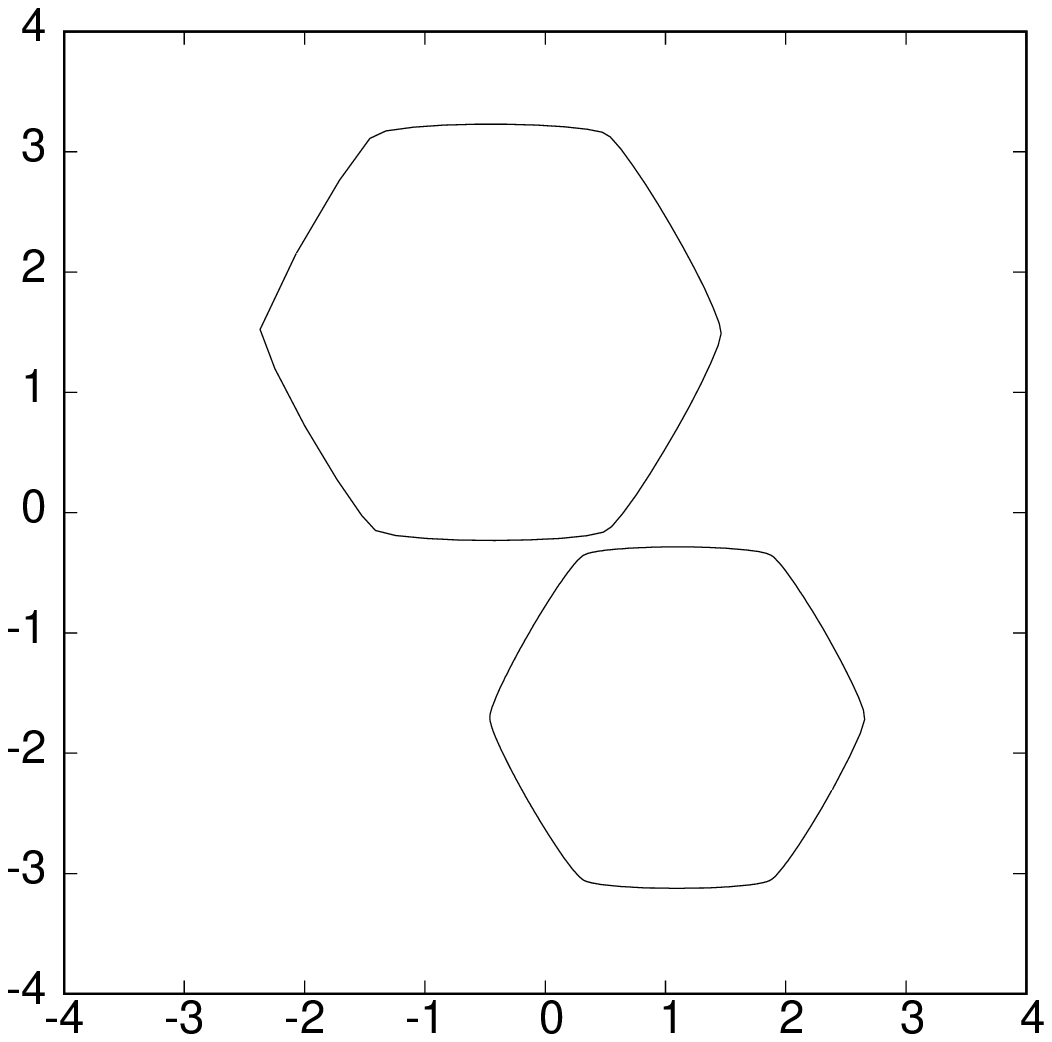}
\includegraphics[angle=-90,width=0.25\textwidth]{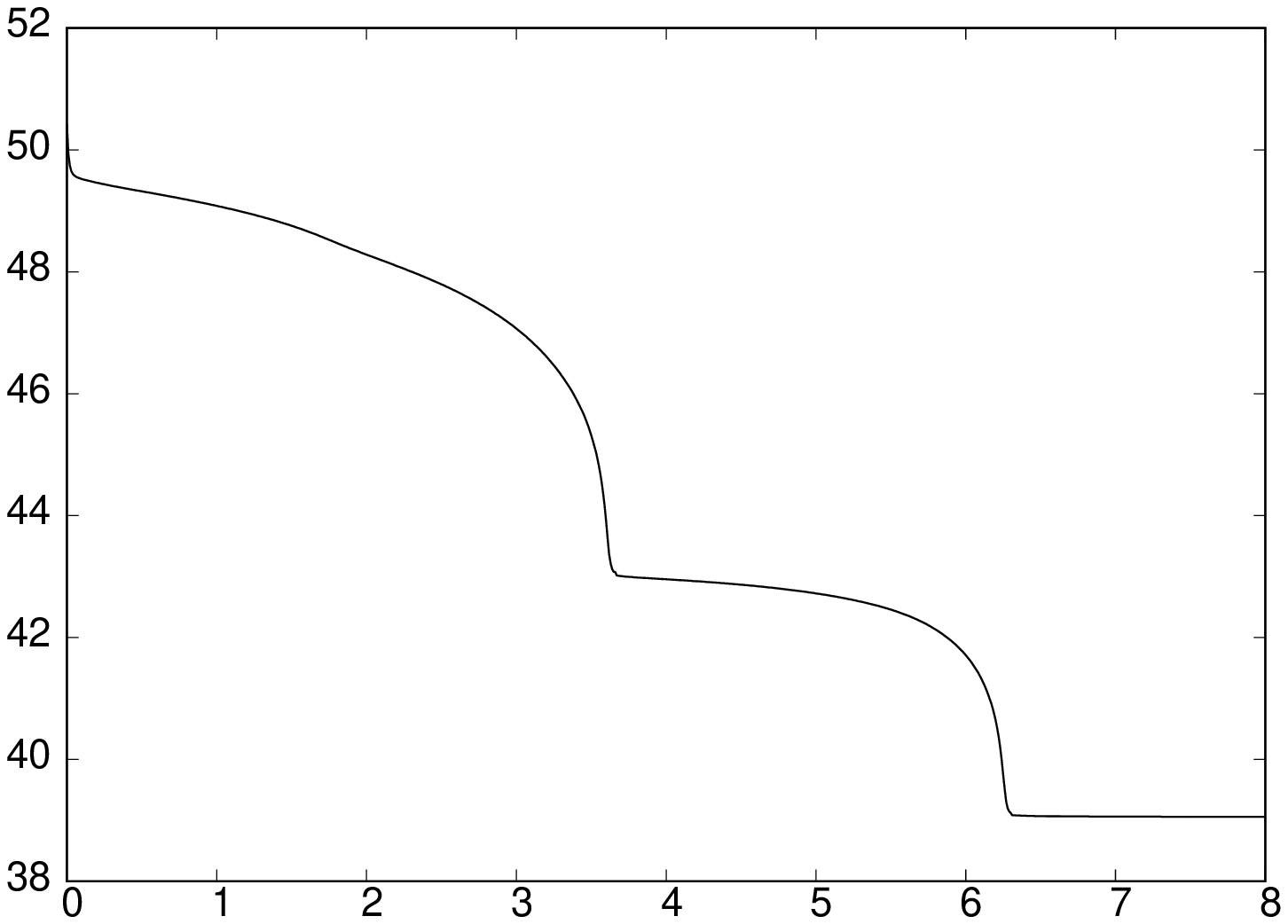}
\caption{The solution at times $t=0, 0.5, 1, 2, 3, 4, 5, 6, 7, 8$, 
and a plot of the discrete energy over time.
}
\label{fig:2dhex_db_times2}
\end{figure}%

\vspace{0.3cm}
\noindent
{\bf Example 3}:
As an example for non-equal surface energy densities for the various curves, we
repeat the simulation from Example~1 in 
Figure~\ref{fig:2dhex_db_plus_one}, but now choose
$(\gamma_1,\gamma_2,\gamma_3,\gamma_4) = 
(2\gamma_{\rm hex}, \gamma_{\rm hex}, 2\gamma_{\rm hex}, \gamma_{\rm hex})$.
That is, the curves 1 and 3 in the double bubble have twice the surface energy
densities of the curves 2 and 4. This now means that in contrast to
Figure~\ref{fig:2dhex_db_plus_one}, it makes energetically more sense to increase
the size of the single bubble, while shrinking the bubble that is surrounded by
the more expensive interfaces. See Figure~\ref{fig:2dhex_2121_db_plus_one} 
for the observed evolution.
\begin{figure}
\center
\includegraphics[angle=-90,width=0.18\textwidth]{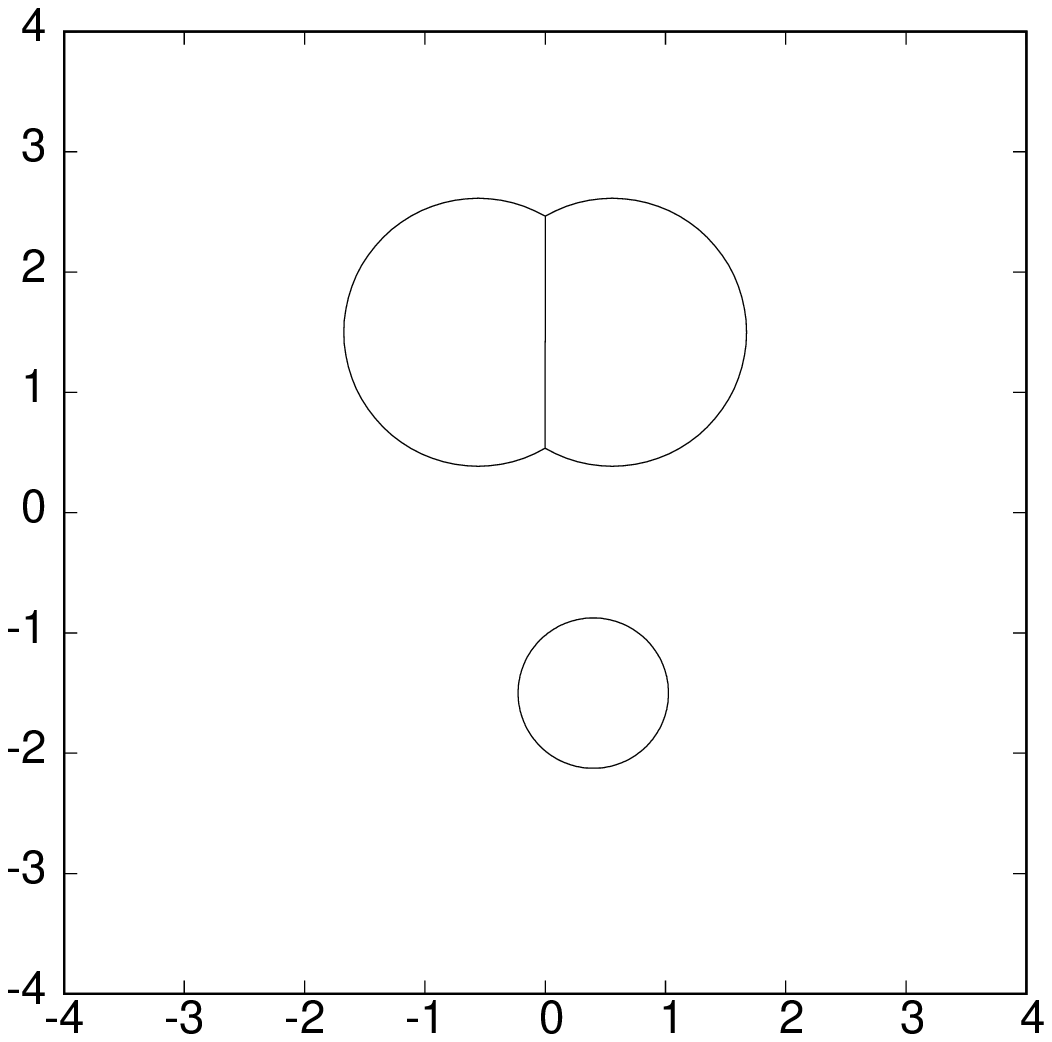}
\includegraphics[angle=-90,width=0.18\textwidth]{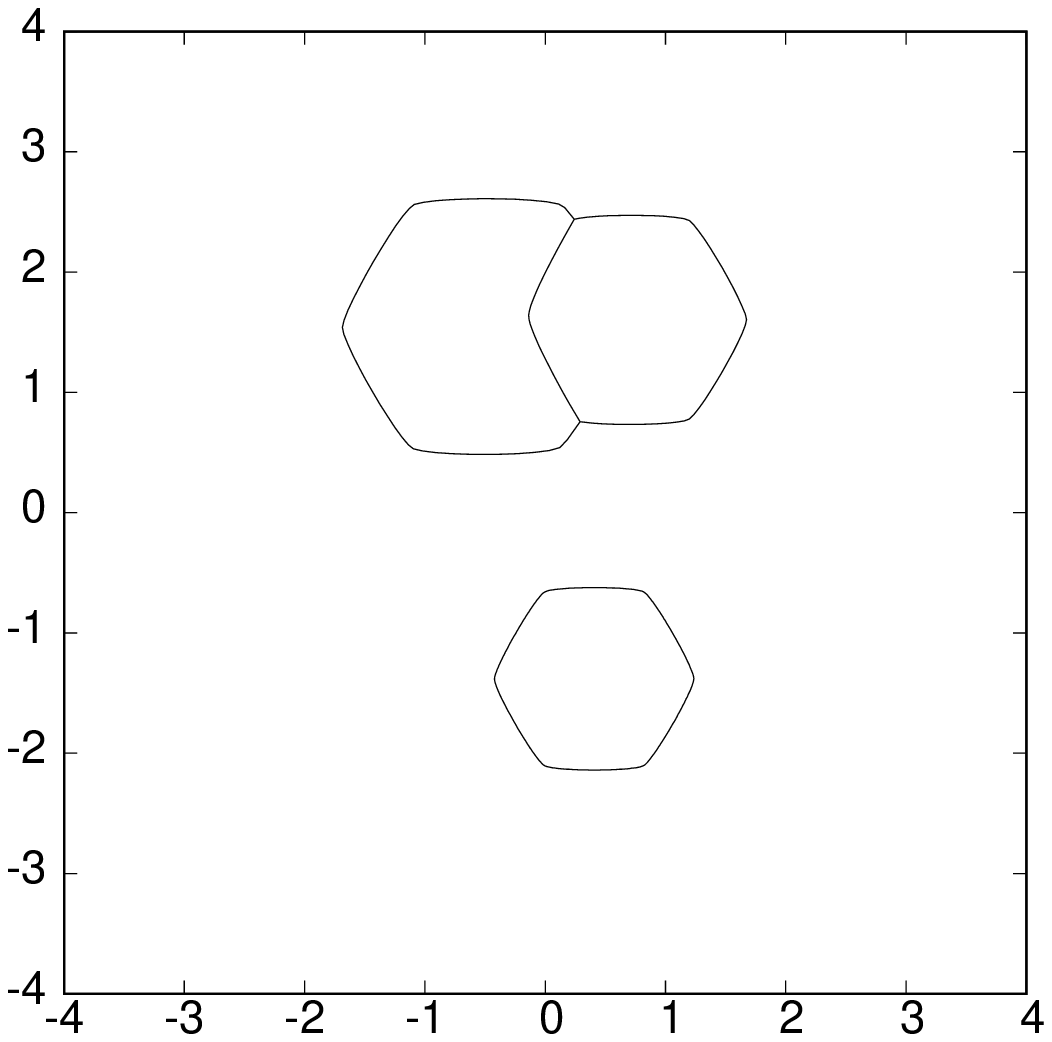}
\includegraphics[angle=-90,width=0.18\textwidth]{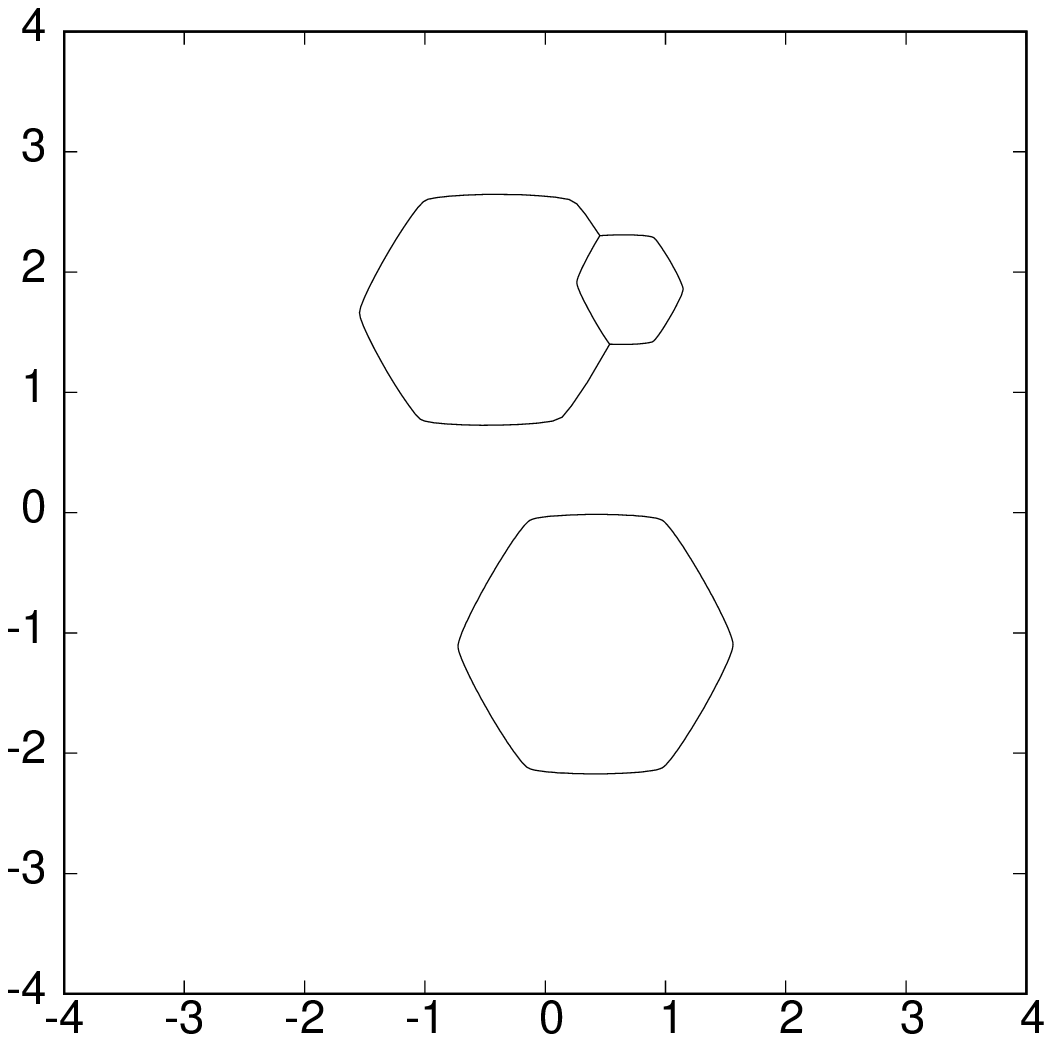}
\includegraphics[angle=-90,width=0.18\textwidth]{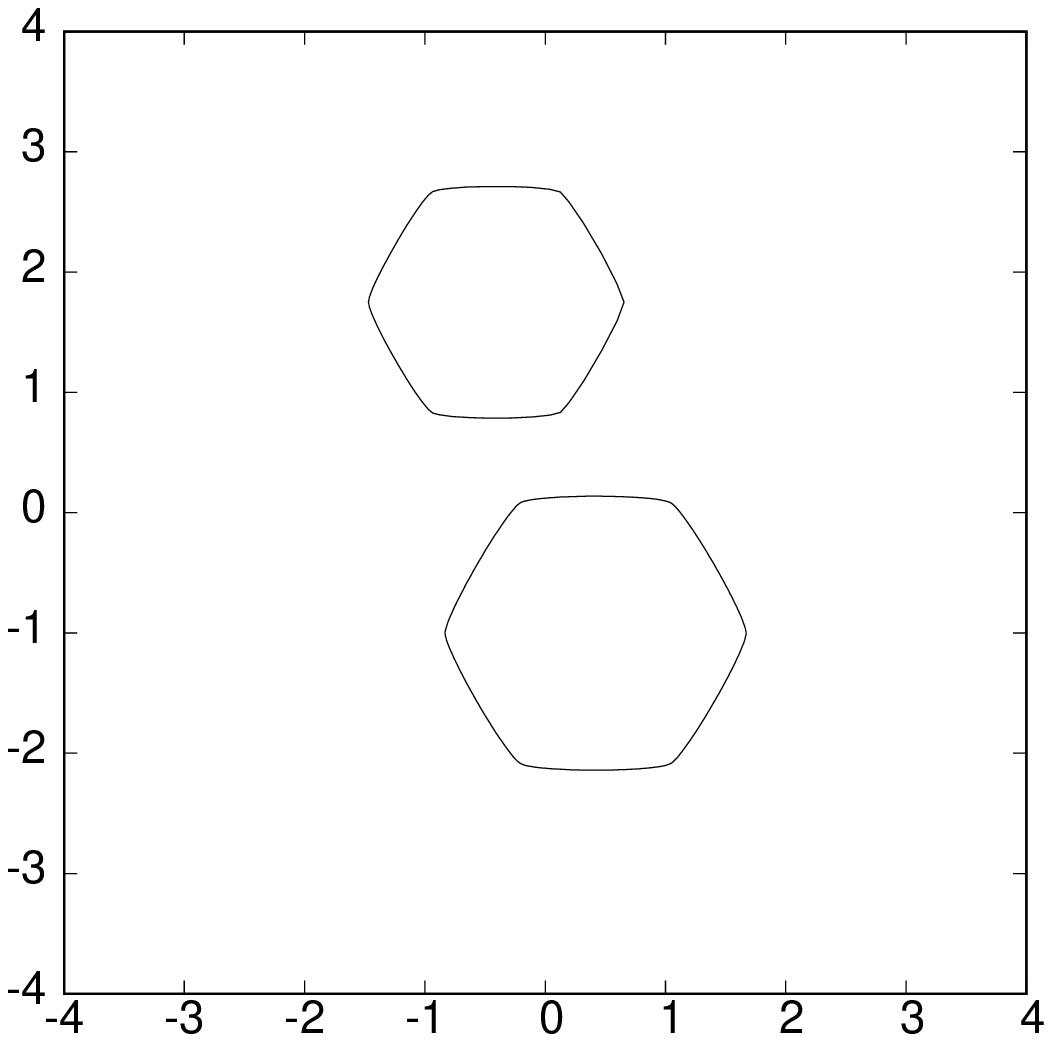}
\includegraphics[angle=-90,width=0.25\textwidth]{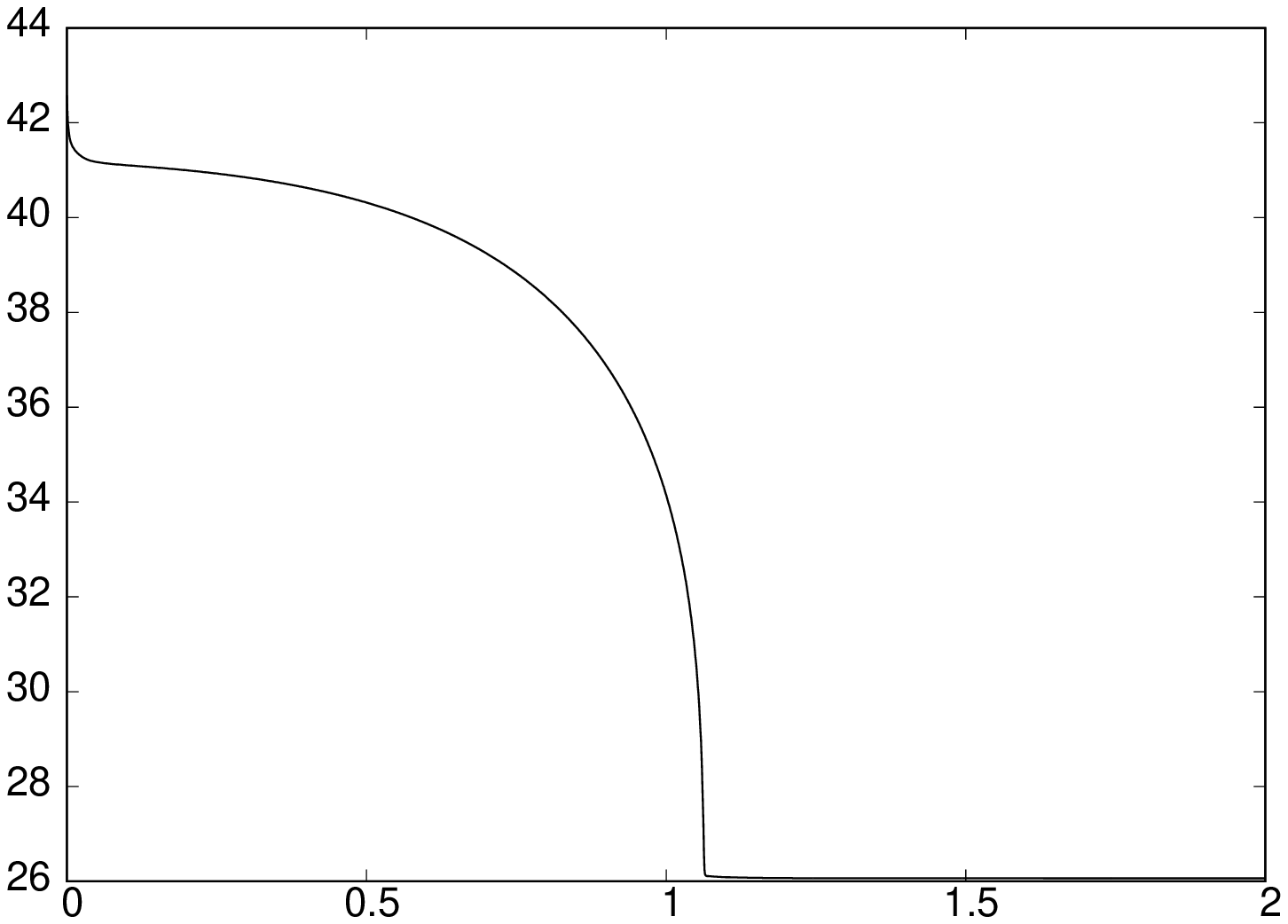}
\caption{The solution at times $t=0, 0.5, 1, 2$, and a plot of the discrete 
energy over time.
}
\label{fig:2dhex_2121_db_plus_one}
\end{figure}%

\subsection{Simulations in 2D with undercooling on $\pOmega$}

In this subsection we consider some numerical simulations for $d=2$ and
$\pOmega_D \not= \emptyset$.  
For the anisotropy we define $\gamma_{\rm hex}$ as in \eqref{eq:hex2d} but now 
with $\delta=0.01$. This leads to sharper corners in the Wulff shape.

\vspace{0.3cm}
\noindent
{\bf Example 4}:
On the boundary $\pOmega=\pOmega_D$ we choose the undercooling parameters
$\bv w_D = (20, 10, -30)^\top$, and start with a very small seed consisting of a
standard double bubble. 
In fact, the two bubbles of the double bubble enclose an 
area of about $0.031$ each. We also let $\rho=1$.
Moreover, we have $I_S = 3$, $I_R = 3$, $I_T = 2$, 
$(\curveindex{1}{1},\curveindex{1}{2},\curveindex{1}{3}) = (\curveindex{2}{1},\curveindex{2}{2},\curveindex{2}{3}) = (1,2,3)$
and
$\dcmap = \begin{pmatrix}0 &-1&1 \\ 1 & 0 &-1 \\ -1 &1&0 \end{pmatrix}$.
The evolution is shown in Figure~\ref{fig:2dDbc21hex_db_rho_small}.
\begin{figure}
\center
\includegraphics[angle=-90,width=0.18\textwidth]{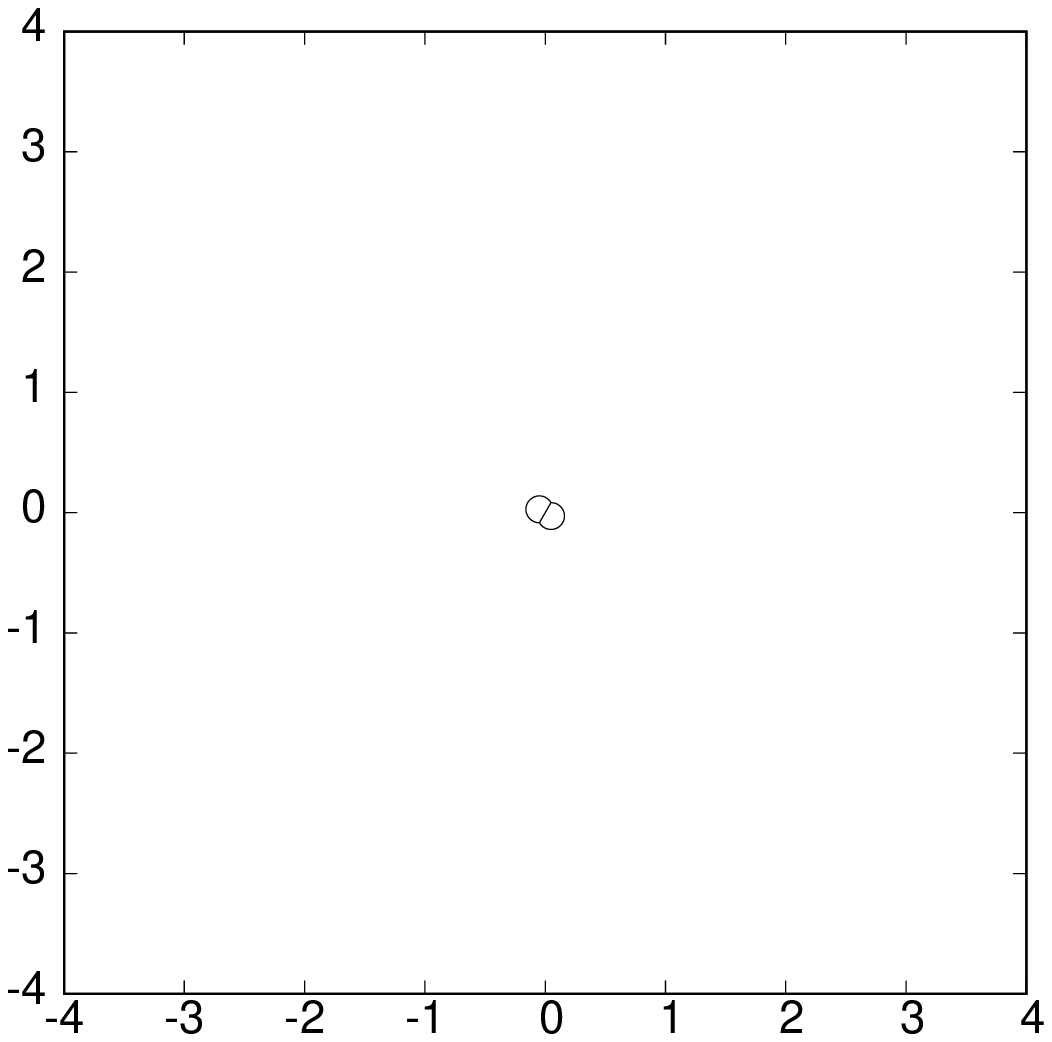}
\includegraphics[angle=-90,width=0.18\textwidth]{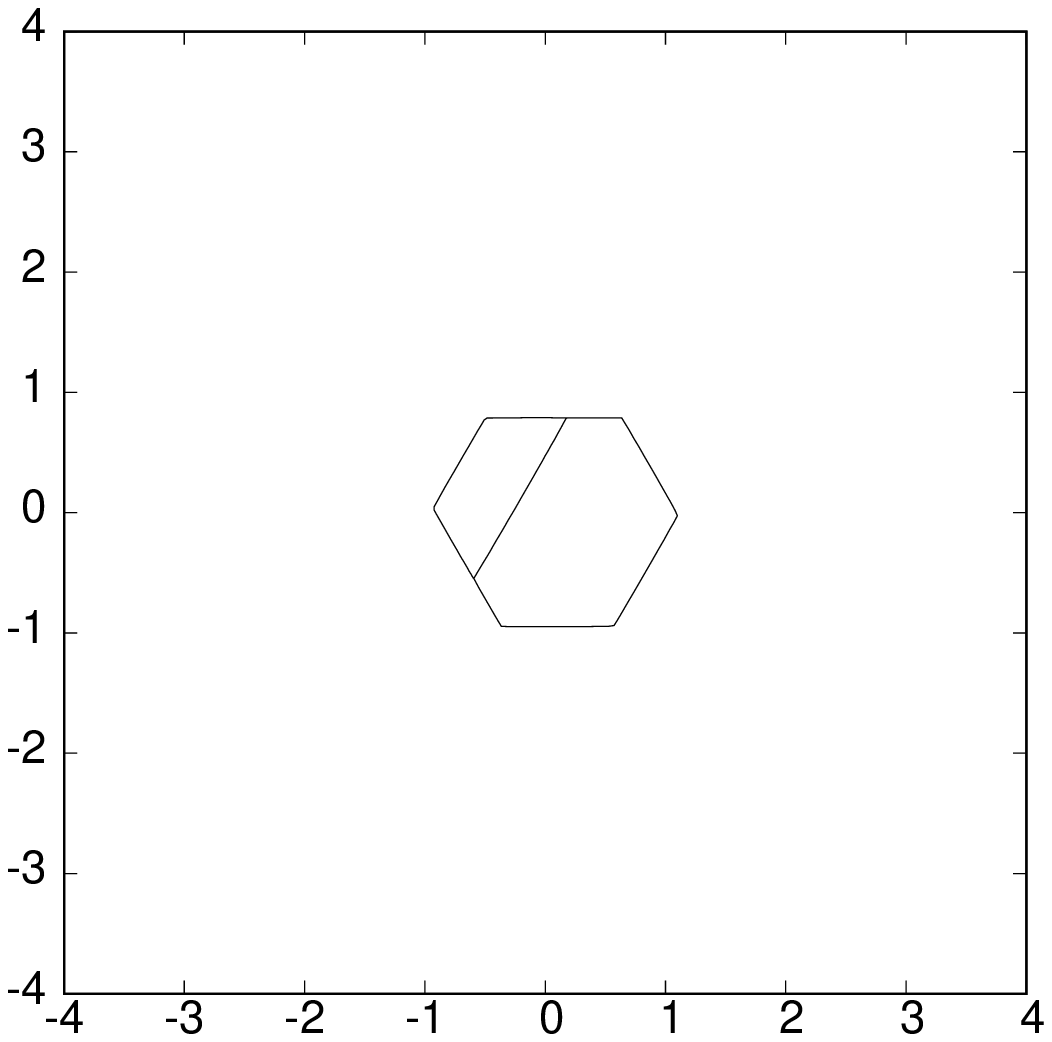}
\includegraphics[angle=-90,width=0.18\textwidth]{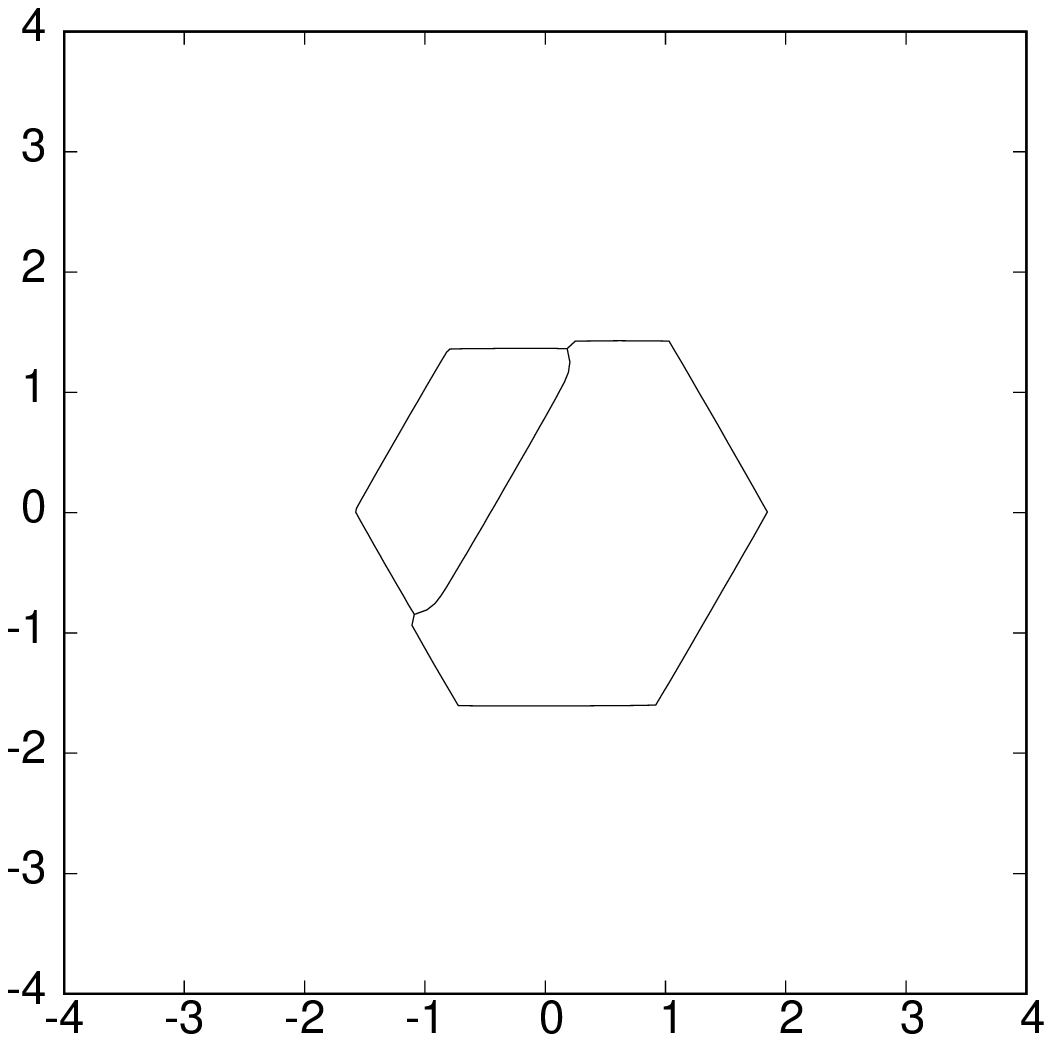}
\includegraphics[angle=-90,width=0.18\textwidth]{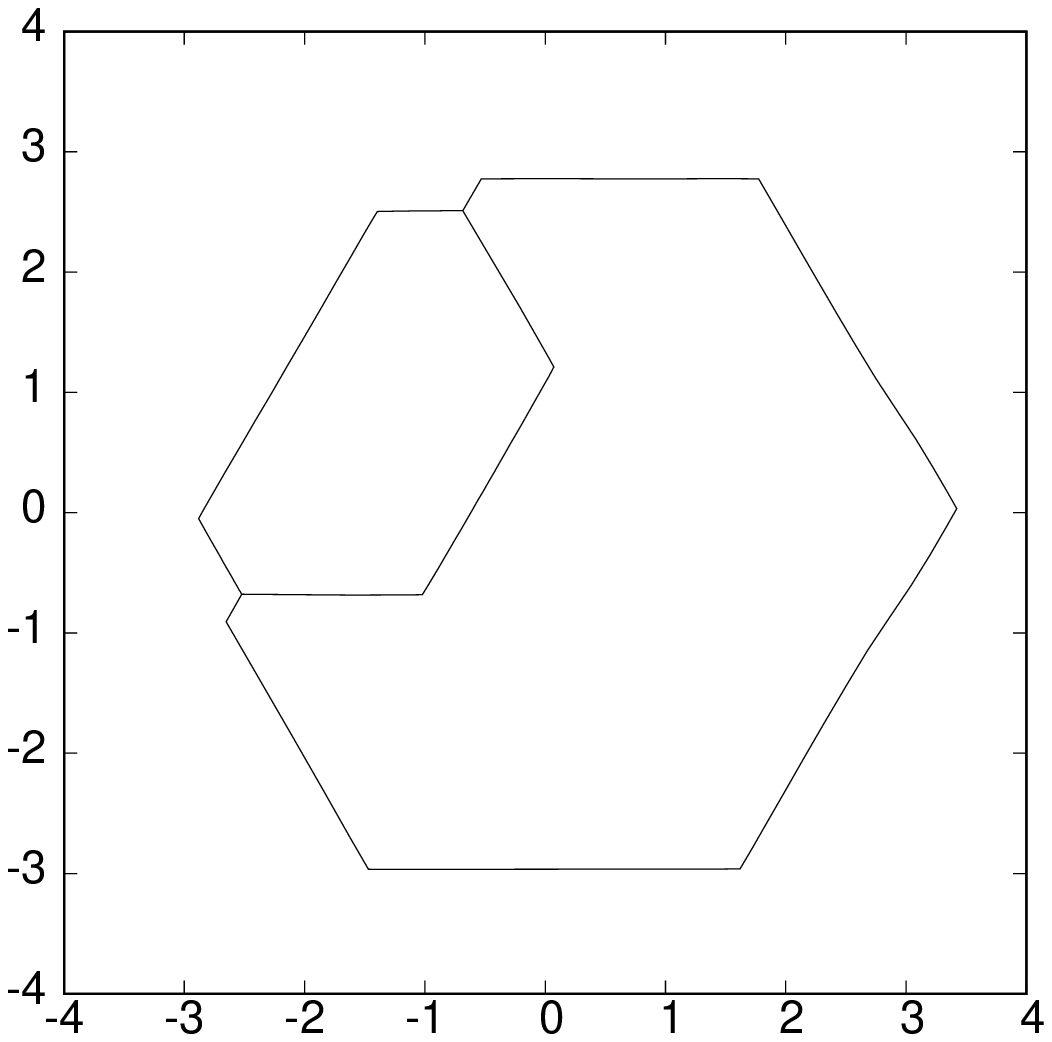}
\includegraphics[angle=-90,width=0.25\textwidth]{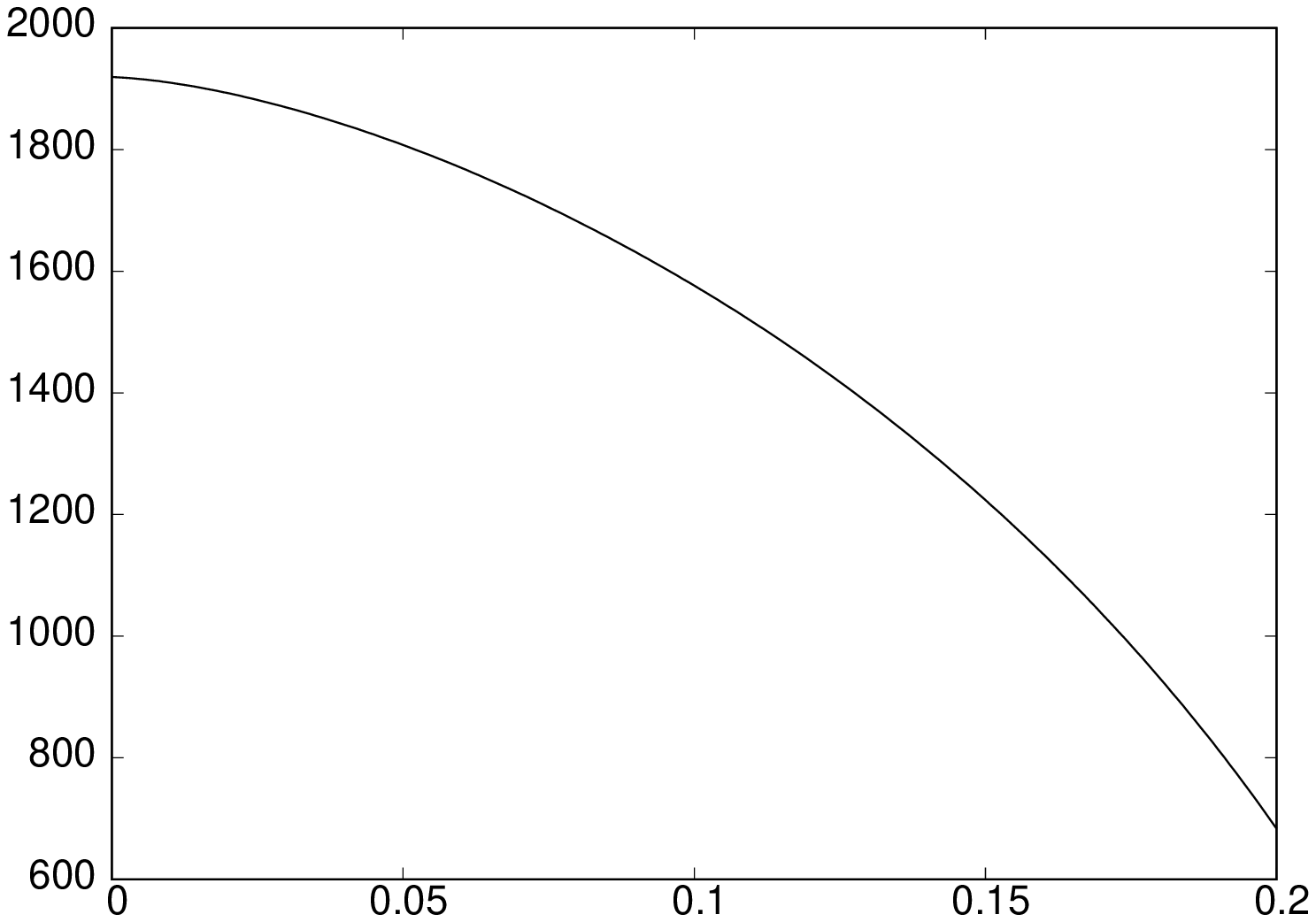}
\caption{($\bv w_D = (20, 10, -30)^\top$, $\rho=1$)
The solution at times $t=0, 0.05, 0.1, 0.2$, and a plot of the discrete 
energy over time.
}
\label{fig:2dDbc21hex_db_rho_small}
\end{figure}%

\vspace{0.3cm}
\noindent
{\bf Example 5}:
In order to provoke some more unstable growths, we use the initial data from
Example~4 and now set $\bv w_D = (12, 11, -23)^\top$,
$\rho=0.05$, as well as $(\gamma_1,\gamma_2,\gamma_3) = \alpha
(\gamma_{\rm hex},\gamma_{\rm hex},\gamma_{\rm hex})$ where $\alpha=0.05$.
The evolution is shown in
Figure~\ref{fig:2dDbc1211hex_db_rho_005}.
\begin{figure}
\center
\includegraphics[angle=-90,width=0.18\textwidth]{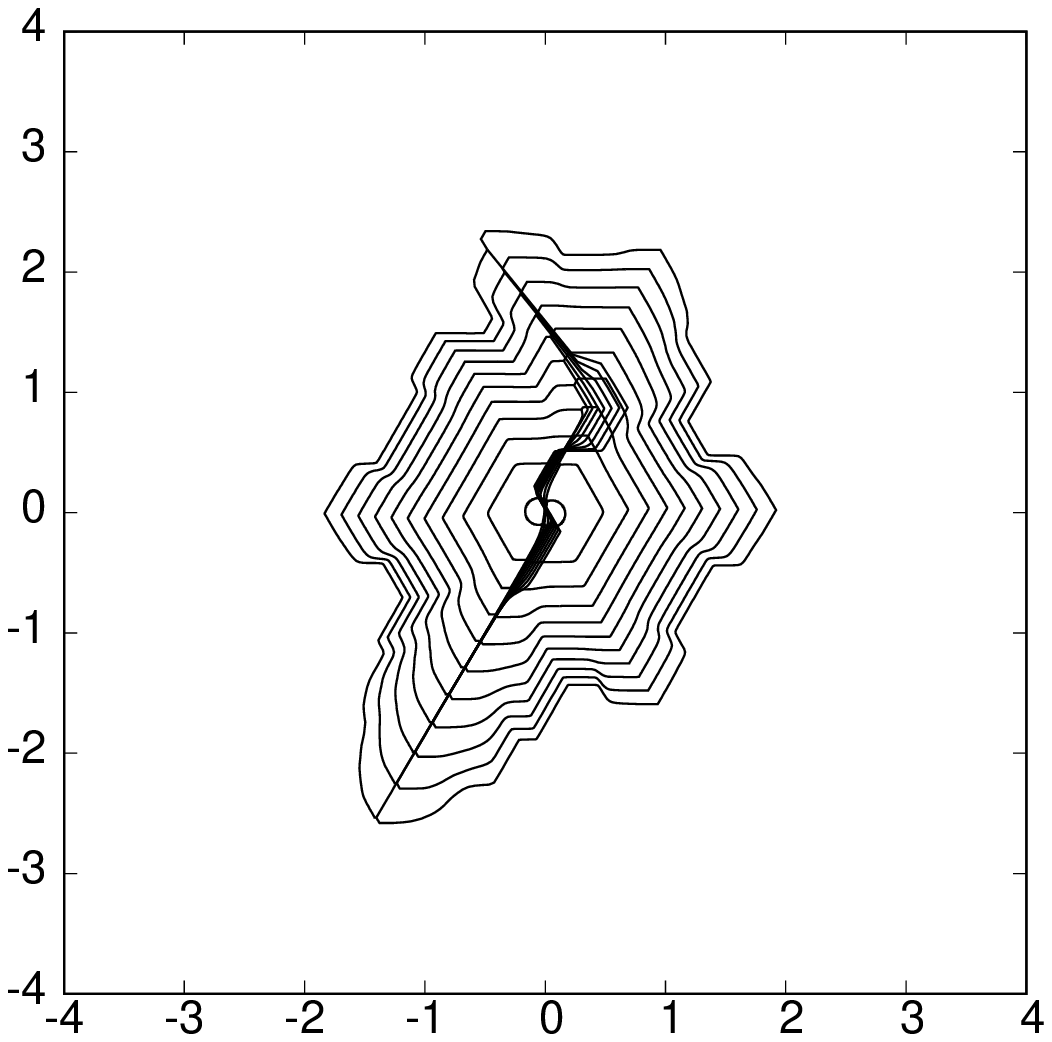}
\includegraphics[angle=-90,width=0.18\textwidth]{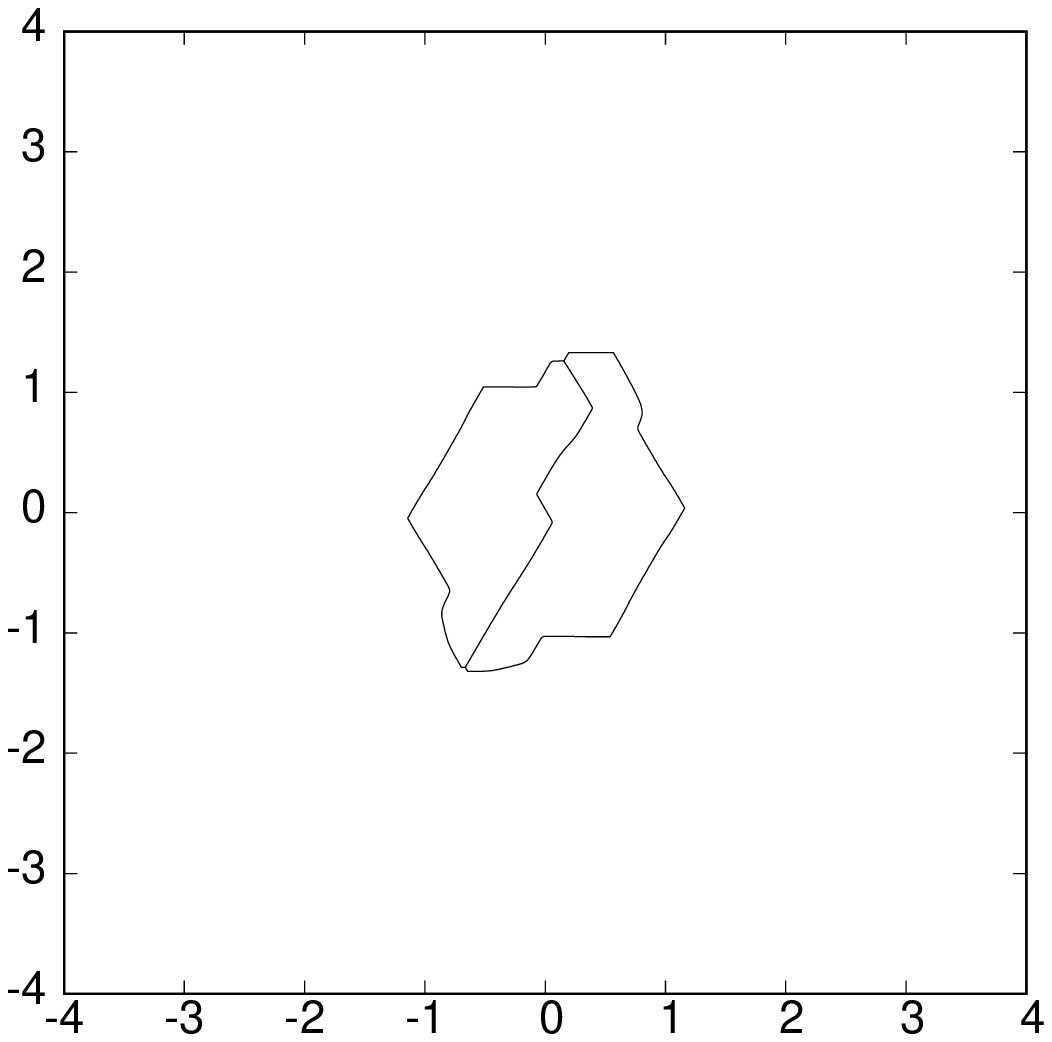}
\includegraphics[angle=-90,width=0.18\textwidth]{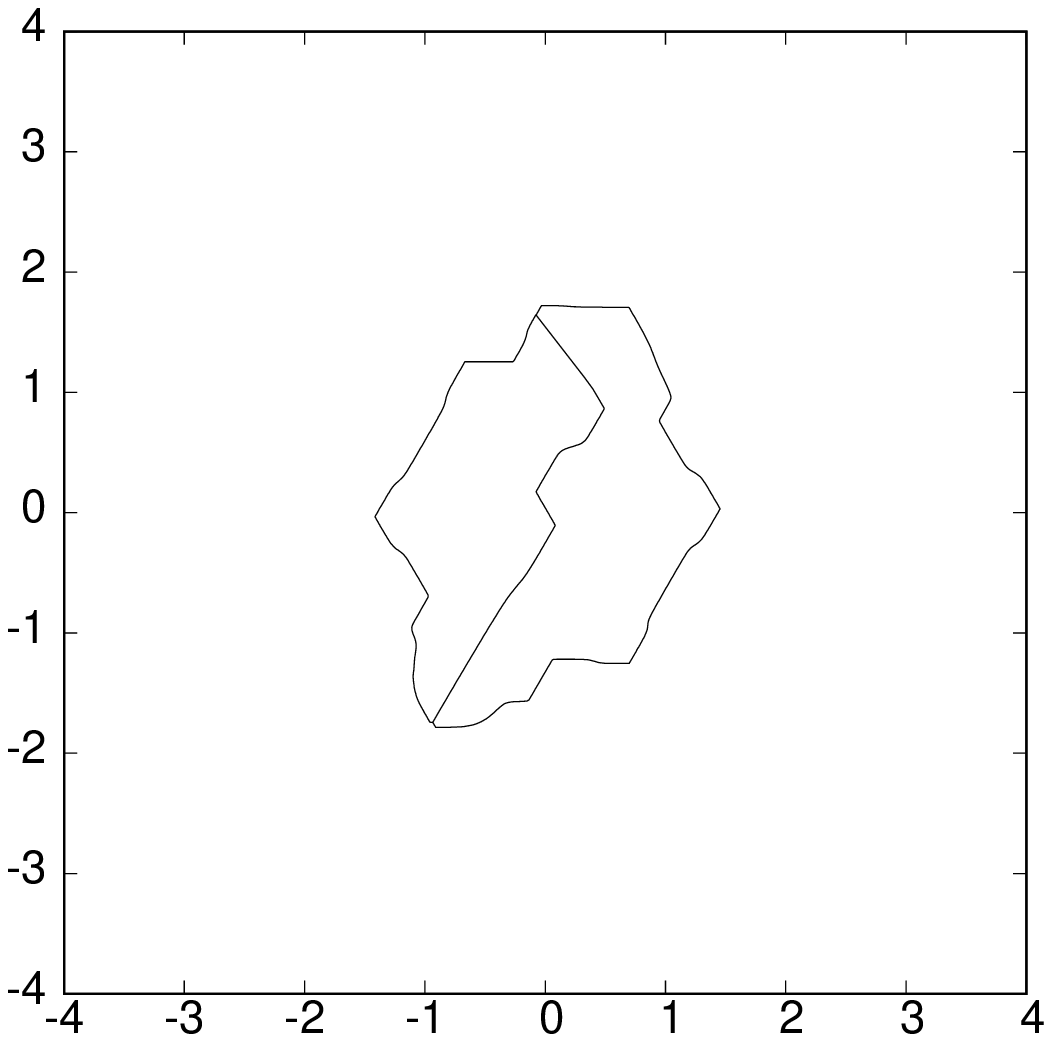}
\includegraphics[angle=-90,width=0.18\textwidth]{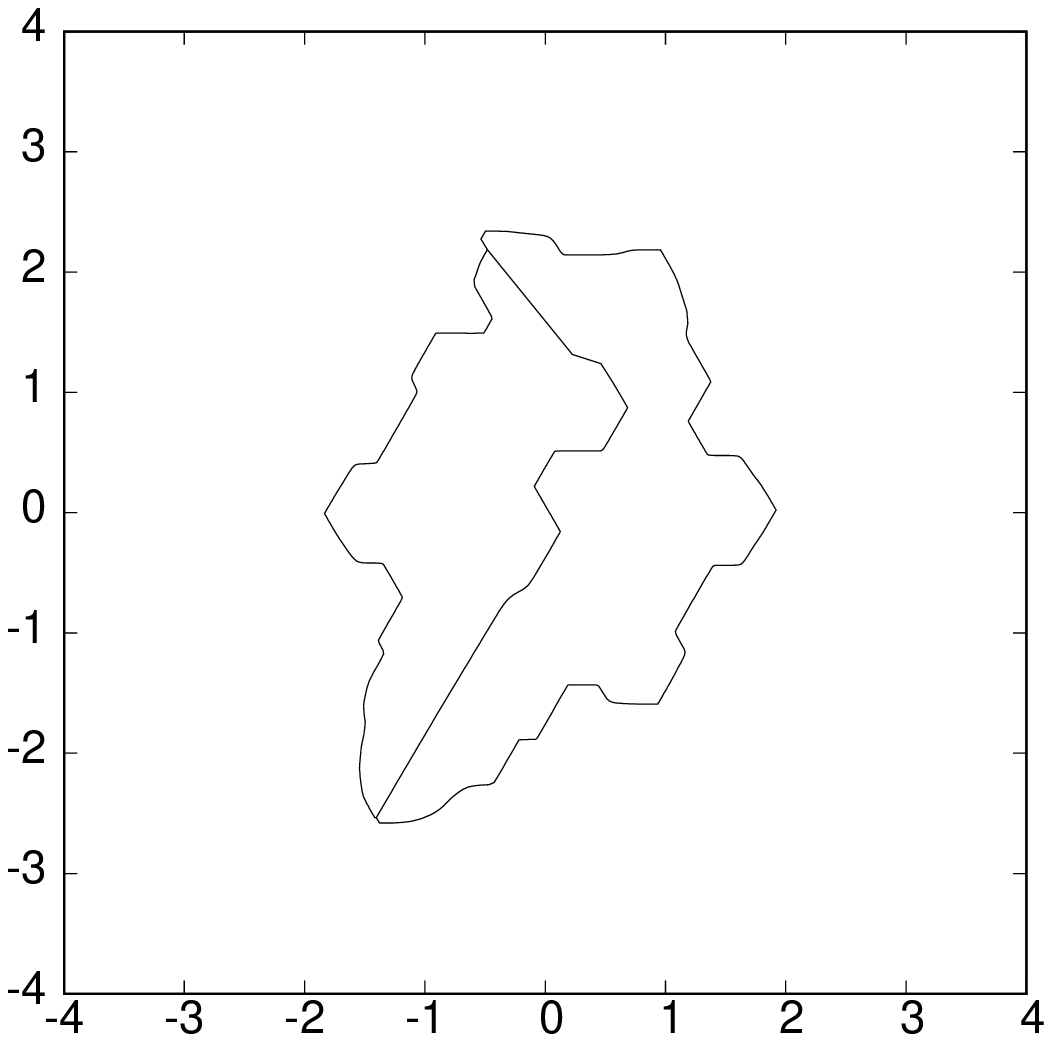}
\includegraphics[angle=-90,width=0.25\textwidth]{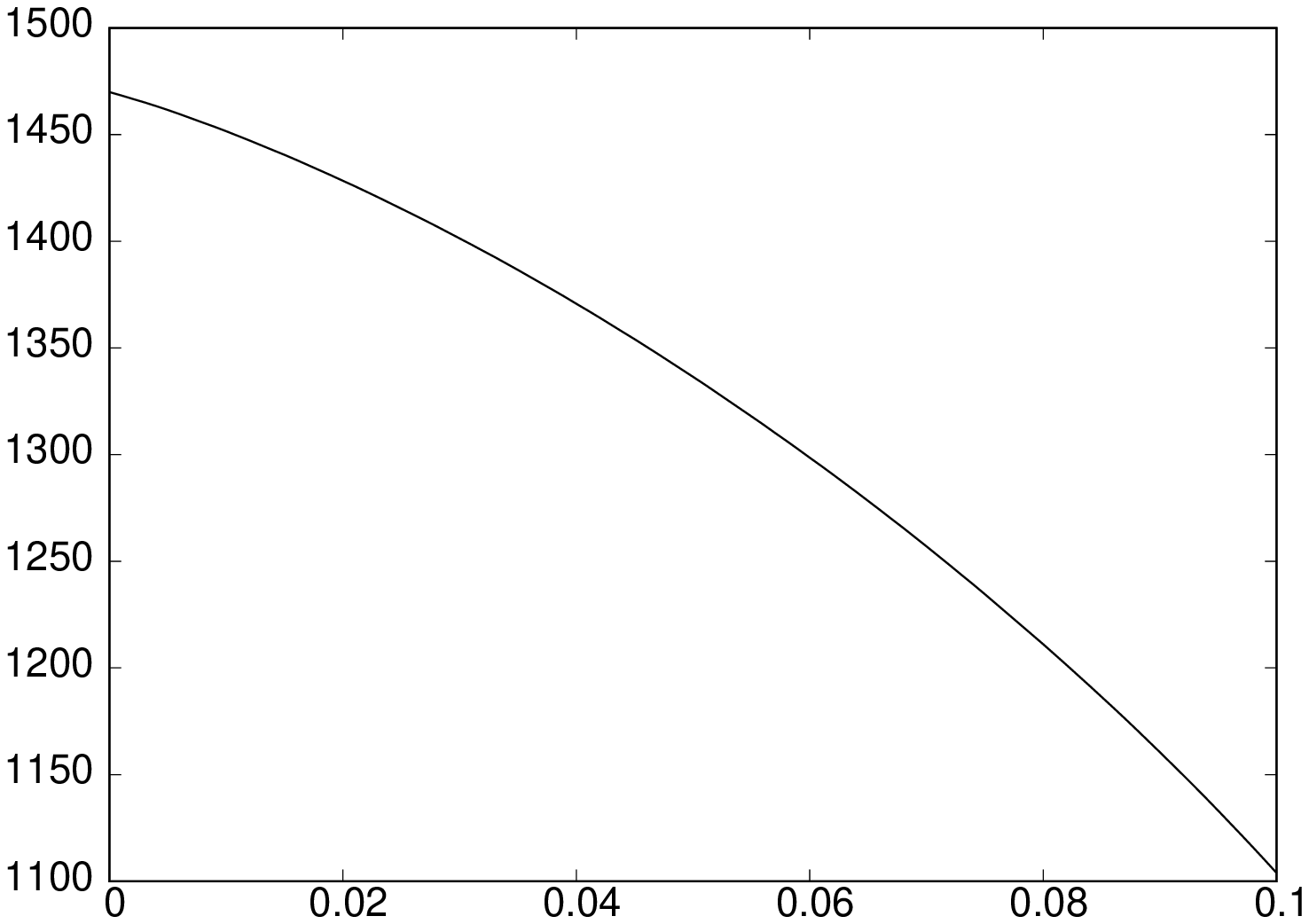} 
\caption{($\bv w_D = (12, 11, -23)^\top$, $\rho=\alpha=0.05$)
The solution at times $t=0, 0.01, \ldots, 0.1$, and separately at times 
$t=0.05$, $t=0.07$ and $t=0.1$,
and a plot of the discrete energy over time.
}
\label{fig:2dDbc1211hex_db_rho_005}
\end{figure}%
The same simulation but with the undercooling only applied to the right
boundary is shown in Figure~\ref{fig:2dNDbc1211hex_db_rho_005}. That is,
here $\pOmega_D = \{4\}\times(-4,4)$, and we also move the initial seed further
to the left to allow it more space to grown into.
In fact, we can observe dendritic growth towards the undercooled part of the
external boundary during the evolution.
\begin{figure}
\center
\includegraphics[angle=-90,width=0.18\textwidth]{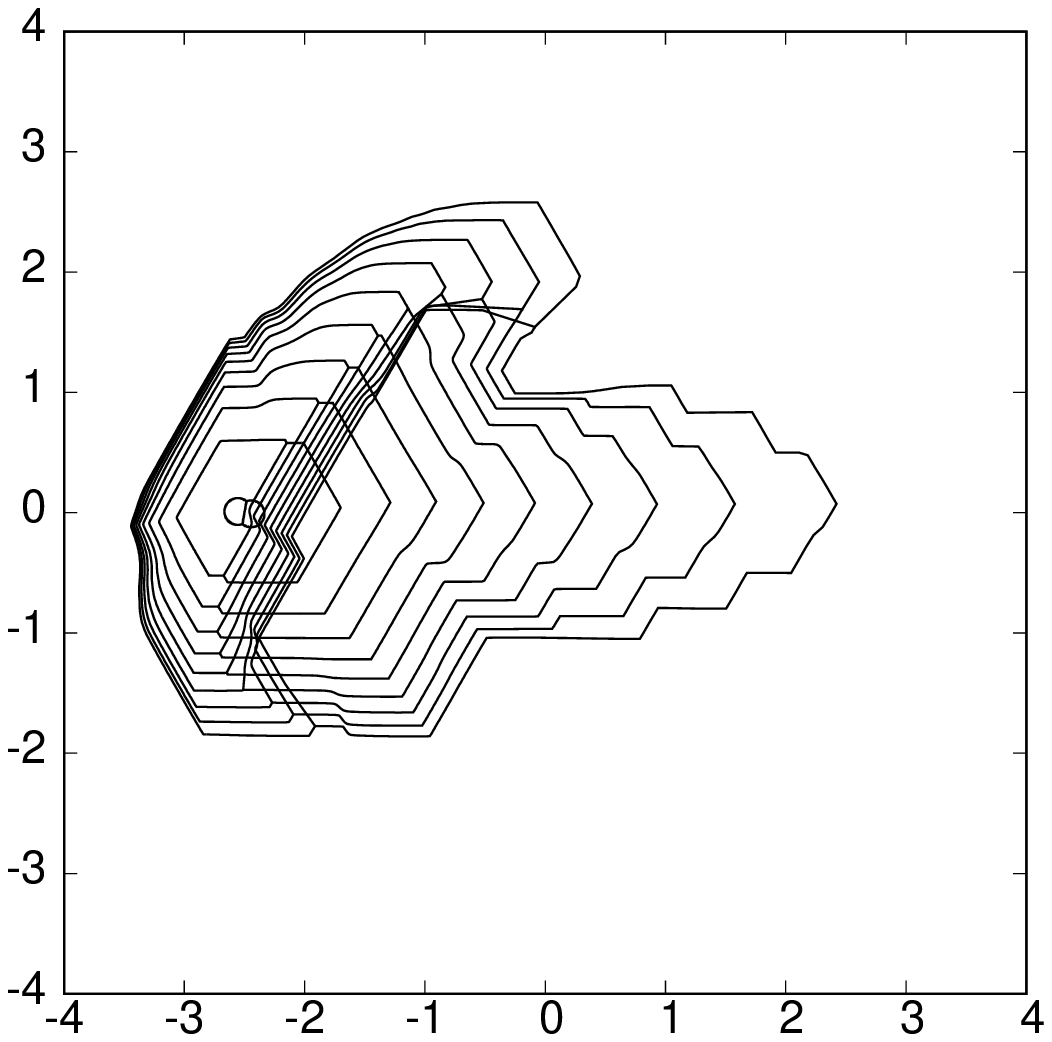}
\includegraphics[angle=-90,width=0.18\textwidth]{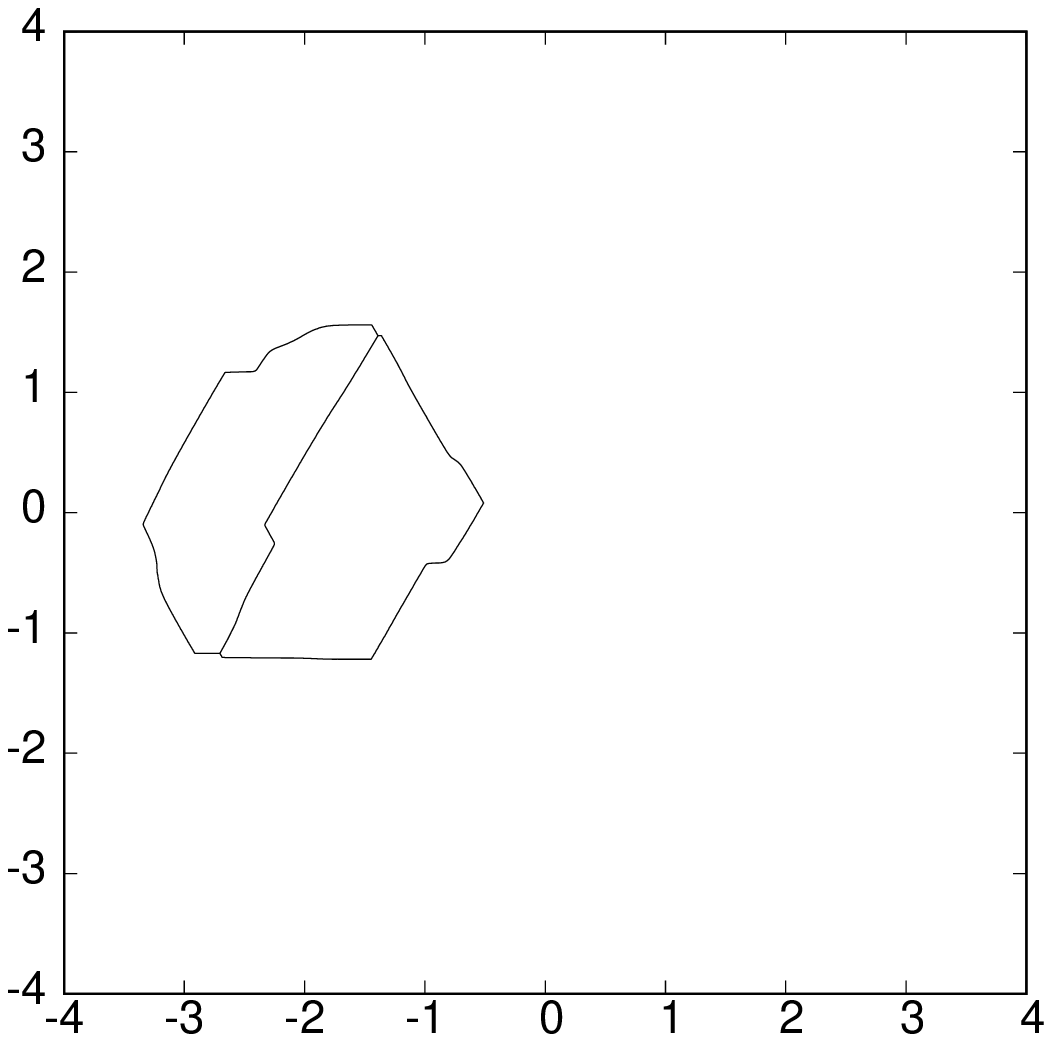}
\includegraphics[angle=-90,width=0.18\textwidth]{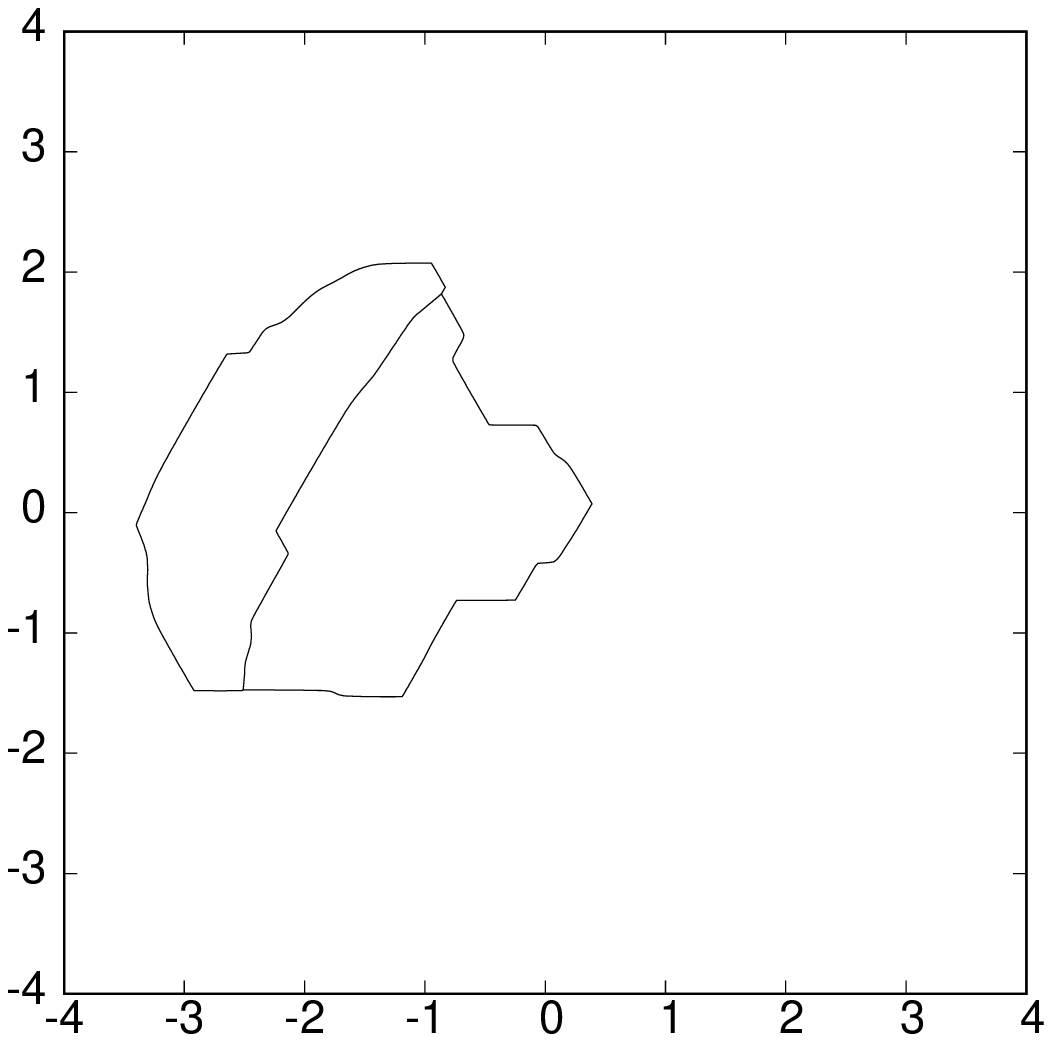}
\includegraphics[angle=-90,width=0.18\textwidth]{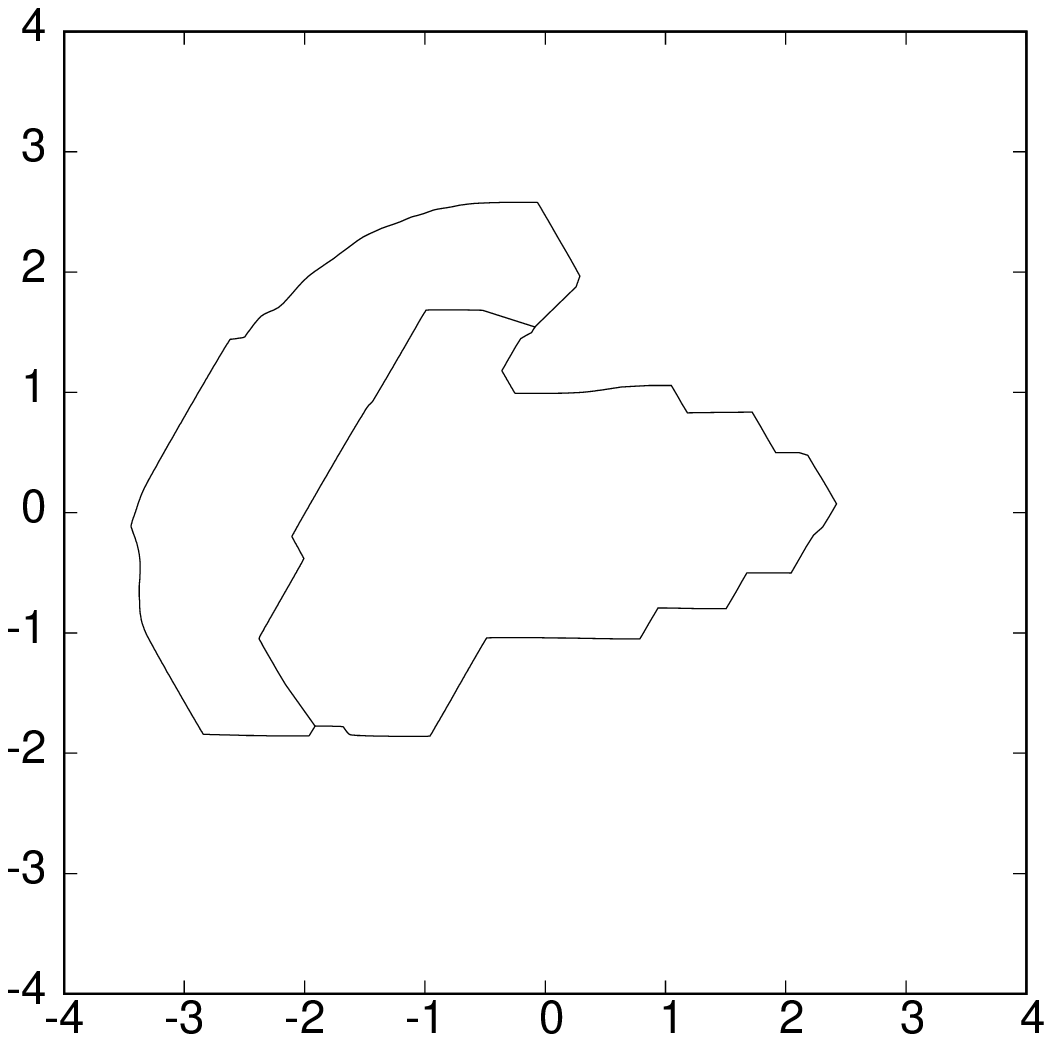}
\includegraphics[angle=-90,width=0.25\textwidth]{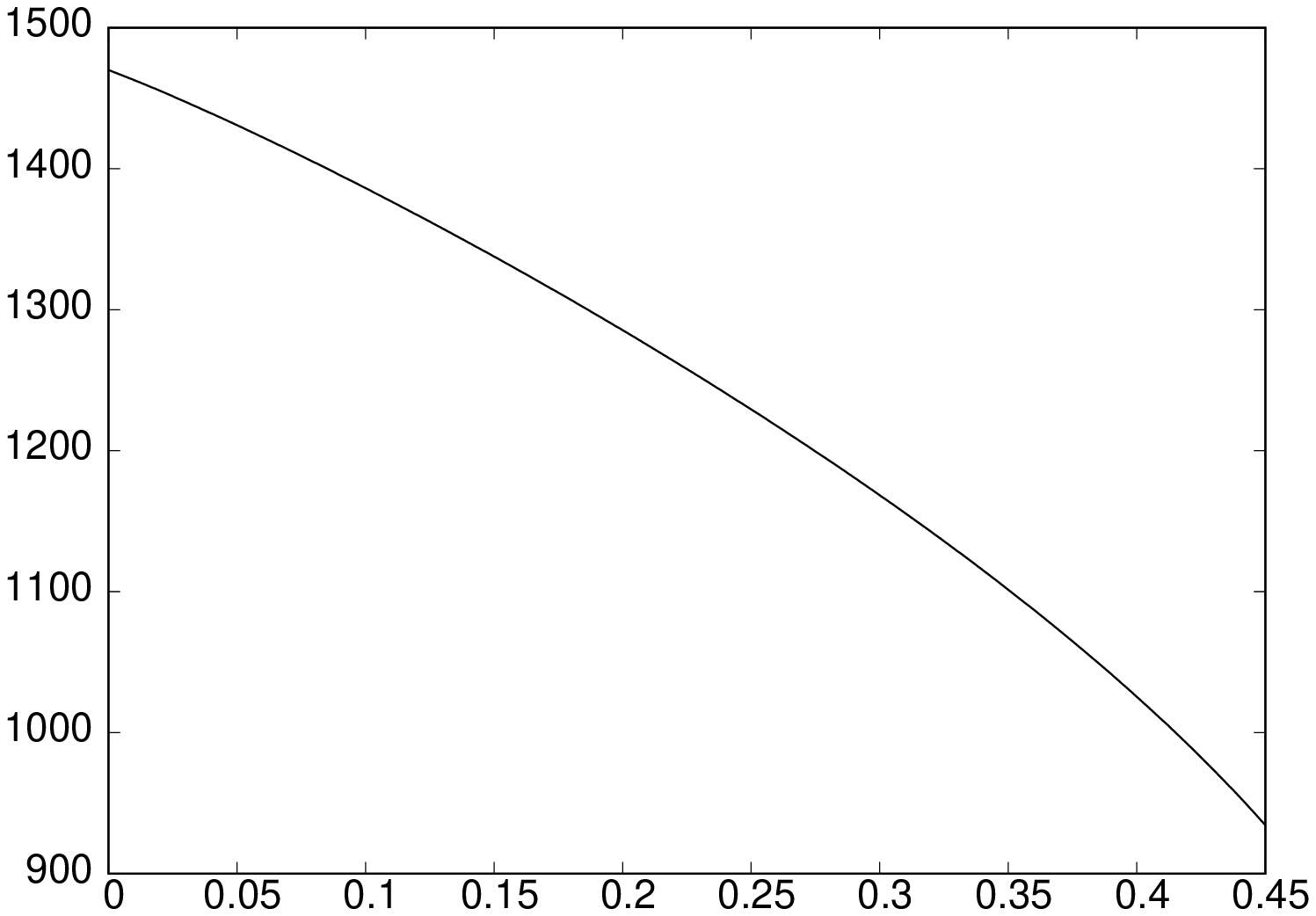} 
\caption{($\bv w_D = (12, 11, -23)^\top$, $\rho=\alpha=0.05$)
The solution at times $t=0, 0.05, \ldots, 0.45$, and separately at times 
$t=0.2$, $t=0.3$ and $t=0.45$,
and a plot of the discrete energy over time.
}
\label{fig:2dNDbc1211hex_db_rho_005}
\end{figure}%

\vspace{0.3cm}
\noindent
{\bf Example 6}:
We use the initial data from
Example~4 and now set $\bv w_D = (5, 5 -10)^\top$,
$\rho=0.05$, as well as $(\gamma_1,\gamma_2,\gamma_3) = \alpha
(\gamma_{\rm hex},\gamma_{\rm hex},\gamma_{\rm hex})$ where $\alpha=0.005$.
The simulation is shown in Figure~\ref{fig:2dDbc_hex_db_a0005_5_to_5}, with the
two crystals growing symmetrically.
\begin{figure}
\center
\includegraphics[angle=-90,width=0.18\textwidth]{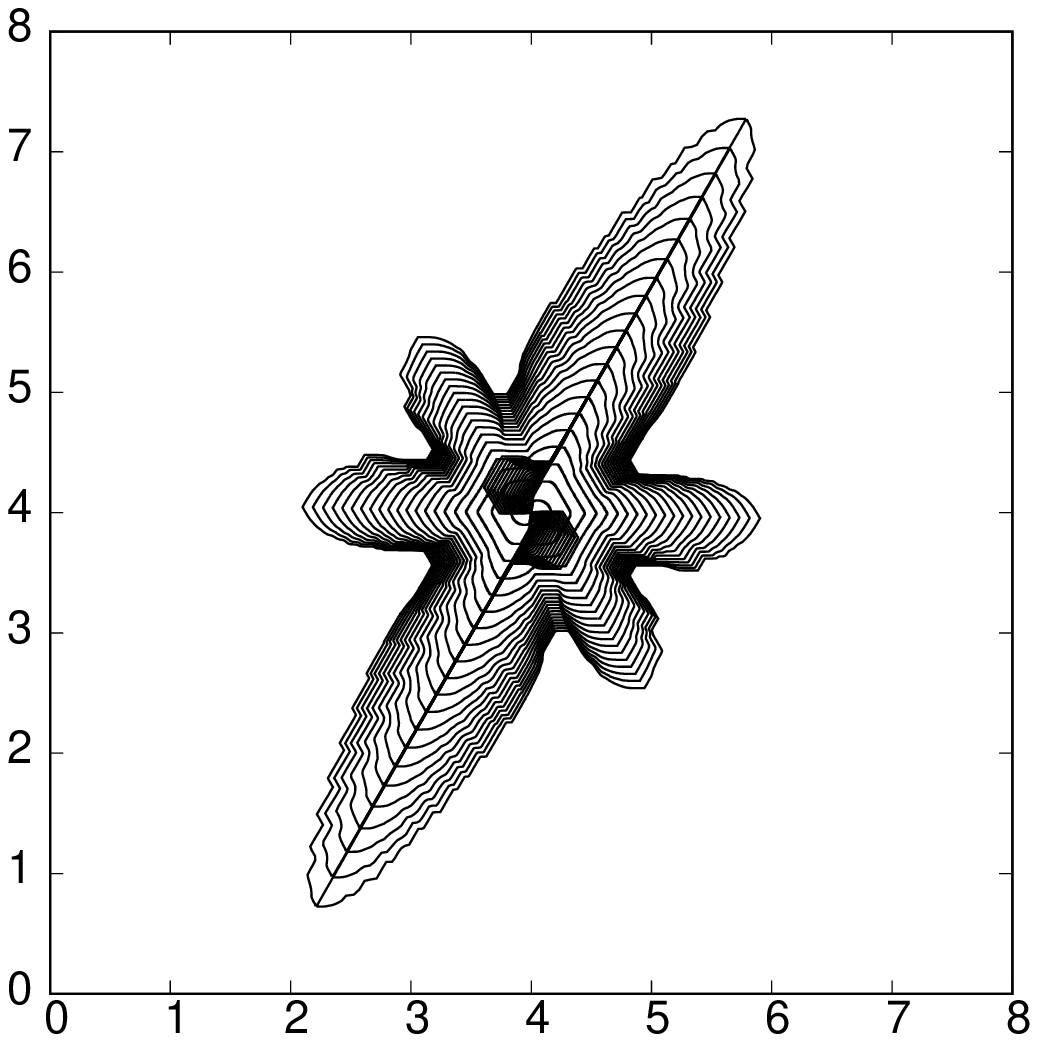}
\includegraphics[angle=-90,width=0.18\textwidth]{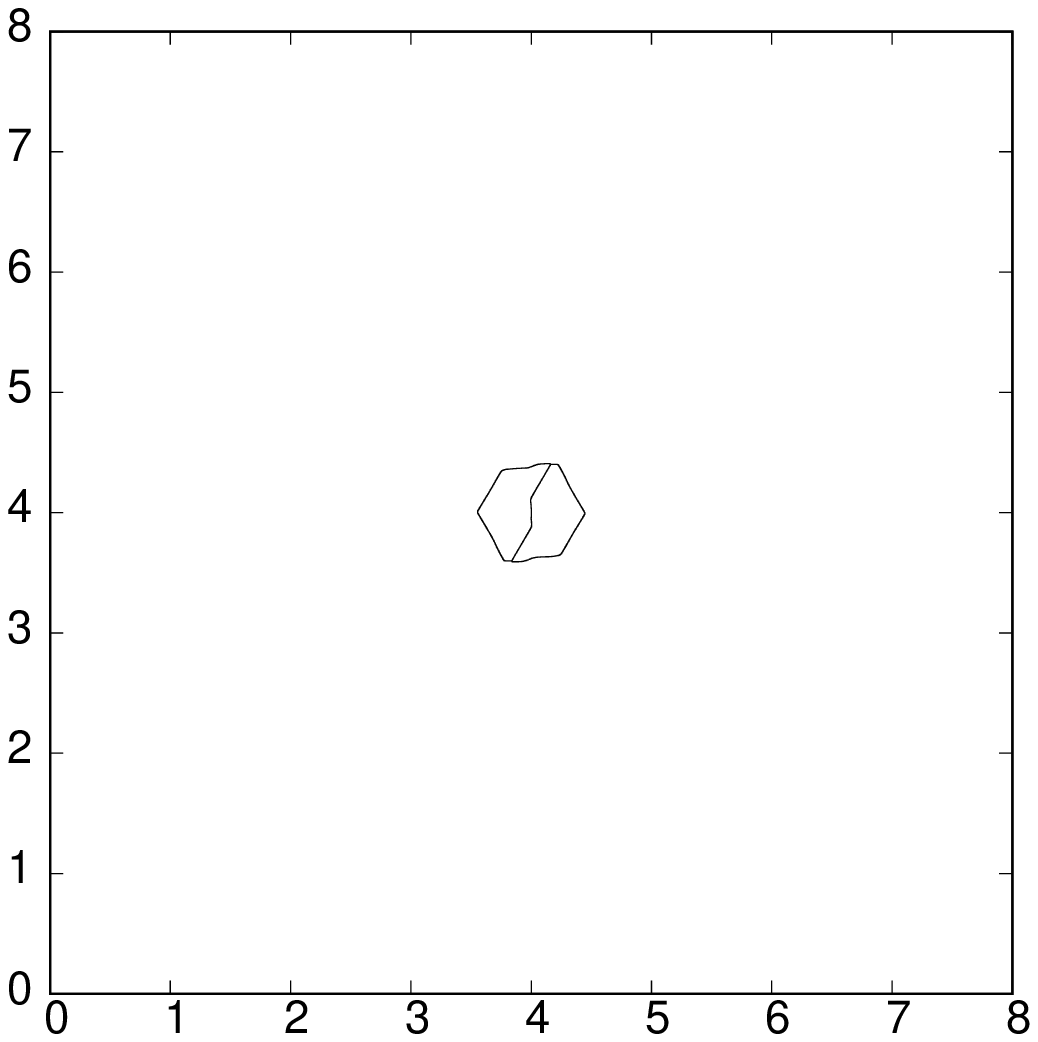}
\includegraphics[angle=-90,width=0.18\textwidth]{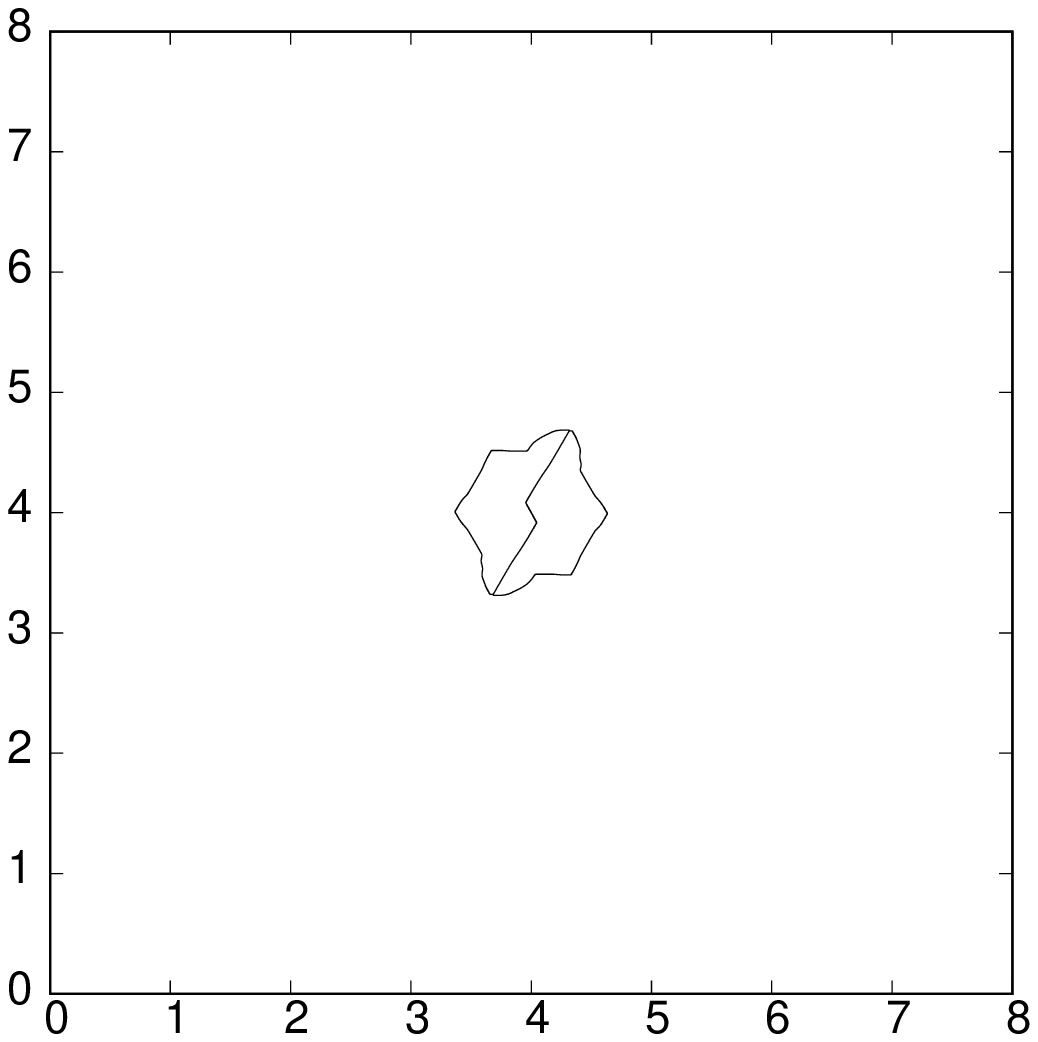}
\includegraphics[angle=-90,width=0.18\textwidth]{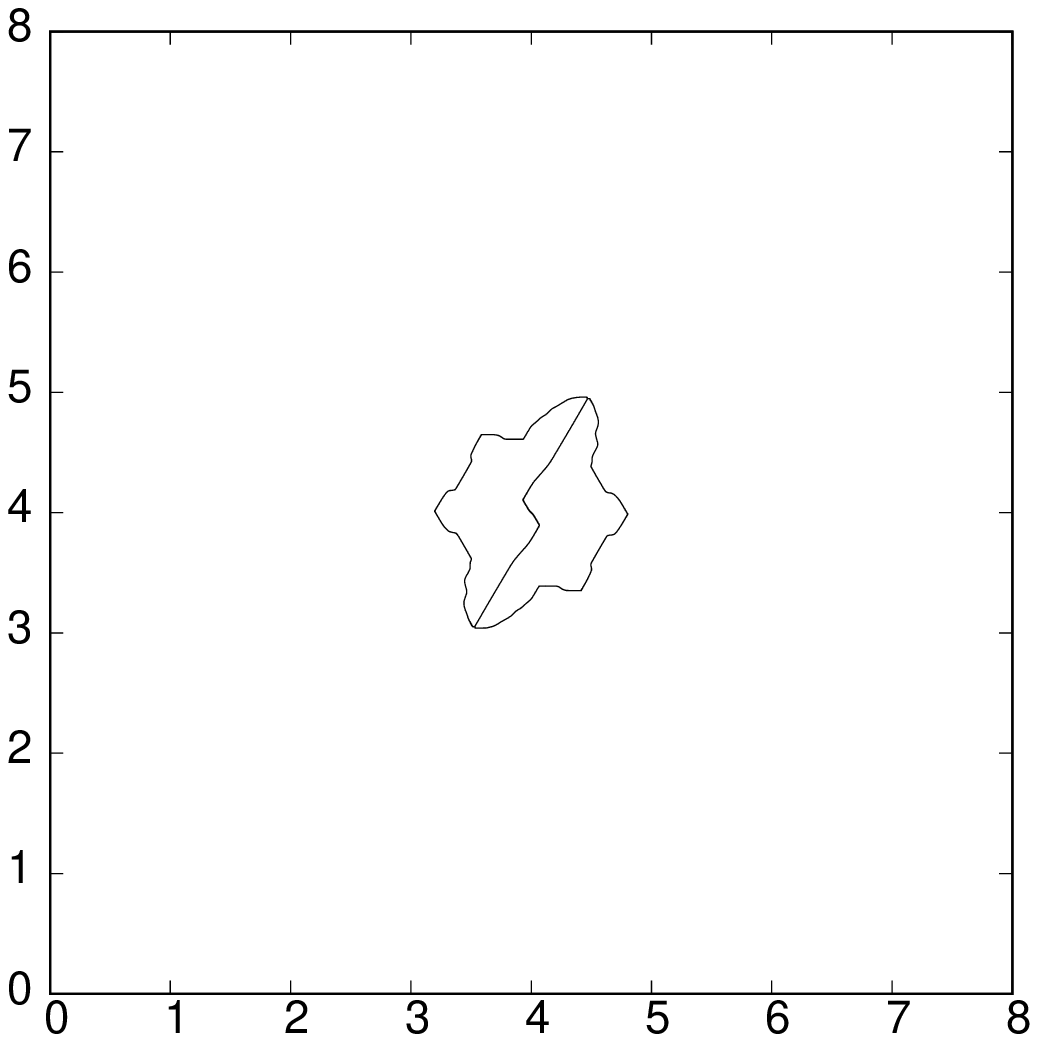}
\includegraphics[angle=-90,width=0.18\textwidth]{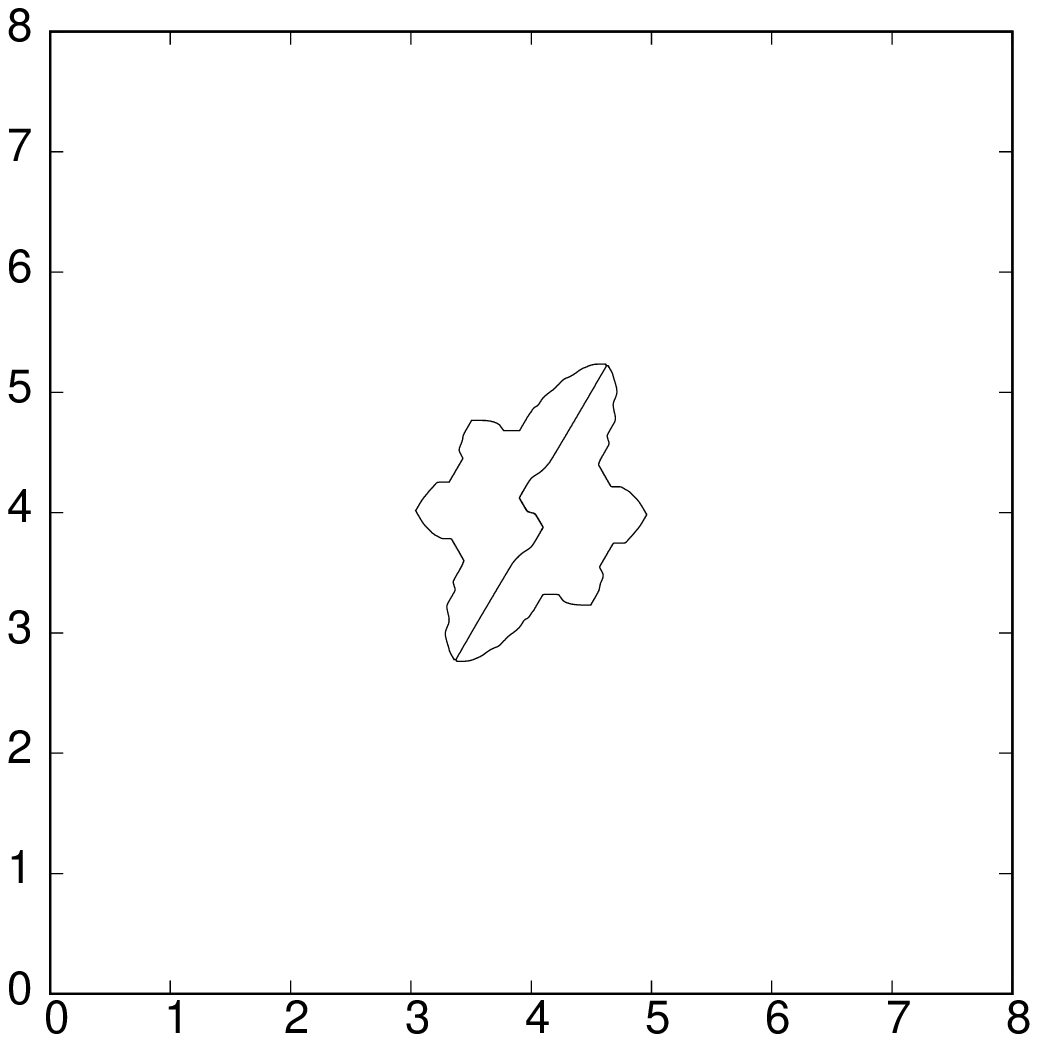}
\includegraphics[angle=-90,width=0.18\textwidth]{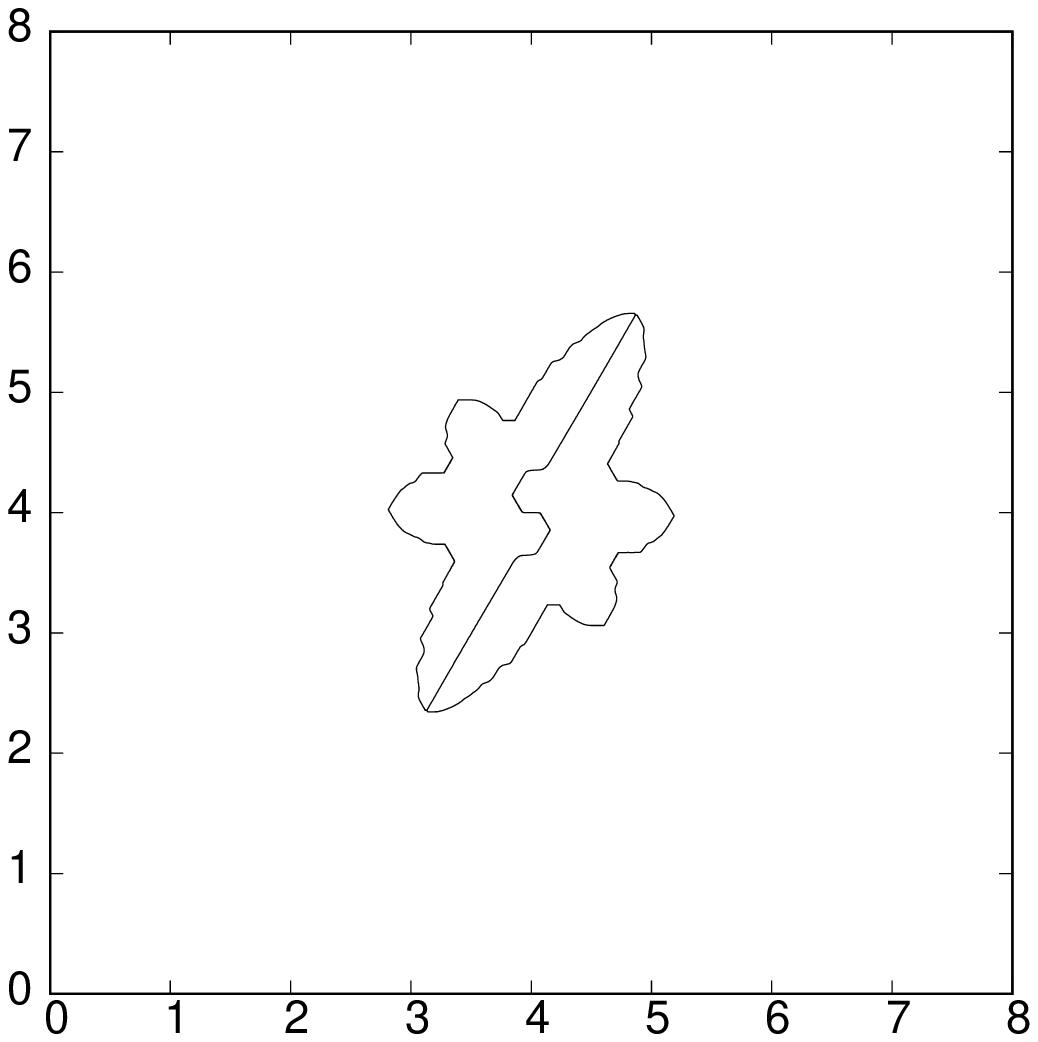}
\includegraphics[angle=-90,width=0.18\textwidth]{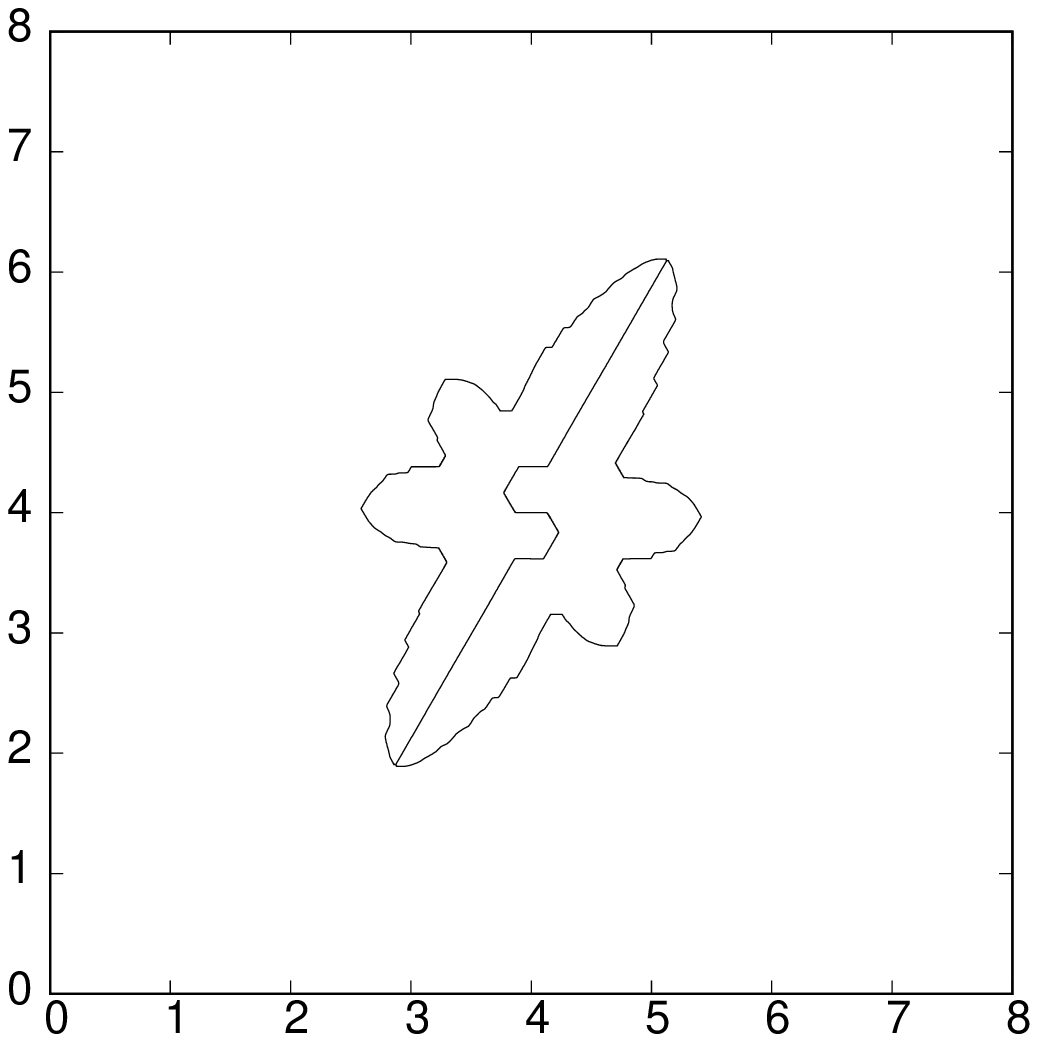}
\includegraphics[angle=-90,width=0.18\textwidth]{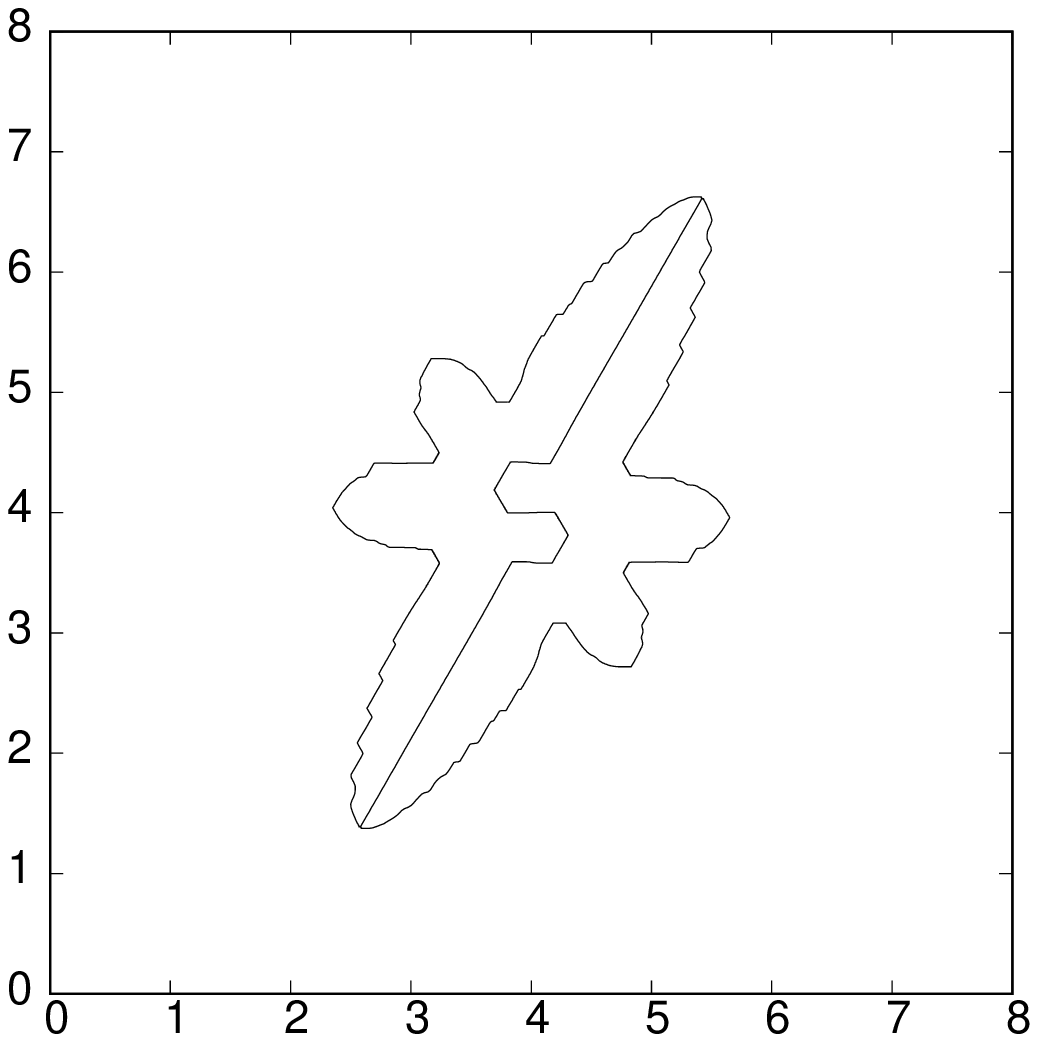}
\includegraphics[angle=-90,width=0.18\textwidth]{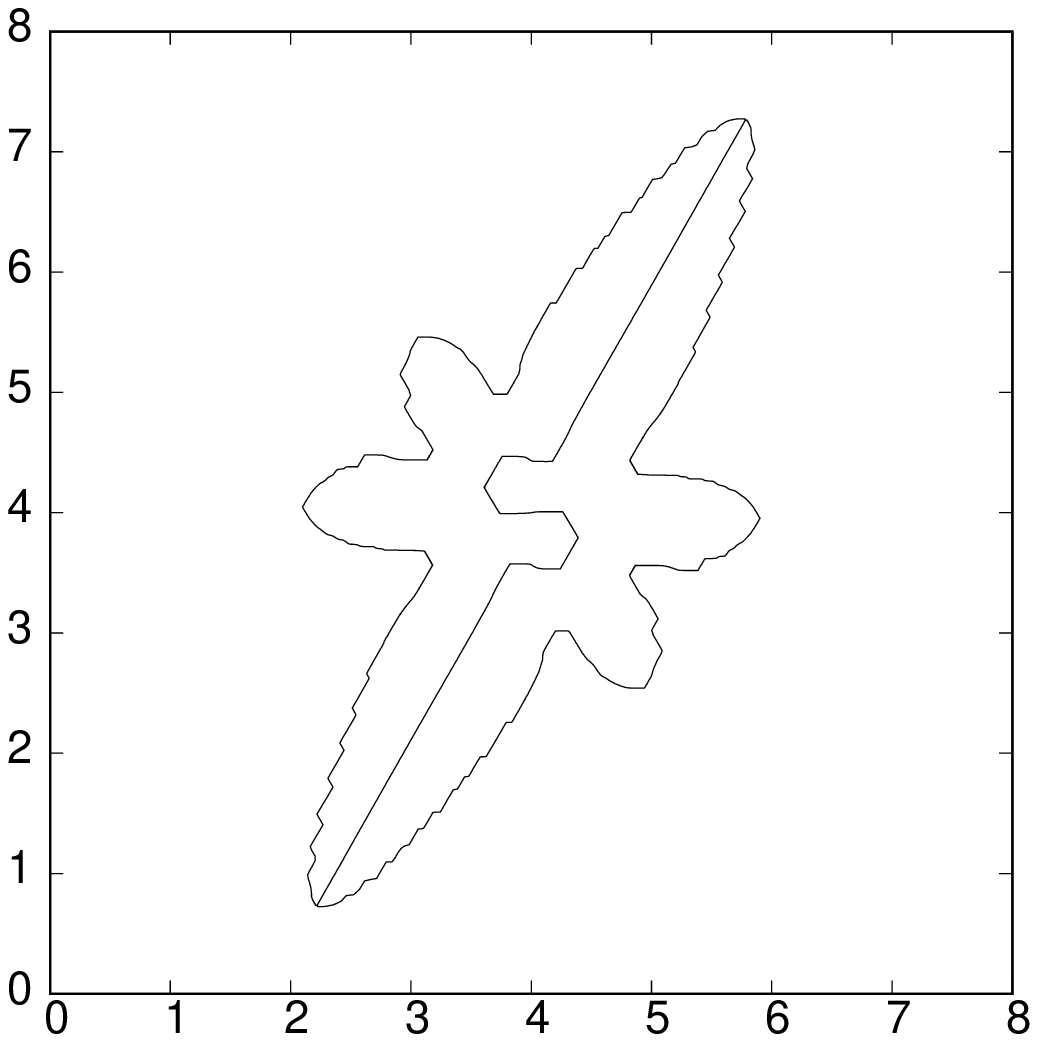}
\includegraphics[angle=-90,width=0.25\textwidth]{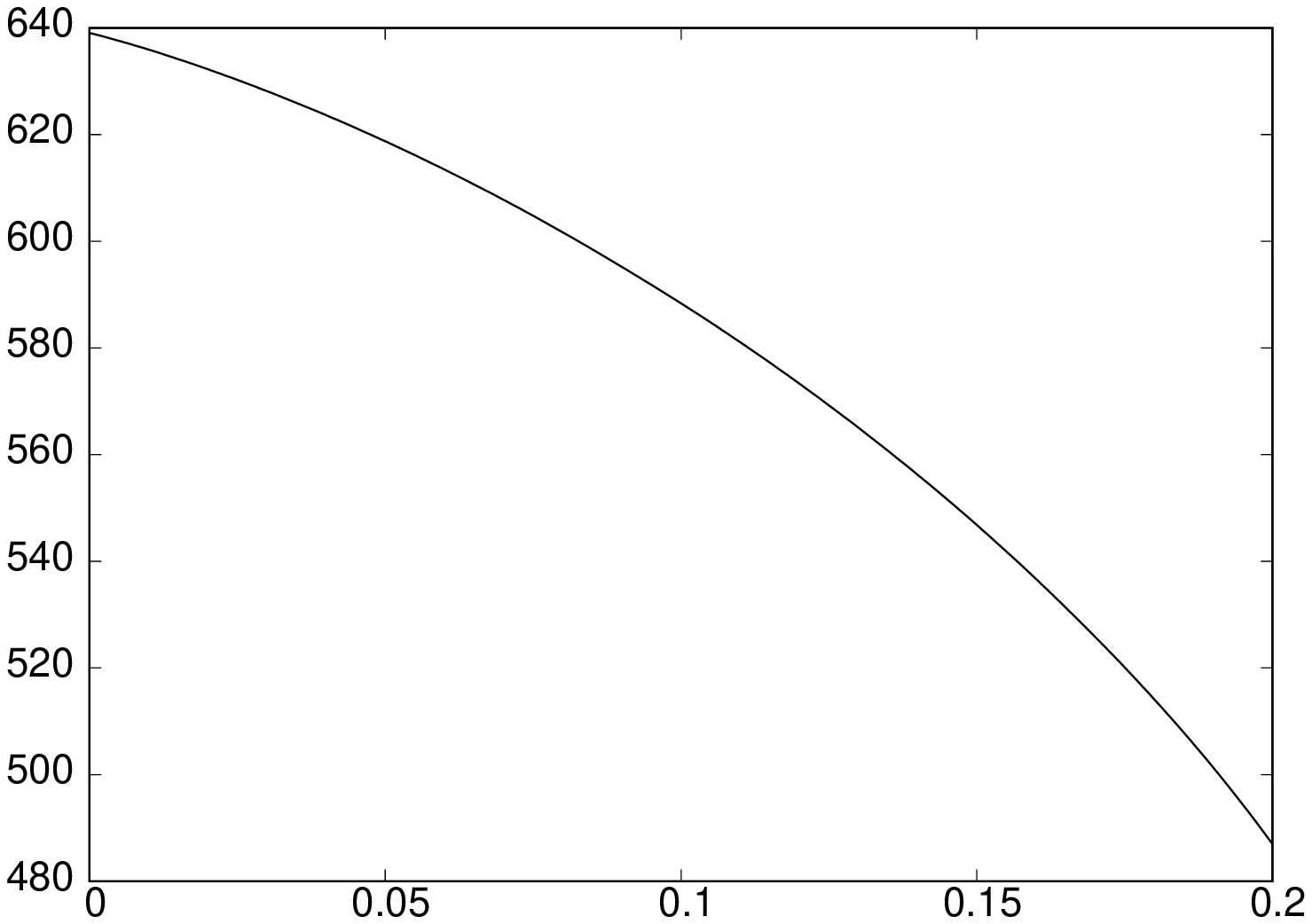} 
\caption{($\bv w_D = (5, 5, -10)^\top$, $\rho=0.05$, $\alpha=0.005$)
The solution at times $t=0, 0.01, \ldots, 0.2$, and separately at times 
$t=0.02$, $t=0.04$, $t=0.06$, $t=0.08$, $t=0.11$, $t=0.14$, $t=0.17$ and $t=0.2$,
and a plot of the discrete energy over time.
}
\label{fig:2dDbc_hex_db_a0005_5_to_5}
\end{figure}%
To break this symmetry, we next set
$\bv w_D = (10, 1, -11)^\top$, and let
$\rho=0.05$, as well as $(\gamma_1,\gamma_2,\gamma_3) = \alpha
(\gamma_{\rm hex},\gamma_{\rm hex},\gamma_{\rm hex})$ where $\alpha=0.05$.
The new evolution is far less symmetric, with the right crystal nearly
enveloping the left one. See Figure~\ref{fig:2dDbc_hex_db_rho_005_10_to_1}.
\begin{figure}
\center
\includegraphics[angle=-90,width=0.18\textwidth]{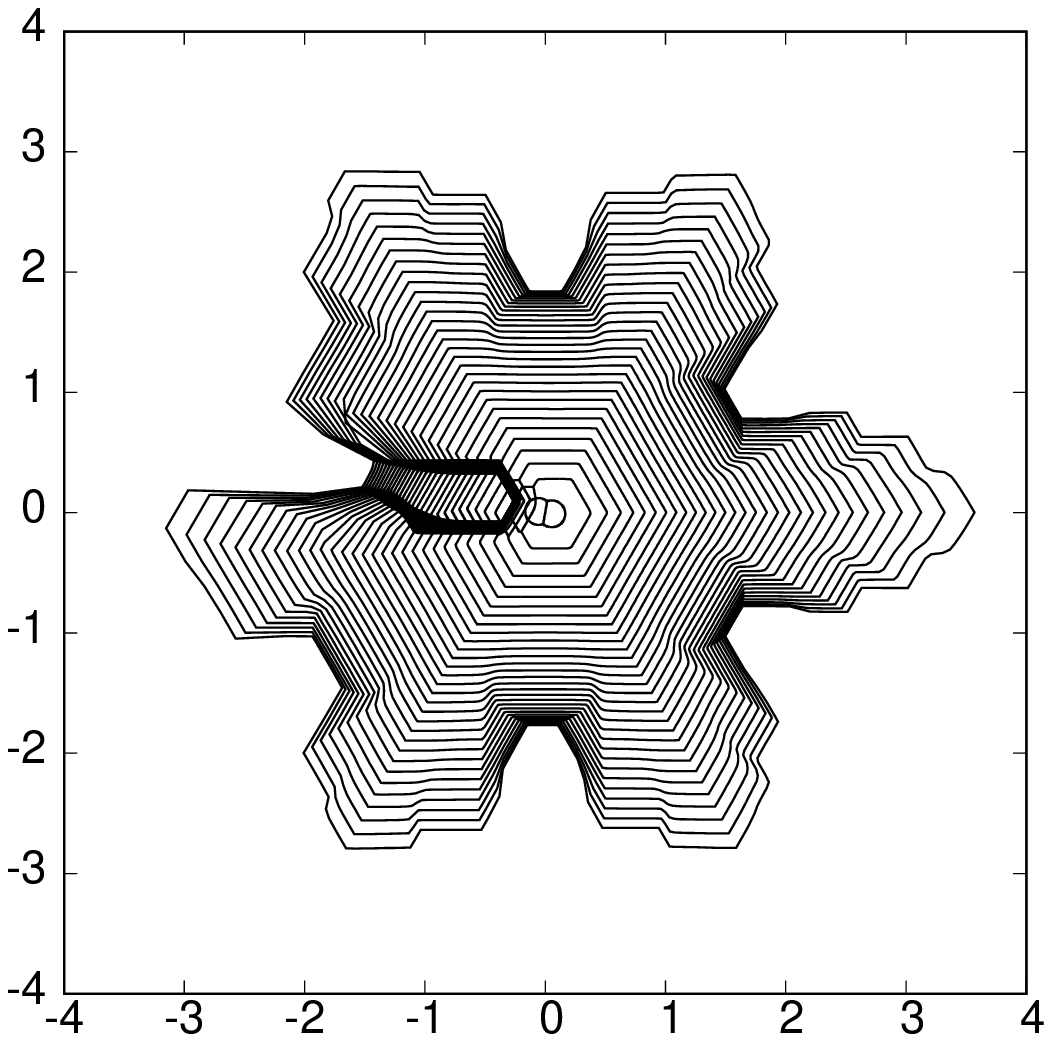}
\includegraphics[angle=-90,width=0.18\textwidth]{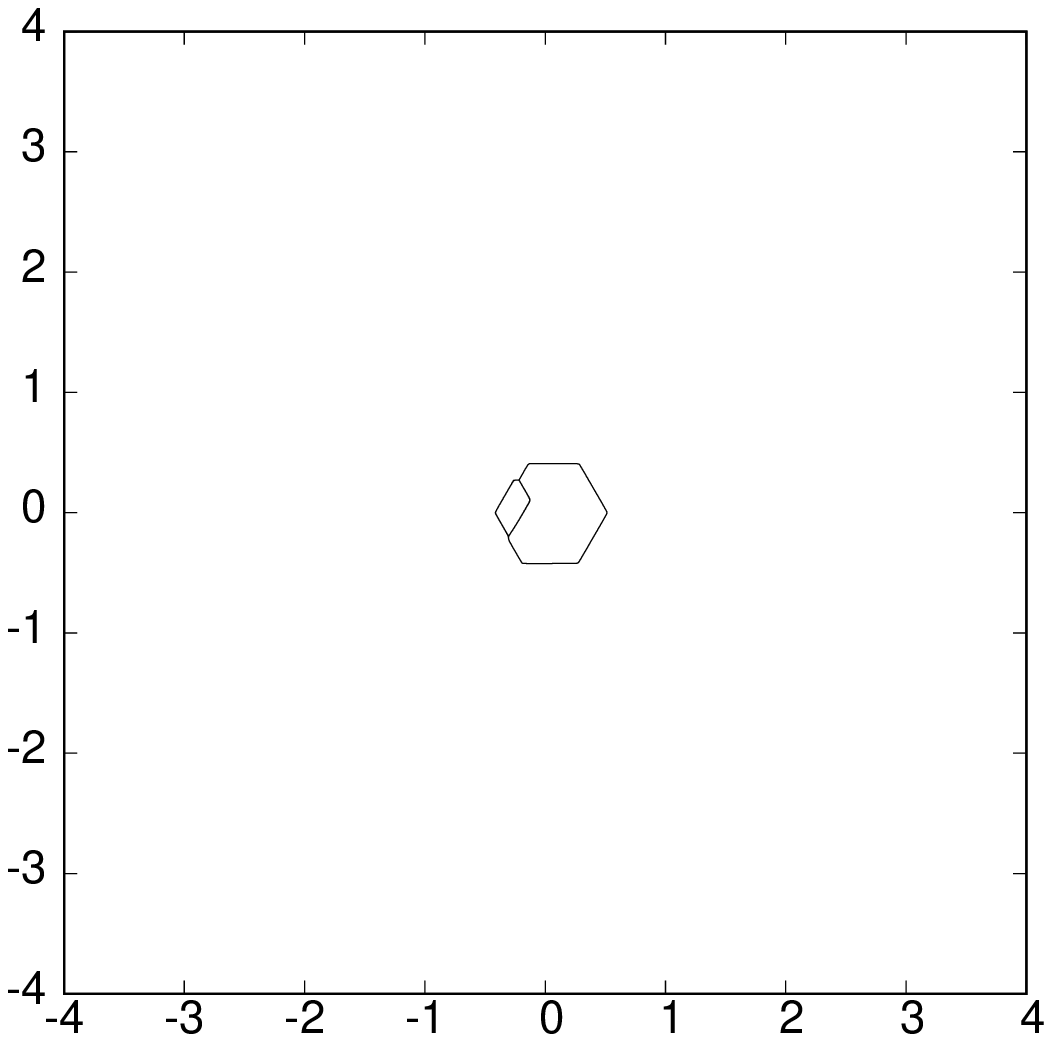}
\includegraphics[angle=-90,width=0.18\textwidth]{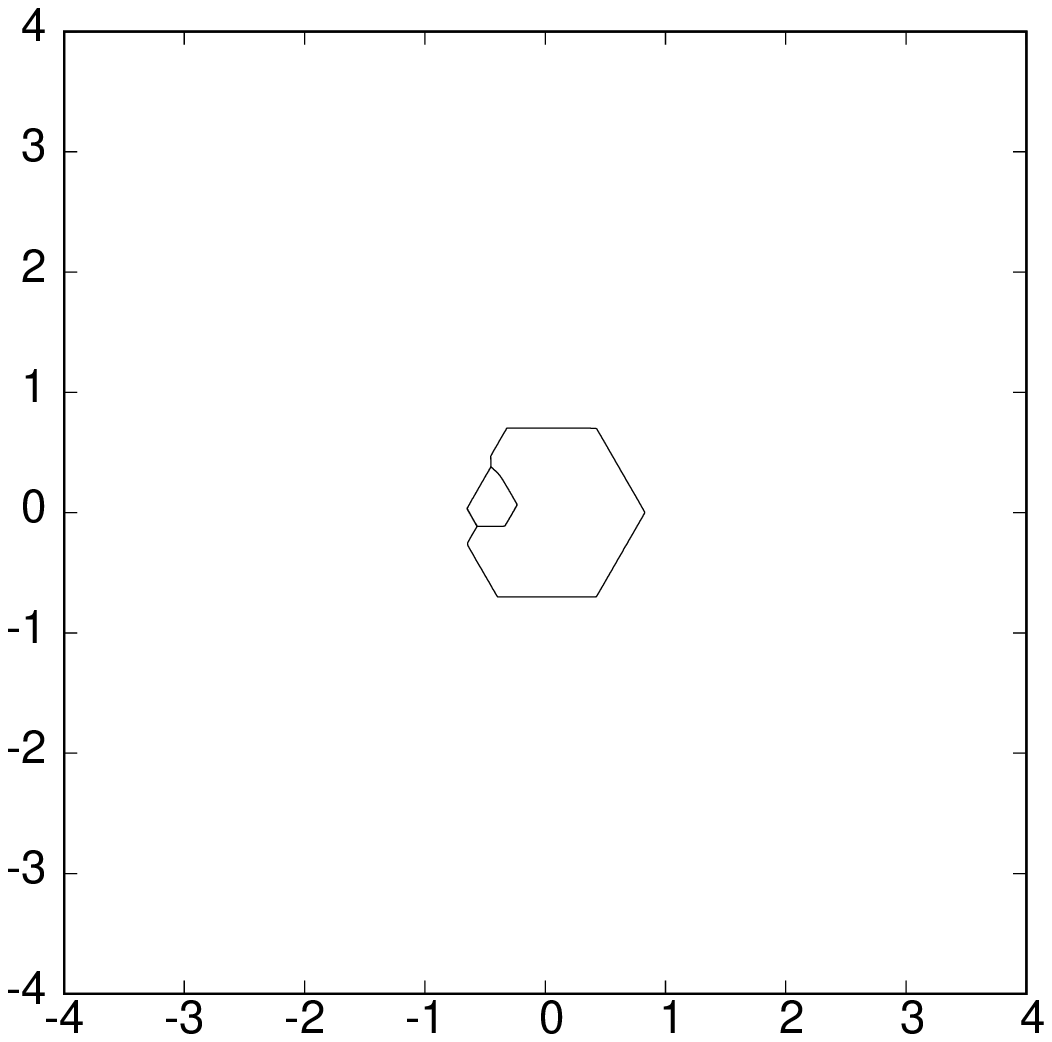}
\includegraphics[angle=-90,width=0.18\textwidth]{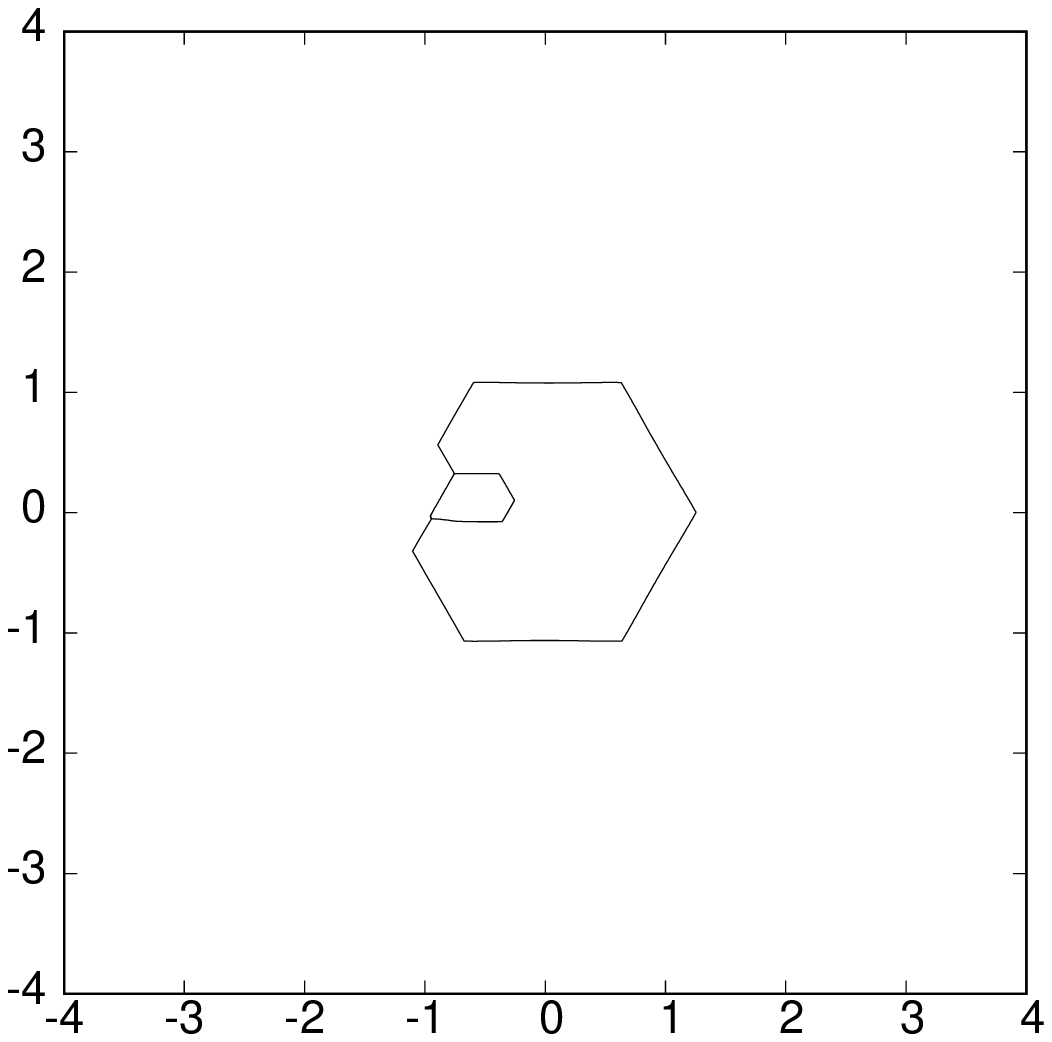}
\includegraphics[angle=-90,width=0.18\textwidth]{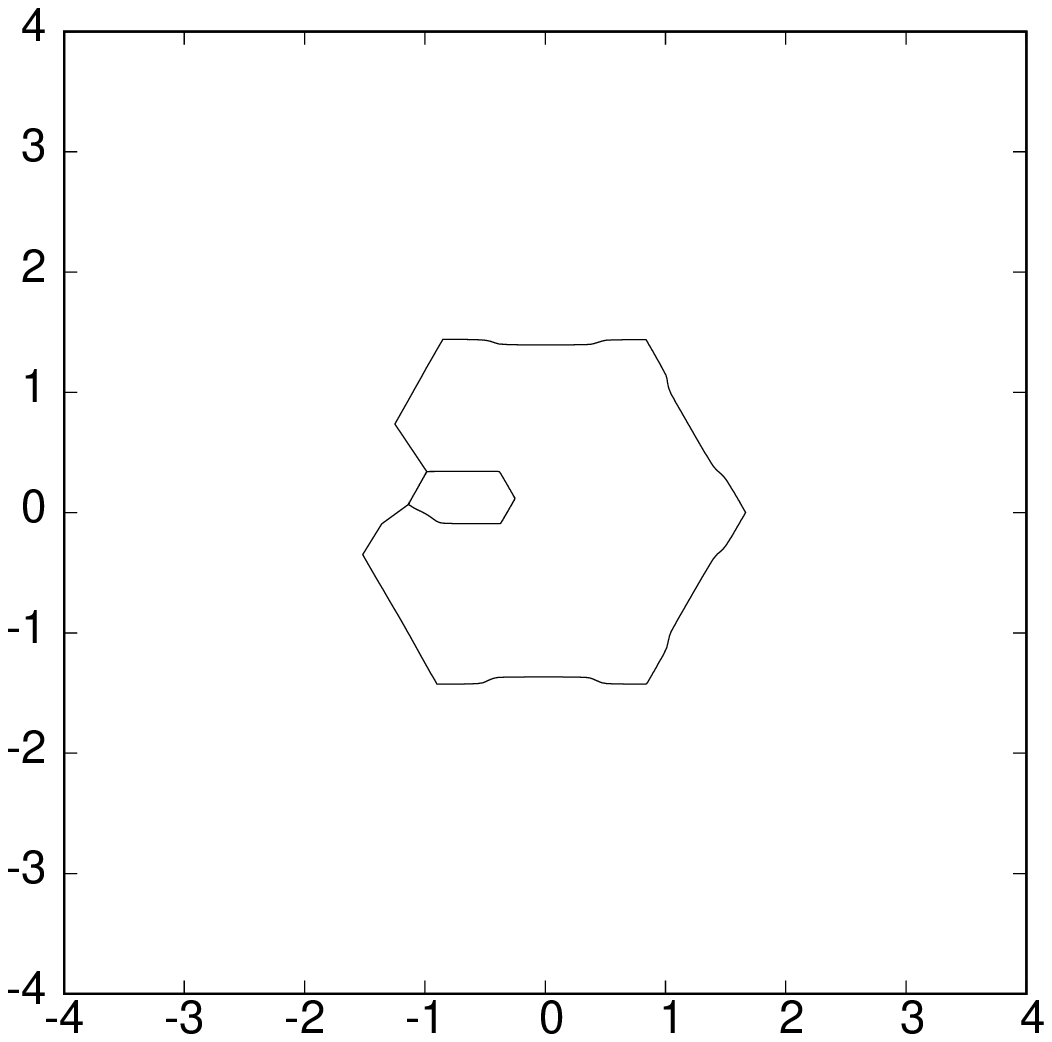}
\includegraphics[angle=-90,width=0.18\textwidth]{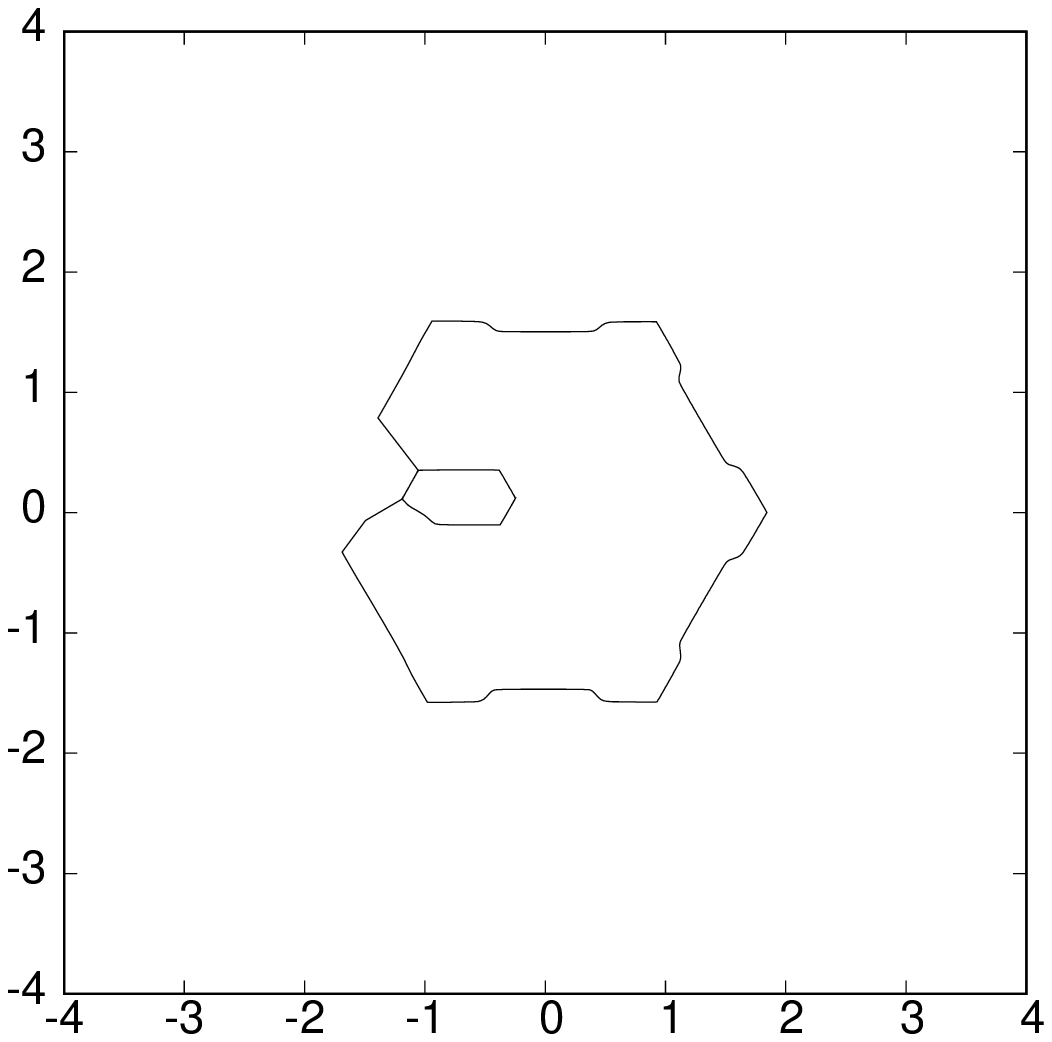}
\includegraphics[angle=-90,width=0.18\textwidth]{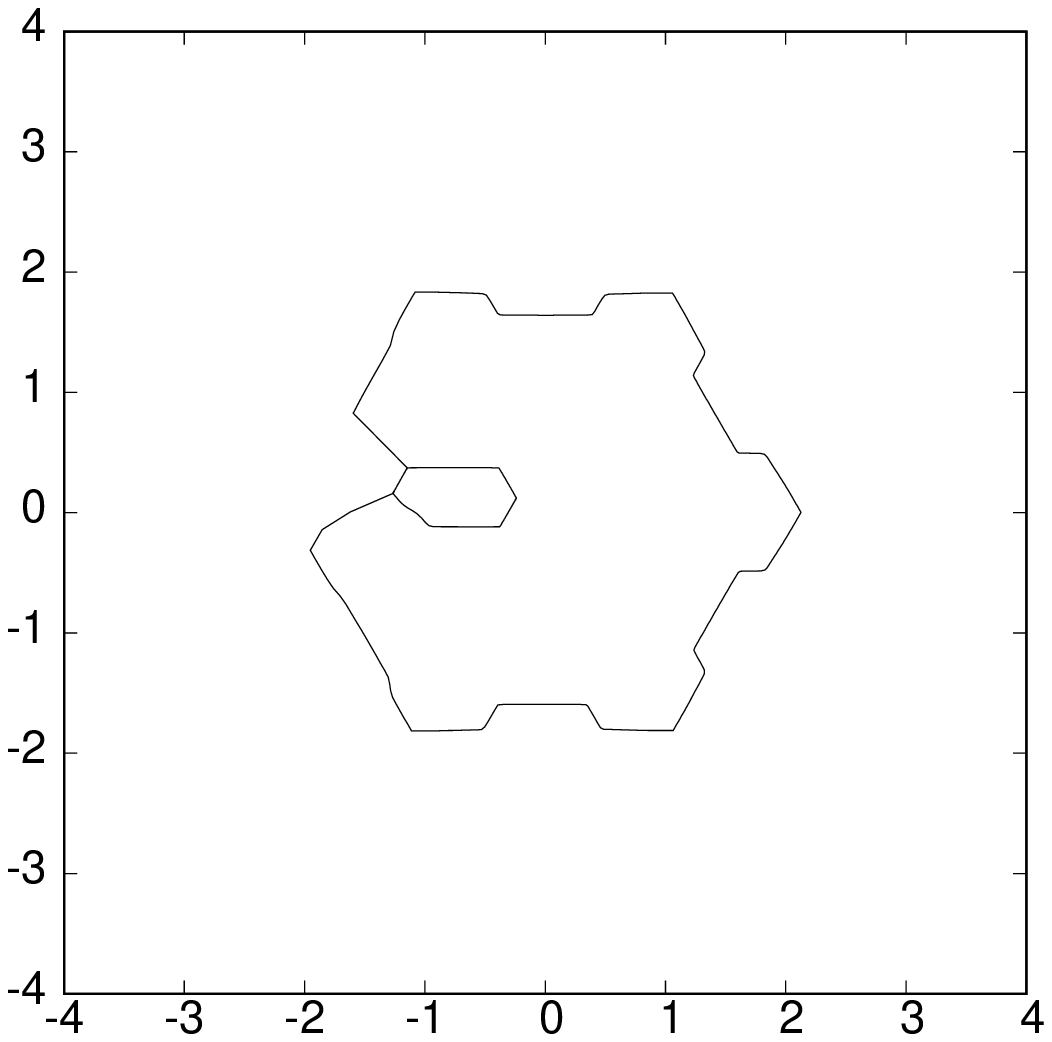}
\includegraphics[angle=-90,width=0.18\textwidth]{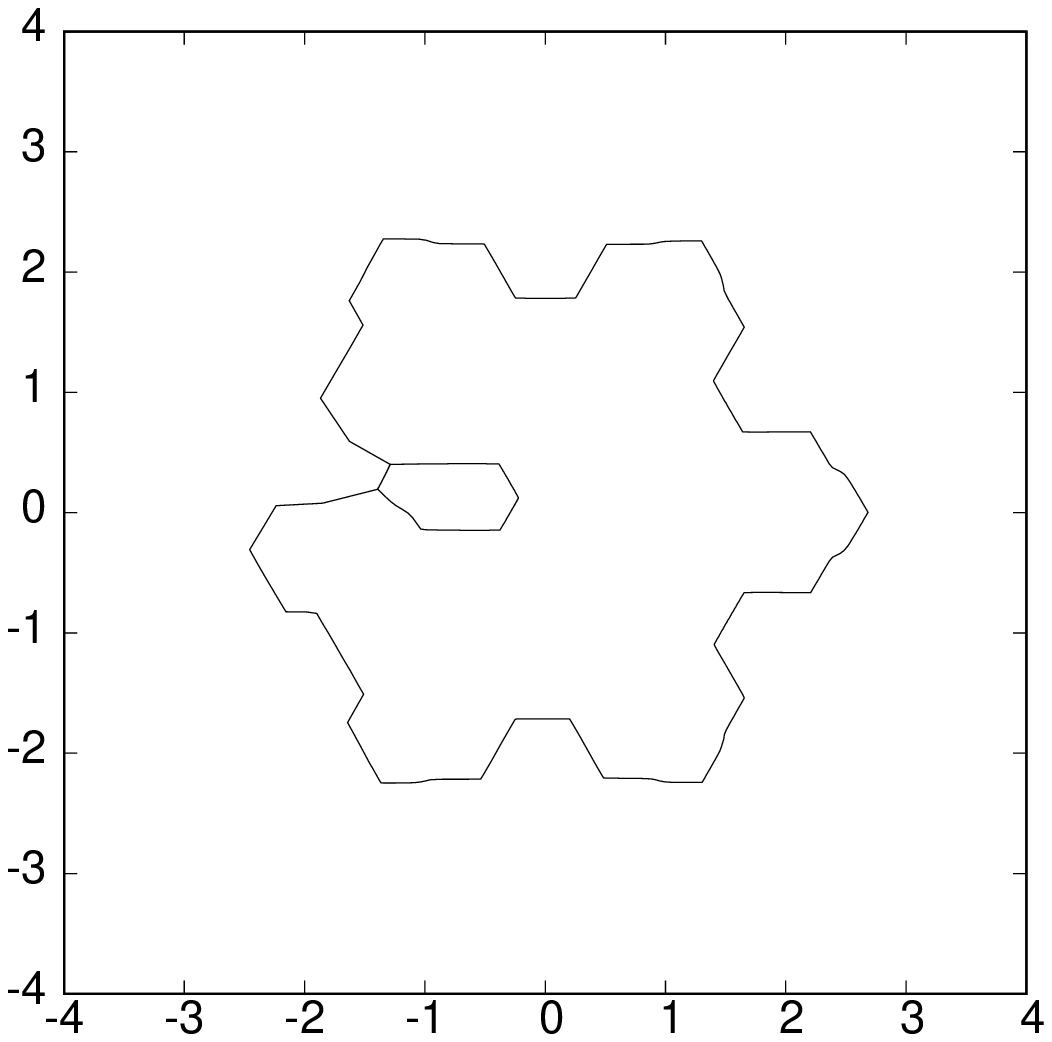}
\includegraphics[angle=-90,width=0.18\textwidth]{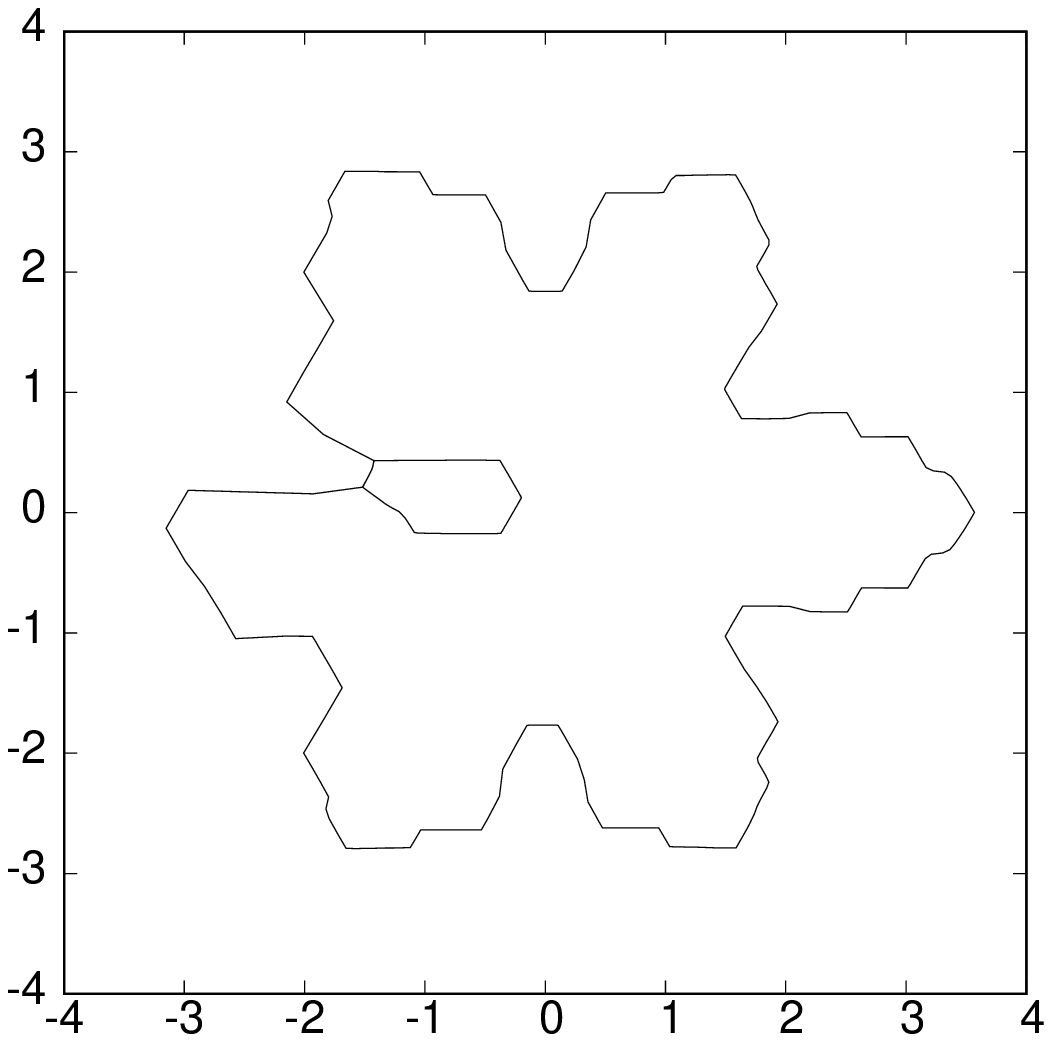}
\includegraphics[angle=-90,width=0.25\textwidth]{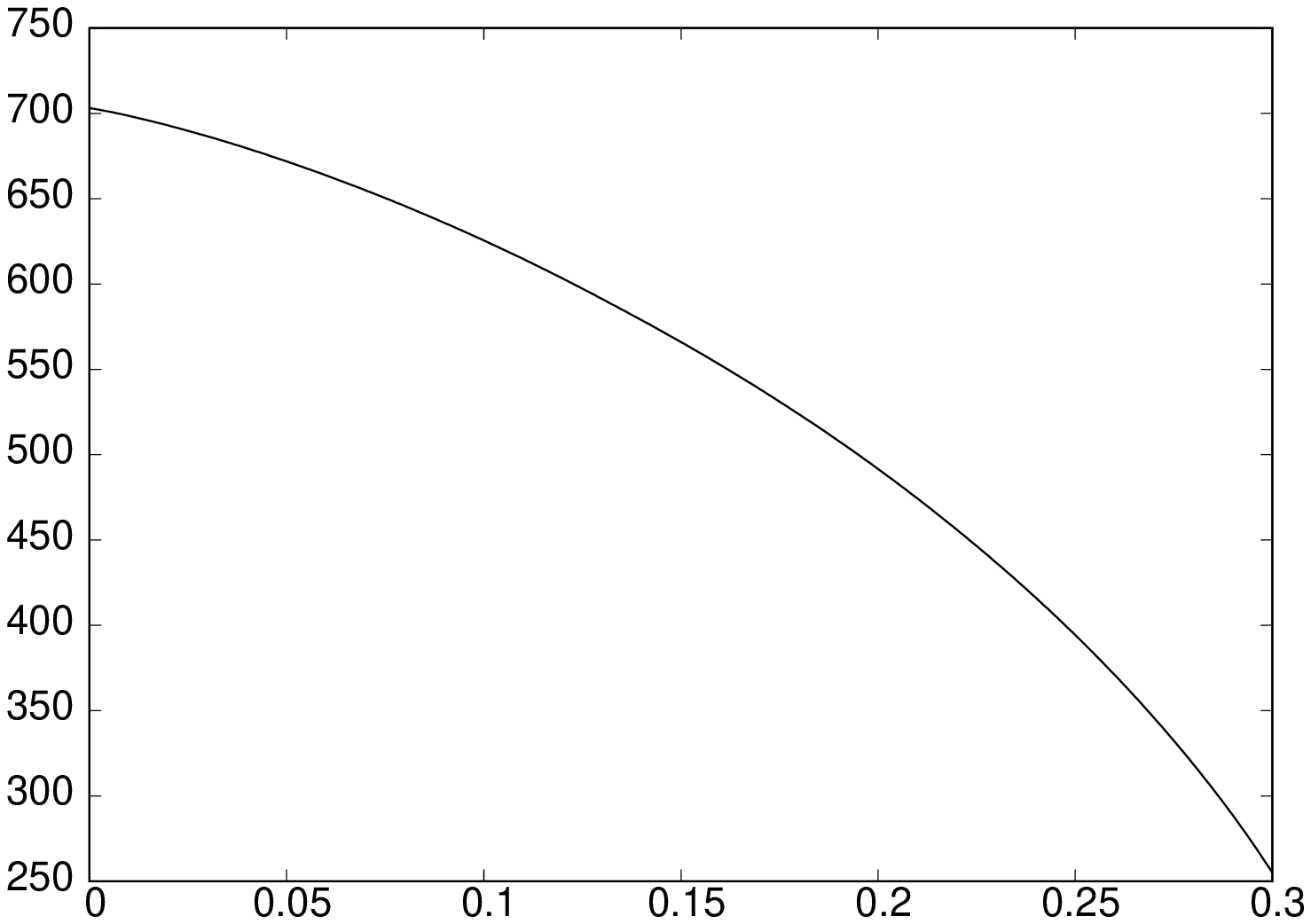} 
\caption{($\bv w_D = (10, 1, -11)^\top$, $\rho=\alpha=0.05$)
The solution at times $t=0, 0.01, \ldots, 0.3$, and separately at times 
$t=0.02$, $t=0.05$, $t=0.1$, $t=0.15$, $t=0.17$, $t=0.2$, $t=0.25$ and $t=0.3$,
and a plot of the discrete energy over time.
}
\label{fig:2dDbc_hex_db_rho_005_10_to_1}
\end{figure}%
A further nonsymmetric example with $\bv w_D = (1, 10, -11)^\top$,
$\rho=0.05$, as well as $(\gamma_1,\gamma_2,\gamma_3) = \alpha
(\gamma_{\rm hex},\gamma_{\rm hex},\gamma_{\rm hex})$ where $\alpha=0.02$
is shown in Figure~\ref{fig:2dDbc_hex_db_rho_005_1_to_10}.
\begin{figure}
\center
\includegraphics[angle=-90,width=0.18\textwidth]{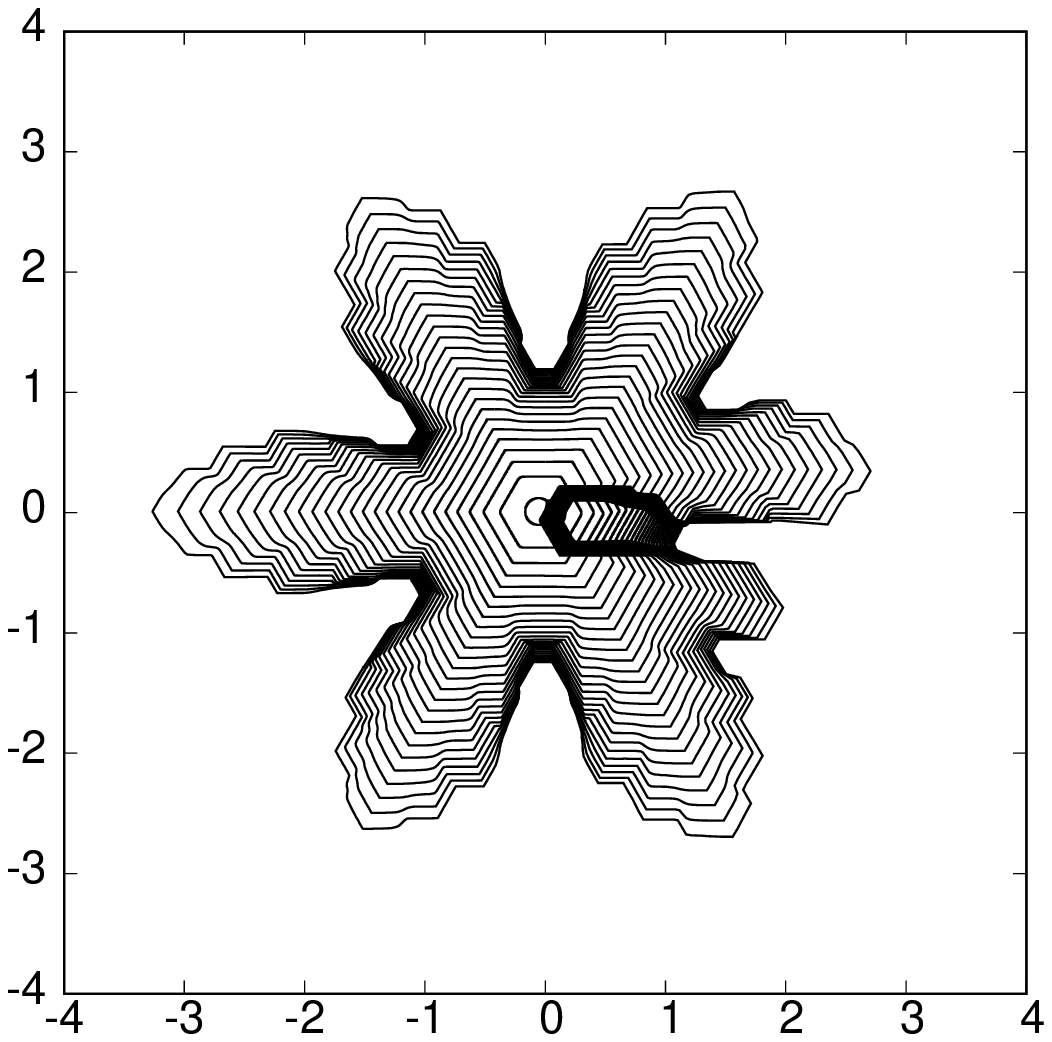}
\includegraphics[angle=-90,width=0.18\textwidth]{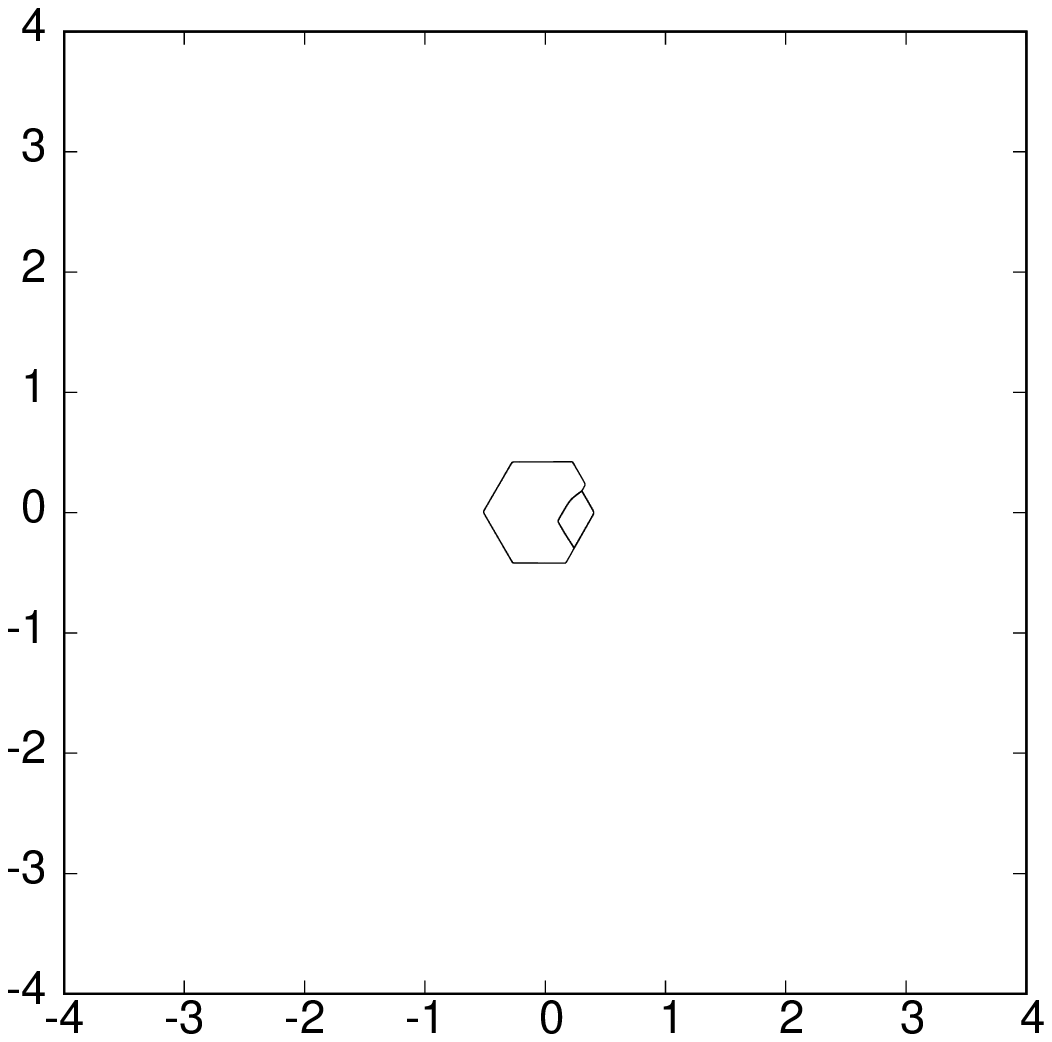}
\includegraphics[angle=-90,width=0.18\textwidth]{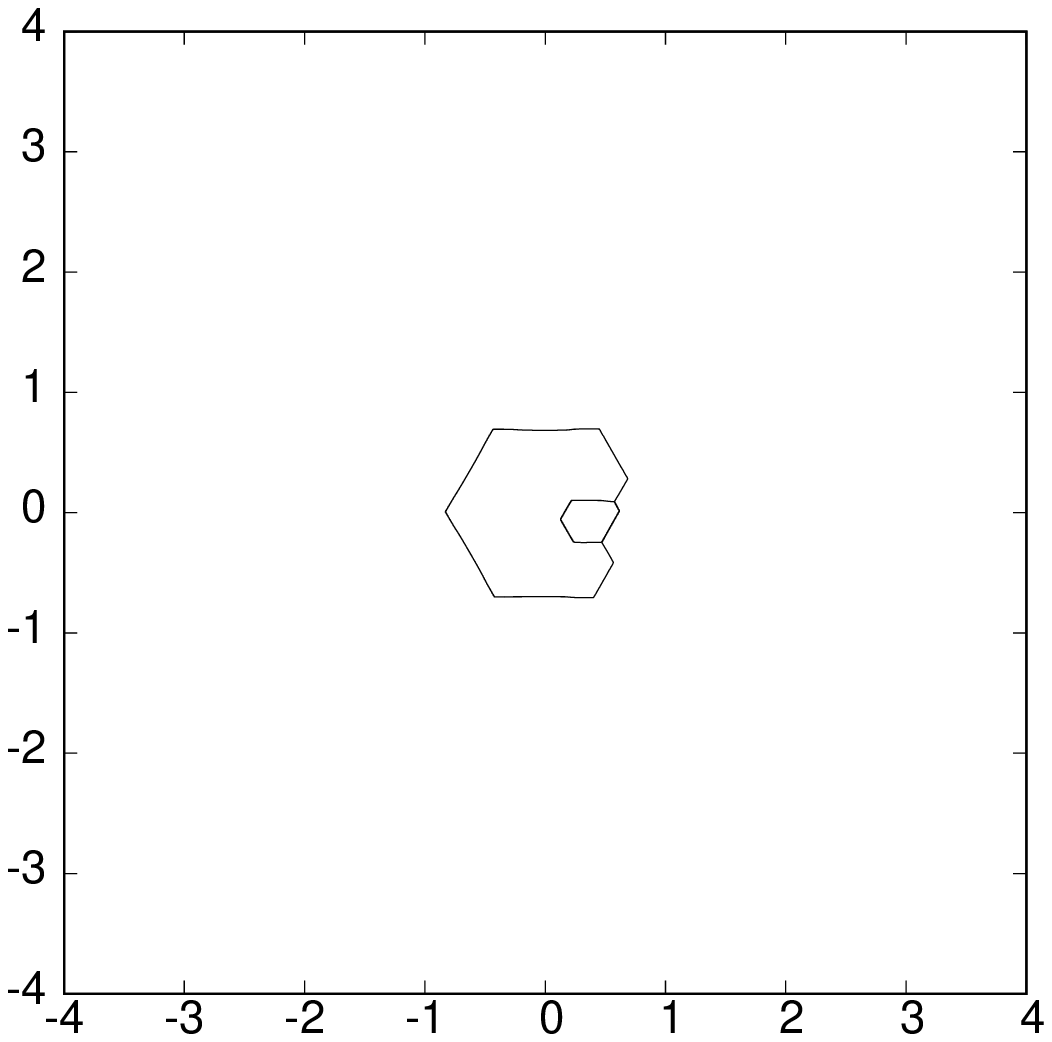}
\includegraphics[angle=-90,width=0.18\textwidth]{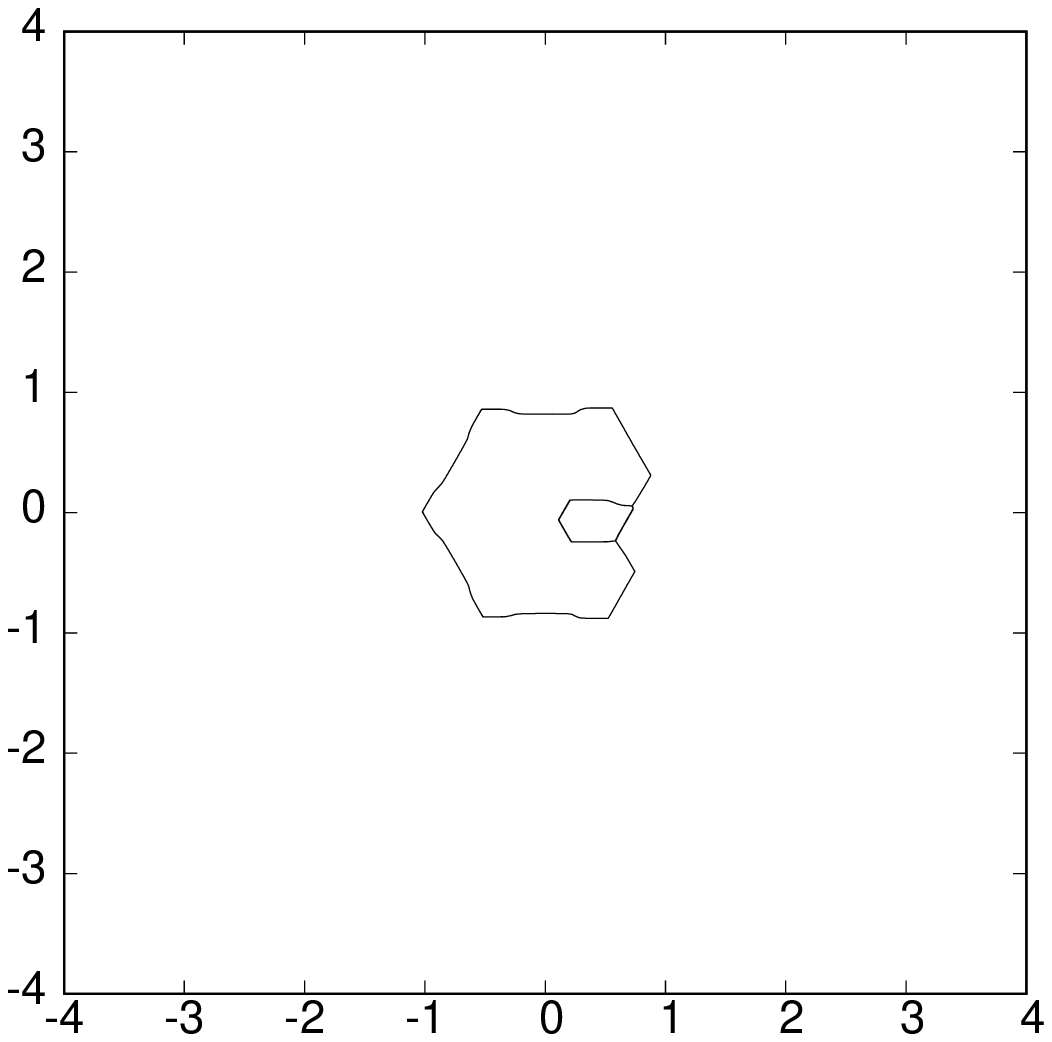}
\includegraphics[angle=-90,width=0.18\textwidth]{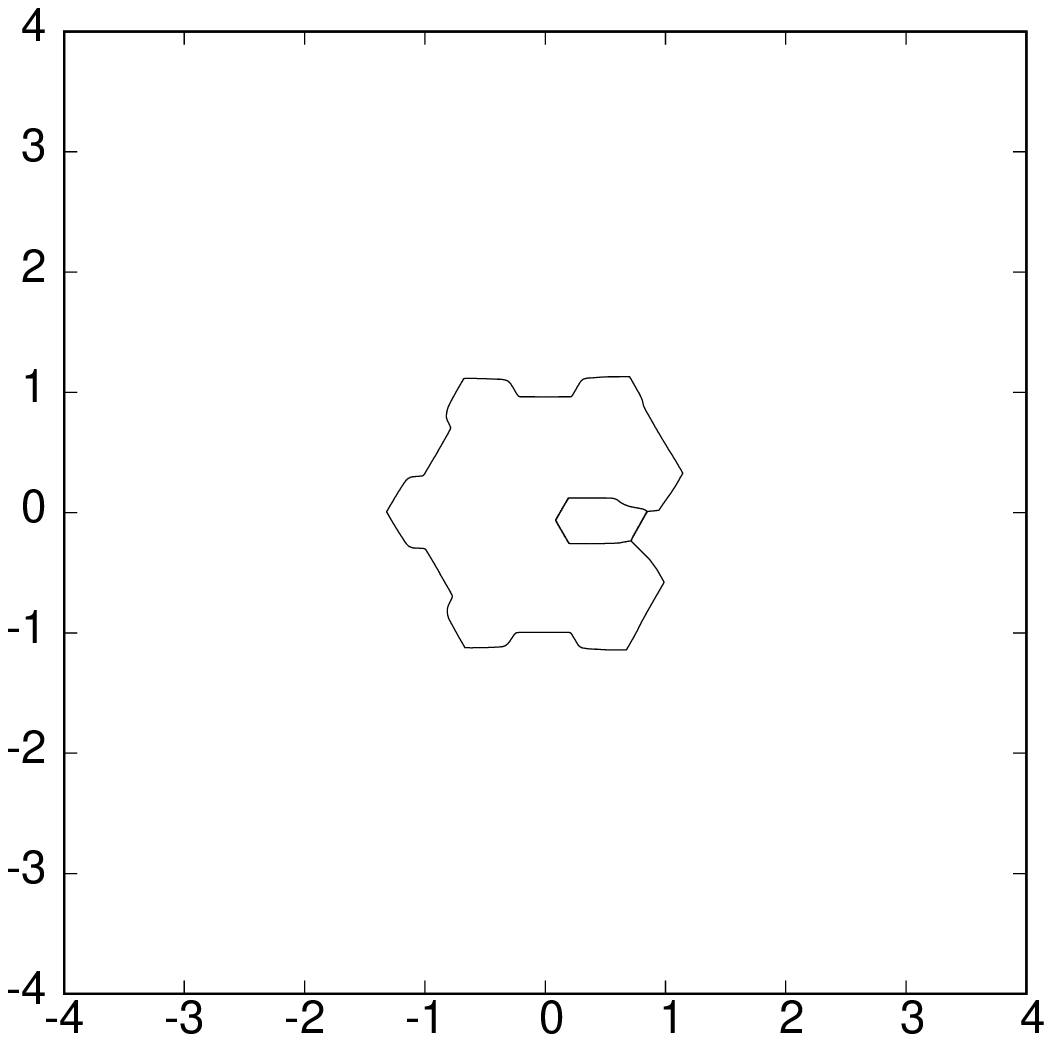}
\includegraphics[angle=-90,width=0.18\textwidth]{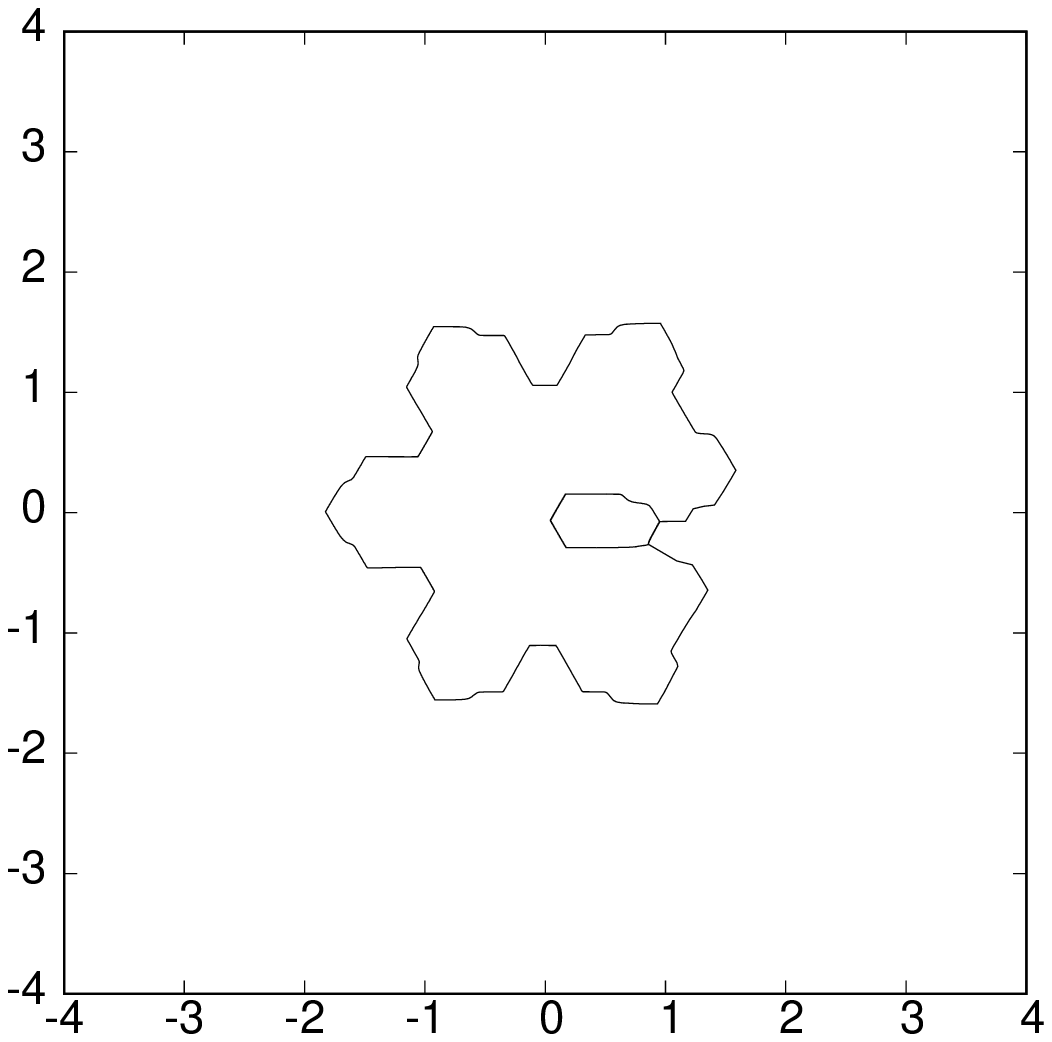}
\includegraphics[angle=-90,width=0.18\textwidth]{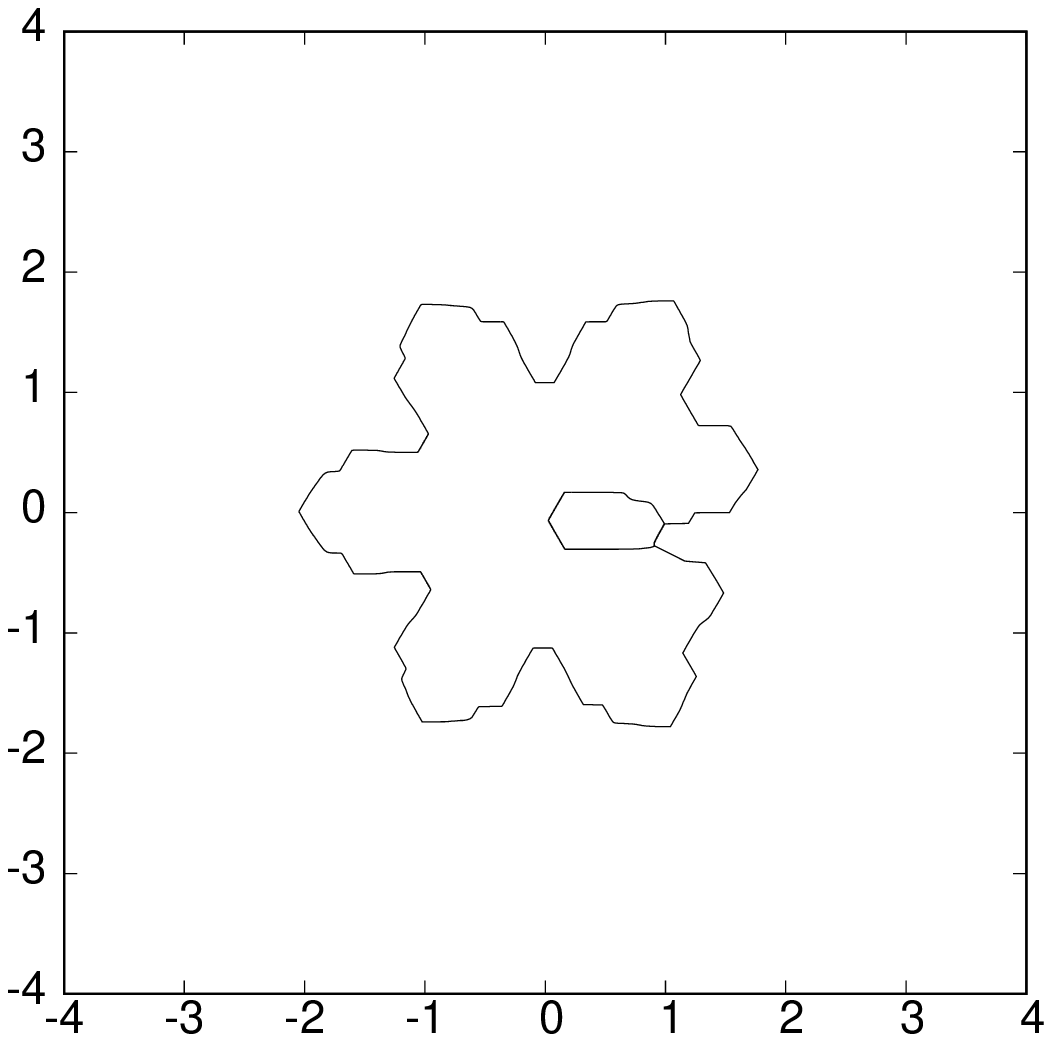}
\includegraphics[angle=-90,width=0.18\textwidth]{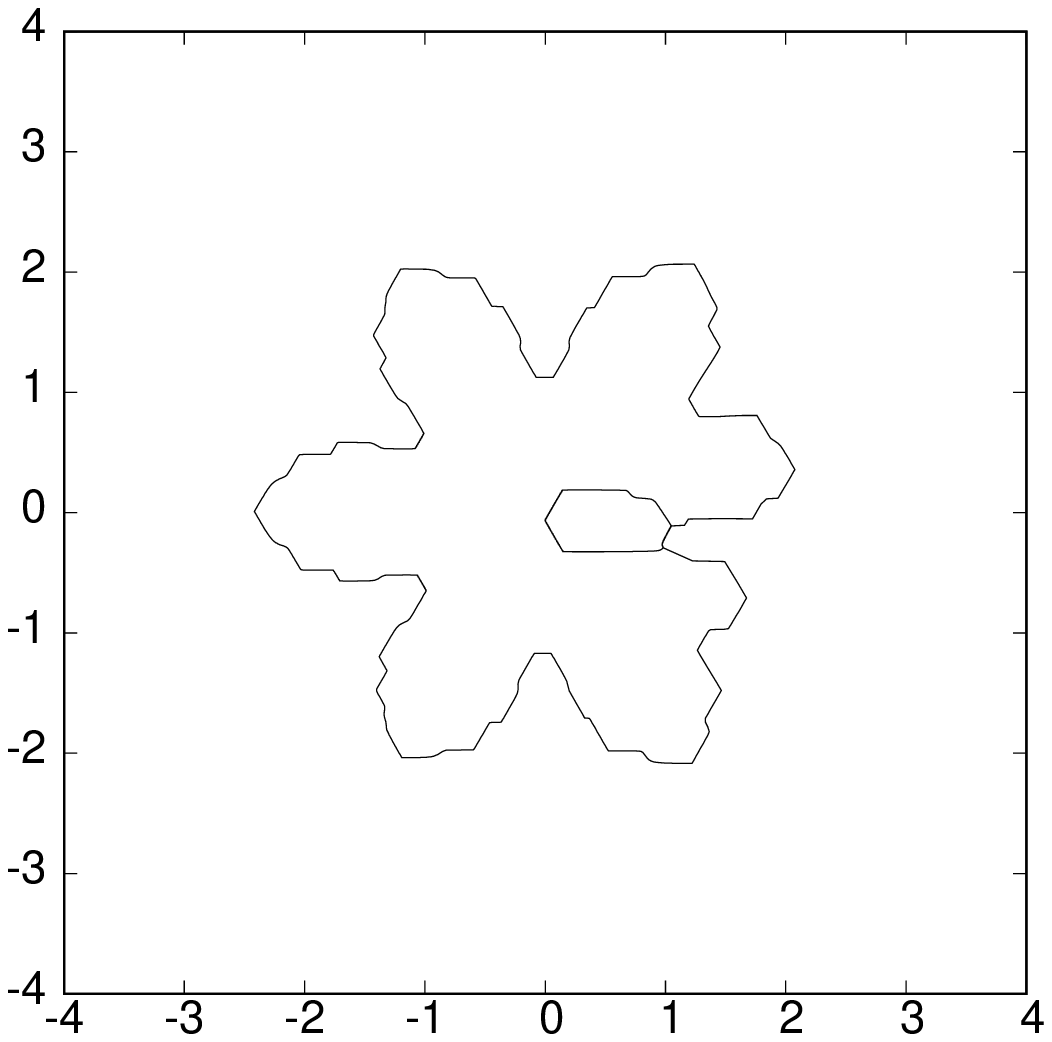}
\includegraphics[angle=-90,width=0.18\textwidth]{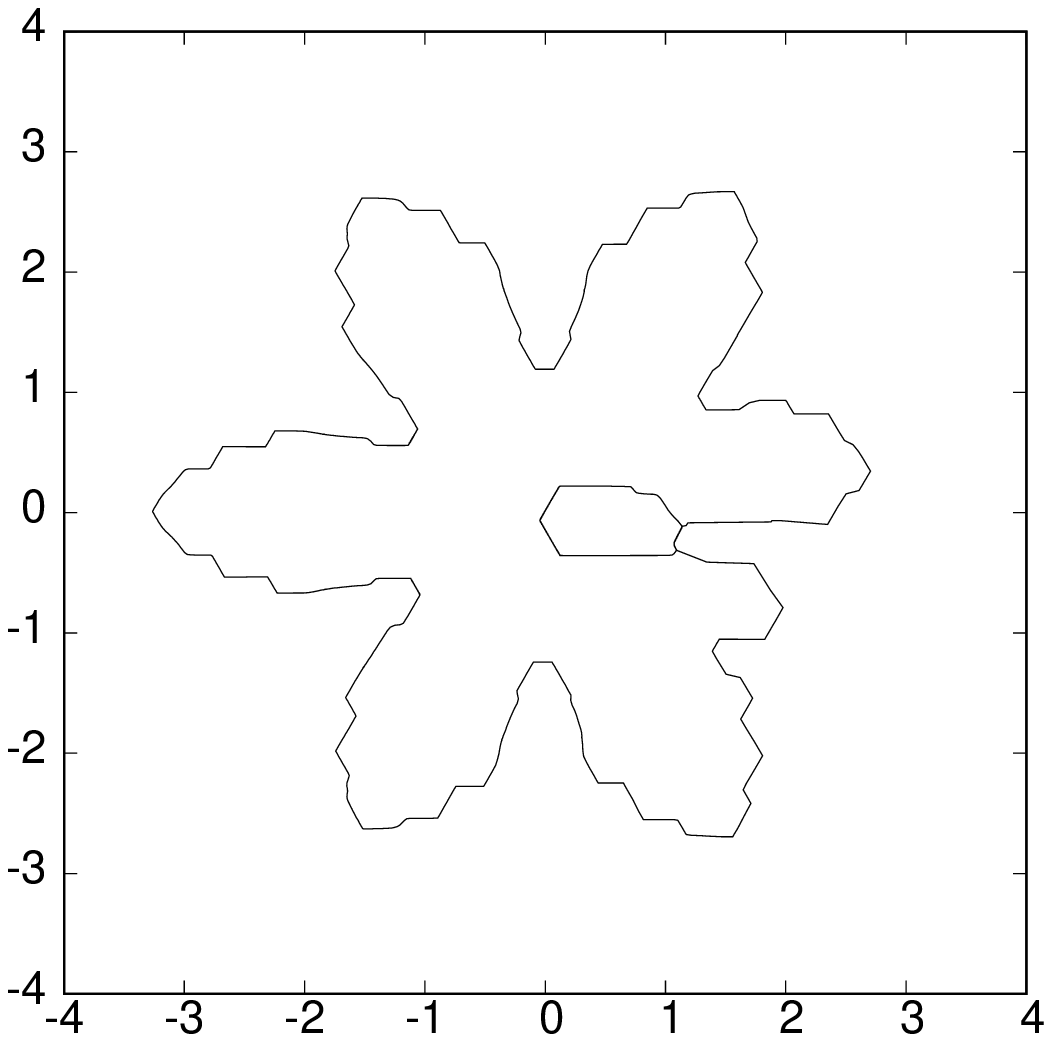}
\includegraphics[angle=-90,width=0.25\textwidth]{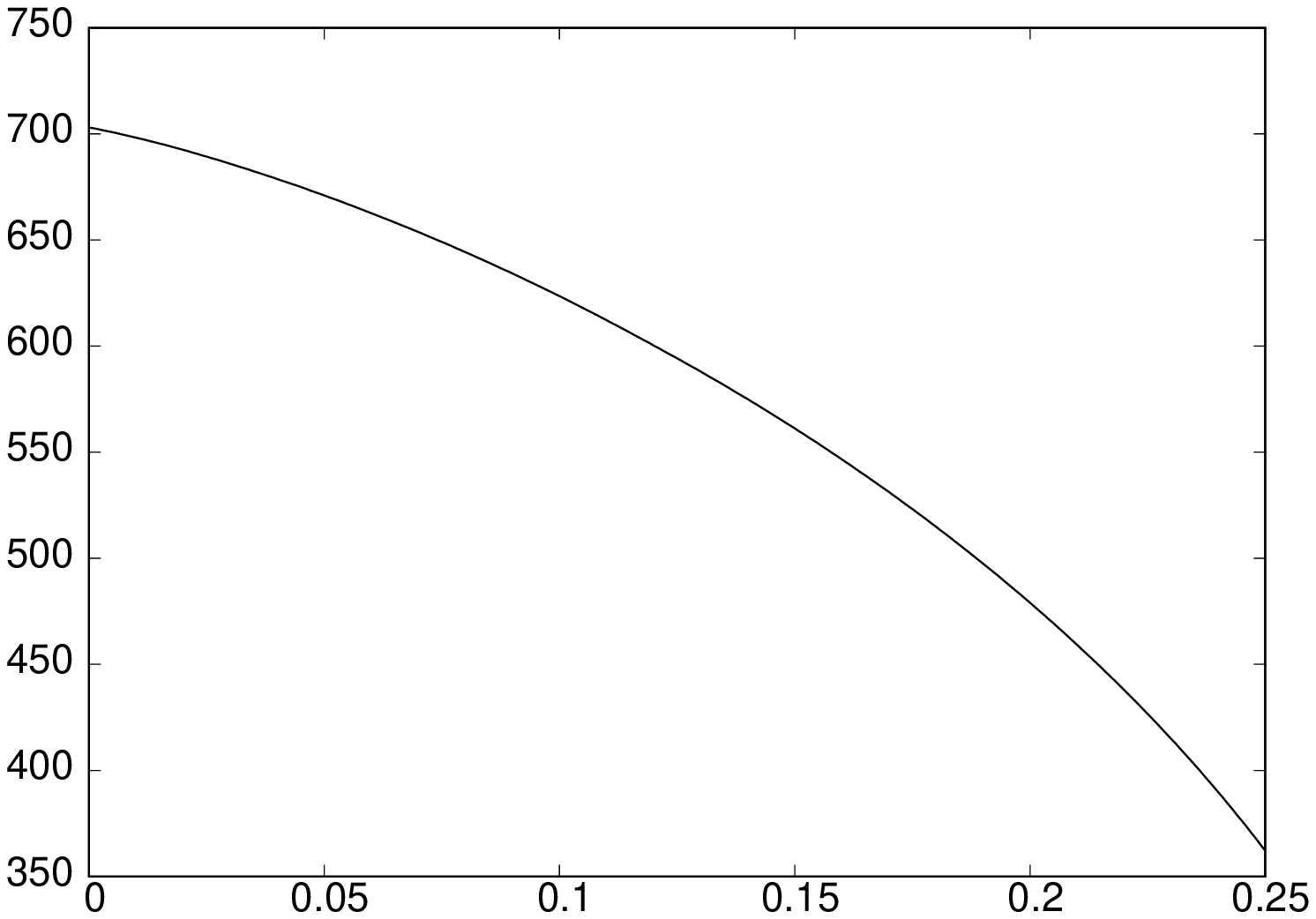} 
\caption{($\bv w_D = (1, 10, -11)^\top$, $\rho=0.05$, $\alpha=0.02$)
The solution at times $t=0, 0.01, \ldots, 0.3$, and separately at times 
$t=0.02$, $t=0.05$, $t=0.07$, $t=0.1$, $t=0.15$, $t=0.17$, $t=0.2$ and $t=0.25$,
and a plot of the discrete energy over time.
}
\label{fig:2dDbc_hex_db_rho_005_1_to_10}
\end{figure}%

\subsection{Numerical results in 3D}\label{sec:NumericalResults3d}

In this subsection we consider some numerical simulations for $d=3$ and
$\pOmega_D = \pOmega$. For the possible anisotropies we define first of all
\begin{equation} \label{eq:l1normL3}
\gamma_{\ell^1}(\vec{p}) = \sum_{i=1}^3 \sqrt{ \delta^2|\vec{p}|^2 +
p_i^2(1-\delta^2)},\quad \delta = 0.01,
\end{equation}
which approximates the $\ell^1$--norm of $\vec p$, see
\cite{BarrettGarckeNurnberg2008NM}. In addition, a 3D analogue of
\eqref{eq:hex2d}, compare with \cite[(13)]{jcg}, is defined by
\begin{equation} \label{eq:L44}
\gamma_{\rm hex}(\vec{p})
:= l_\delta(R_2(\tfrac{\pi}2)\,\vec{p}) + \tfrac{1}{\sqrt{3}}
\sum_{\ell = 1}^3
l_\delta(R_1(\theta_0 + \ell\,\tfrac{\pi}3)\,\vec{p}),\quad \delta = 0.01,
\end{equation}
where $R_{1}(\theta):=\left(\!\!\!\scriptsize
\begin{array}{rrr} \cos\theta & \sin\theta&0 \\
-\sin\theta & \cos\theta & 0 \\ 0 & 0 & 1 \end{array}\!\! \right)$ and 
$R_{2}(\theta):=\left(\!\!\!\scriptsize
\begin{array}{rrr} \cos\theta & 0 & \sin\theta \\
0 & 1 & 0 \\ -\sin\theta & 0 & \cos\theta \end{array}\!\! \right)$
are rotation matrices, and where
$l_\delta(\vec{p}) := 
\left[ \delta^2\,|\vec{p}|^2  + p_1^2\,(1-\delta^2) \right]^{\frac12}$.
The Wulff shape of the anisotropy \eqref{eq:L44} is given by a
smoothed hexagonal prism, see e.g.\ \cite[Fig.~3]{jcg}. 

In order to be able to vary the kinetic coefficient $\beta$ for the simulations
in this subsection, we define
\begin{equation} \label{eq:betaflat}
\beta_{\rm flat}(\vec{p}) = \beta_{\rm flat,\ell}(\vec{p}) 
:= \sqrt{p_1^2 + p_2^2 + 10^{-2\ell}\,p_3^2}
\end{equation}
with $\ell\in\mathbb{N}$. 
For the surface clusters we always choose a double bubble, so
that $I_S = 3$, $I_R = 3$, $I_T = 1$, 
$(\curveindex{1}{1},\curveindex{1}{2},\curveindex{1}{3}) = (1,2,3)$
and
$\dcmap = \begin{pmatrix}0 &-1&1 \\ 1 & 0 &-1 \\ -1 &1&0 \end{pmatrix}$.

\vspace{0.3cm}
\noindent
{\bf Example 7}:
On the boundary $\pOmega=\pOmega_D$ we choose the undercooling parameters
$\bv w_D = (2, 1, -3)^\top$, and start with a seed consisting of a
standard double bubble. The two bubbles of the double bubble enclose a
volume of about $2.1$ each. We also let $\rho=1$.
For the anisotropies we choose $\gamma_i = \gamma_{\ell^1}$, 
$i\in\naturalset{I_S}$, recall \eqref{eq:l1normL3}. 
The evolution is shown in Figure~\ref{fig:3dL3_db}.
\begin{figure}
\center
\includegraphics[angle=-0,width=0.18\textwidth,valign=b]{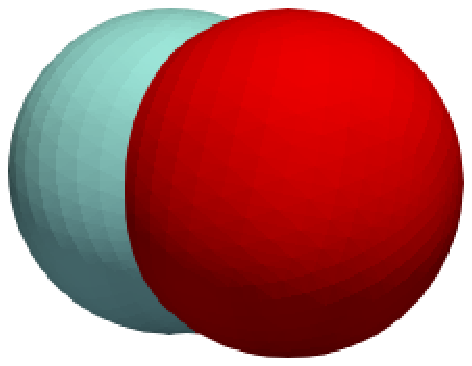}
\includegraphics[angle=-0,width=0.18\textwidth,valign=b]{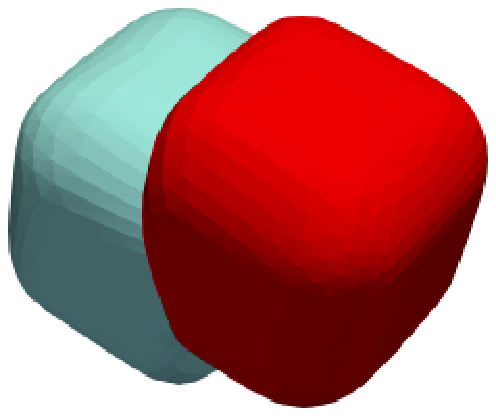}
\includegraphics[angle=-0,width=0.18\textwidth,valign=b]{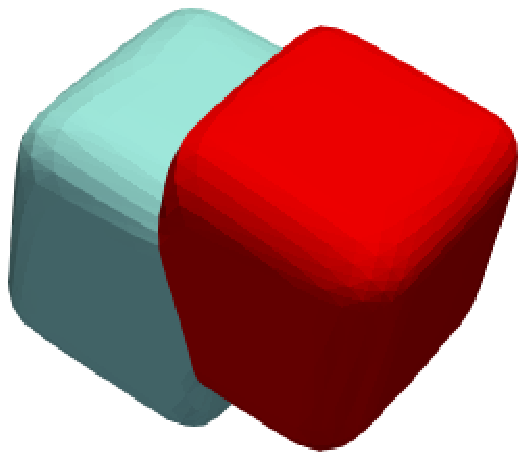}
\includegraphics[angle=-0,width=0.18\textwidth,valign=b]{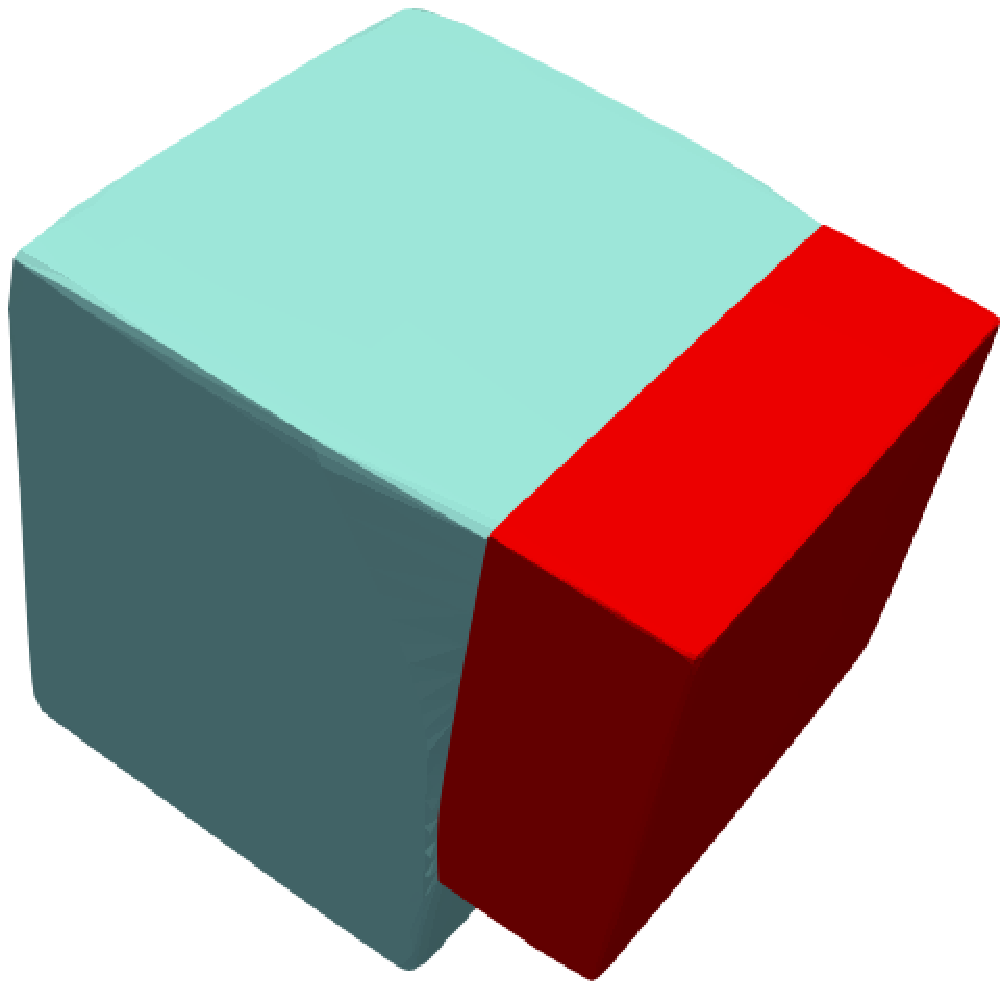}
\includegraphics[angle=-90,width=0.25\textwidth,valign=b]{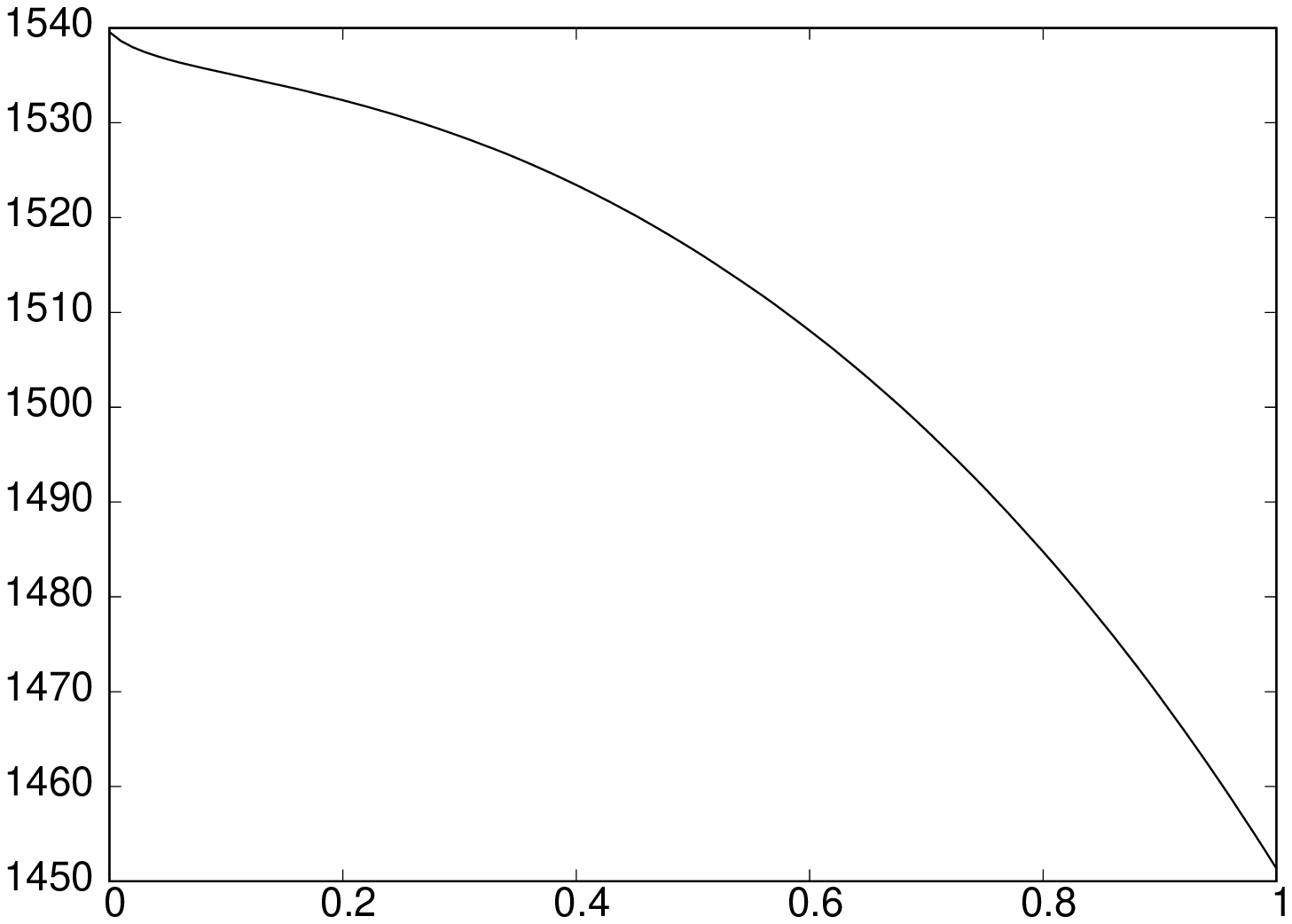}
\caption{($\bv w_D = (2, 1, -3)^\top$, $\rho=1$)
The solution at times $t=0, 0.05, 0.1, 1$,
and a plot of the discrete energy over time. 
}
\label{fig:3dL3_db}
\end{figure}%
When we choose $\gamma_i = \gamma_{\rm hex}$, 
$i\in\naturalset{I_S}$, recall \eqref{eq:L44}, we obtain
the evolution shown in Figure~\ref{fig:3dL4_db}.
\begin{figure}
\center
\includegraphics[angle=-0,width=0.18\textwidth,valign=b]{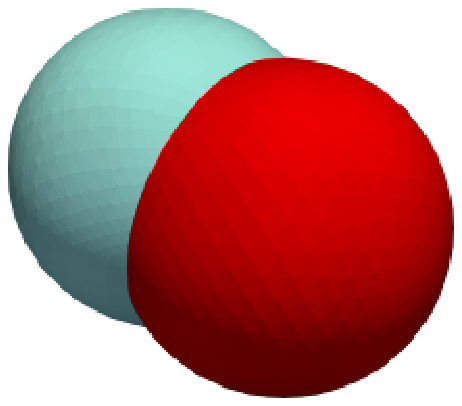}
\includegraphics[angle=-0,width=0.18\textwidth,valign=b]{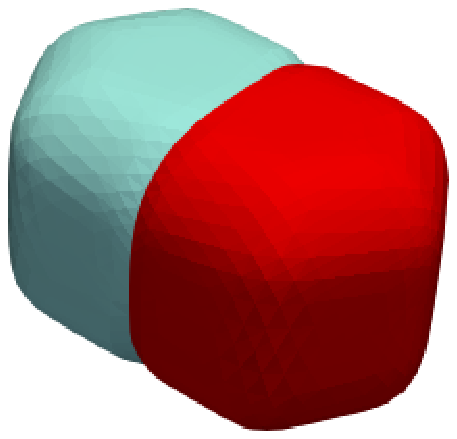}
\includegraphics[angle=-0,width=0.18\textwidth,valign=b]{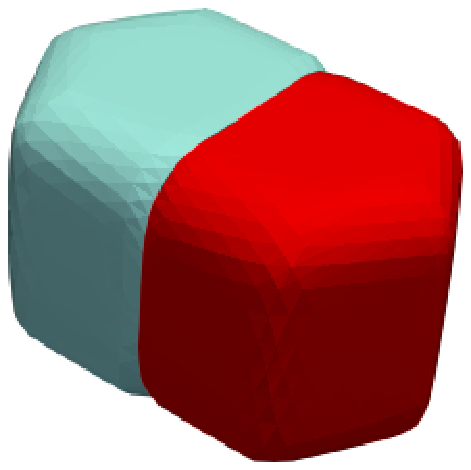}
\includegraphics[angle=-0,width=0.18\textwidth,valign=b]{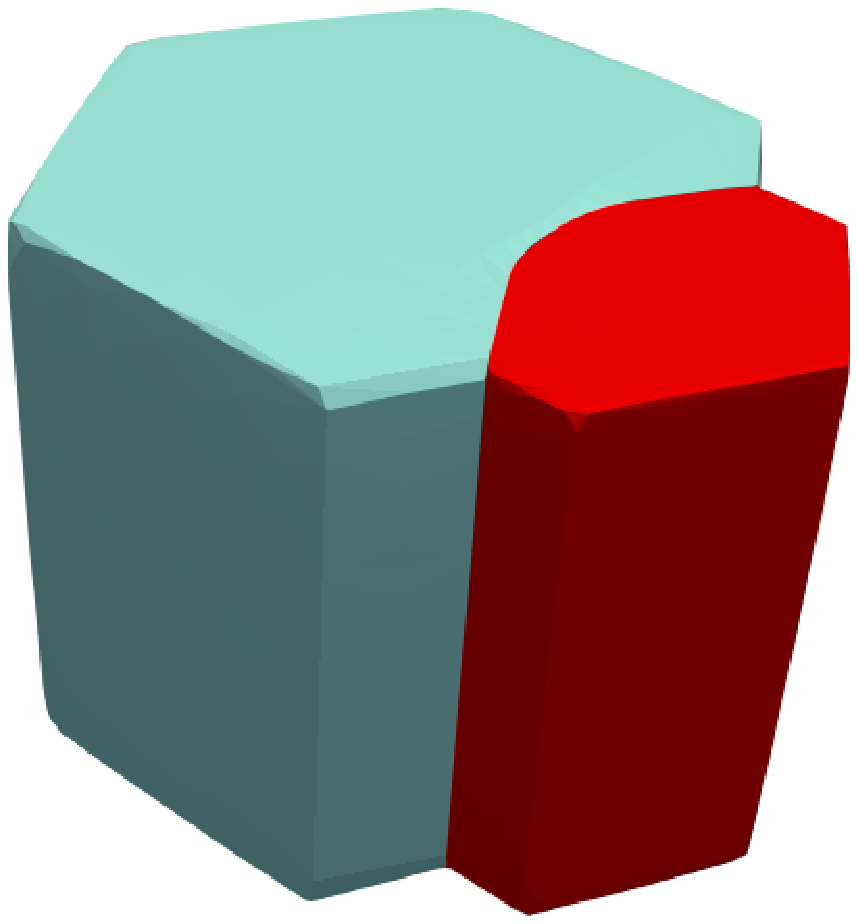}
\includegraphics[angle=-90,width=0.25\textwidth,valign=b]{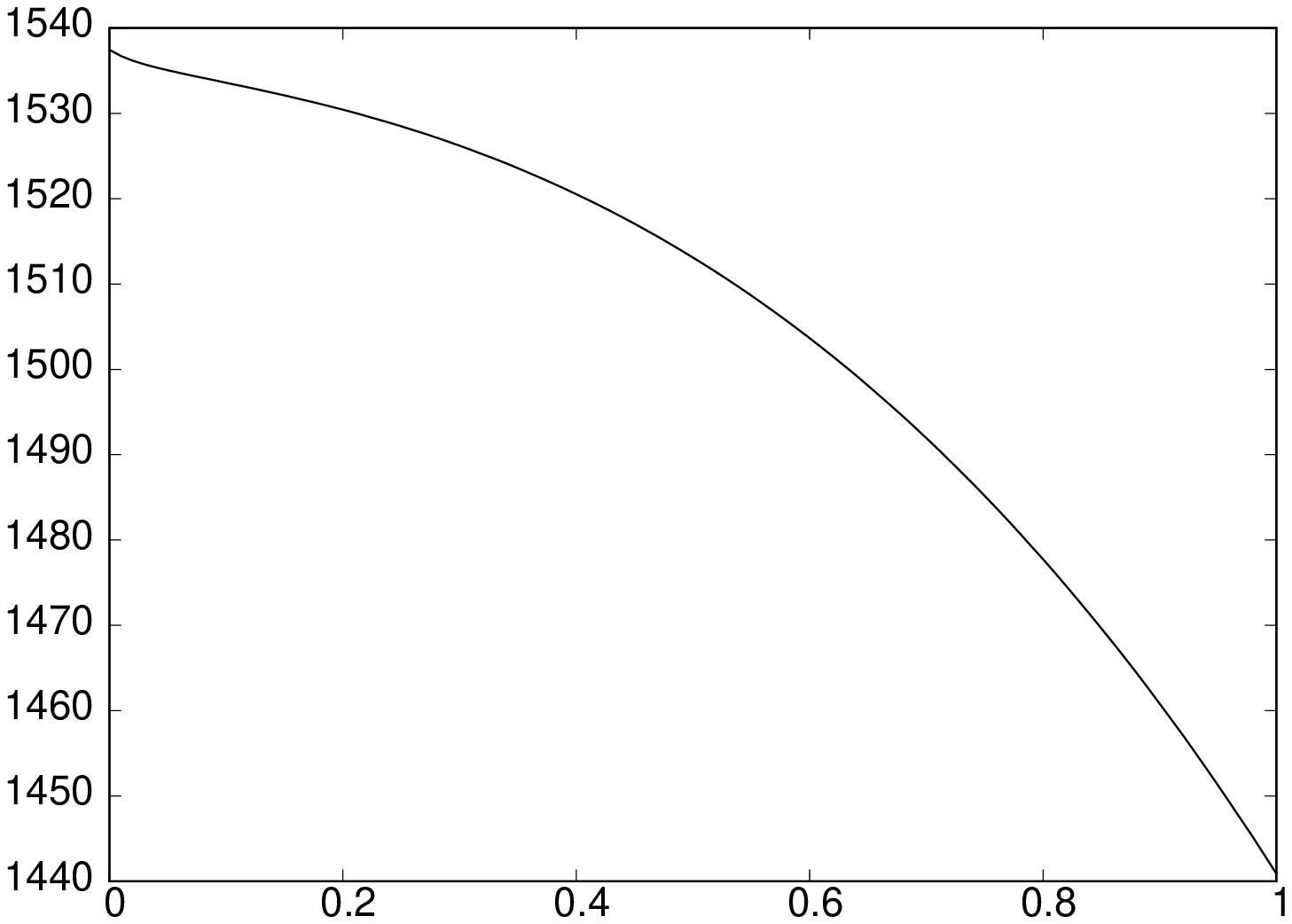}
\caption{($\bv w_D = (2, 1, -3)^\top$, $\rho=1$)
The solution at times $t=0, 0.05, 0.1, 1$,
and a plot of the discrete energy over time. 
}
\label{fig:3dL4_db}
\end{figure}%

\vspace{0.3cm}
\noindent
{\bf Example 8}:
We use the same setup as in Example~7, but now let $\rho=0.05$.
In addition, we choose $(\gamma_1,\gamma_2,\gamma_3) = 
\alpha (\gamma_{\rm hex},\gamma_{\rm hex},\gamma_{\rm hex})$, 
$i\in\naturalset{I_S}$, recall \eqref{eq:L44}, with $\alpha=0.05$.
Moroever, we choose either $\beta_i \equiv 1$ or
$\beta_i = \beta_{\rm flat,3}$, $i\in\naturalset{I_S}$, 
recall \eqref{eq:betaflat}.
The evolutions are shown in Figures~\ref{fig:3dL4_db_noflat} and
\ref{fig:3dL4_db_flat}, respectively.
\begin{figure}
\center
\includegraphics[angle=-0,width=0.18\textwidth,valign=b]{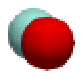}
\includegraphics[angle=-0,width=0.18\textwidth,valign=b]{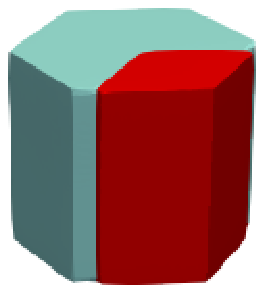}
\includegraphics[angle=-0,width=0.18\textwidth,valign=b]{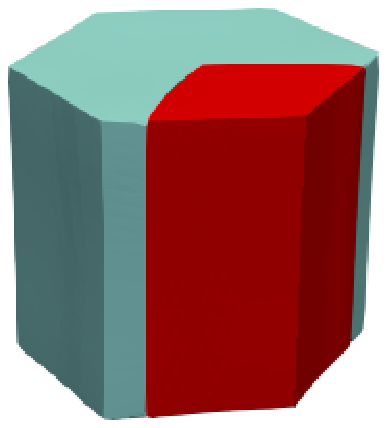}
\includegraphics[angle=-0,width=0.18\textwidth,valign=b]{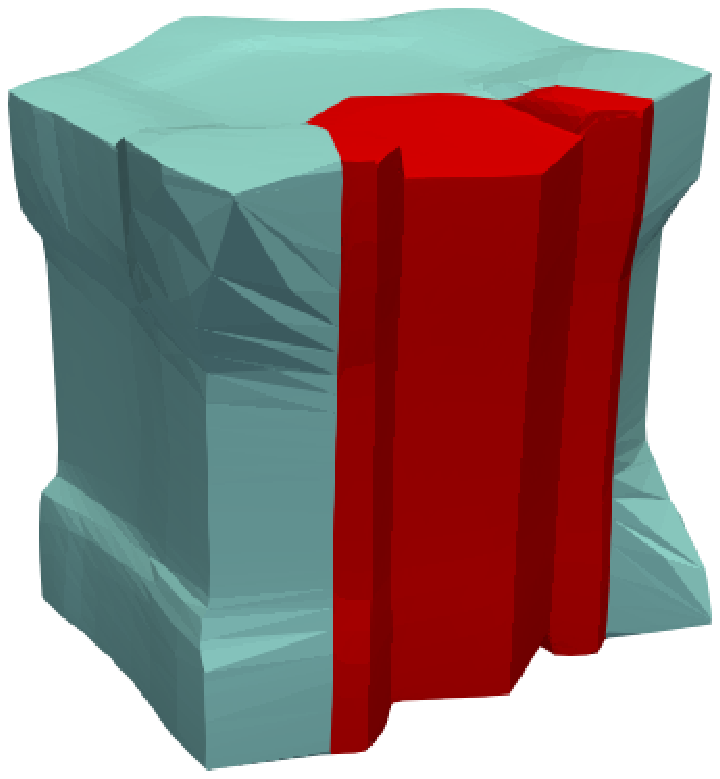}
\includegraphics[angle=-90,width=0.25\textwidth,valign=b]{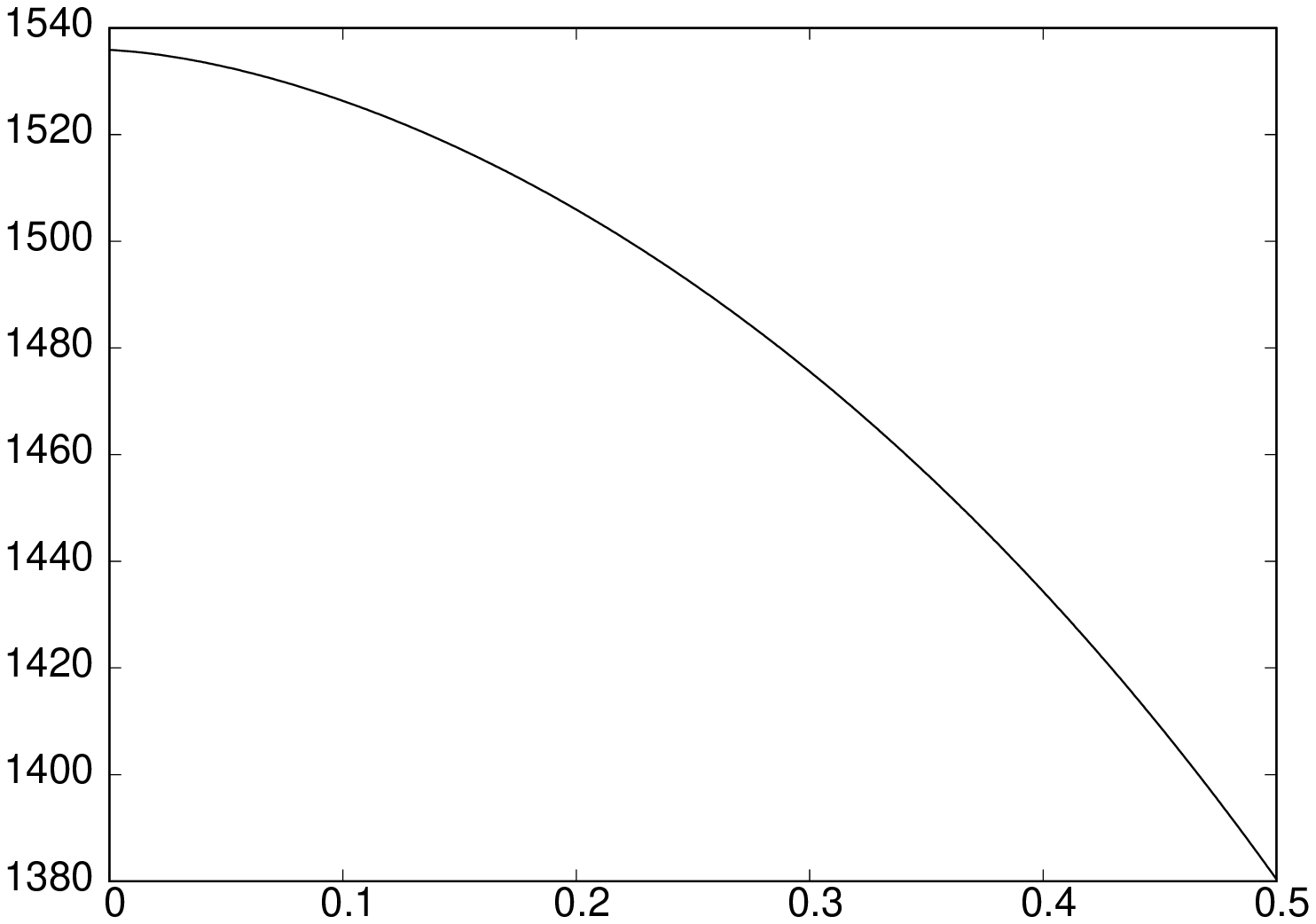}
\caption{($\bv w_D = (2, 1, -3)^\top$, $\rho=\alpha=0.05$)
The solution at times $t=0, 0.1, 0.2, 0.5$,
and a plot of the discrete energy over time. 
}
\label{fig:3dL4_db_noflat}
\end{figure}%
\begin{figure}
\center
\includegraphics[angle=-0,width=0.18\textwidth,valign=b]{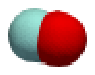}
\includegraphics[angle=-0,width=0.18\textwidth,valign=b]{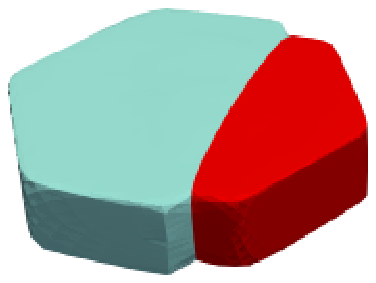}
\includegraphics[angle=-0,width=0.18\textwidth,valign=b]{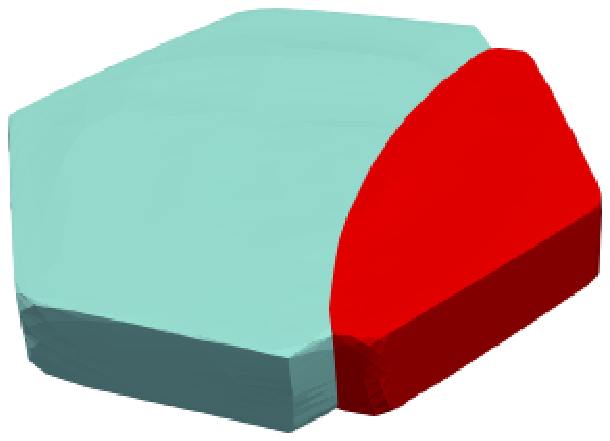}
\includegraphics[angle=-0,width=0.18\textwidth,valign=b]{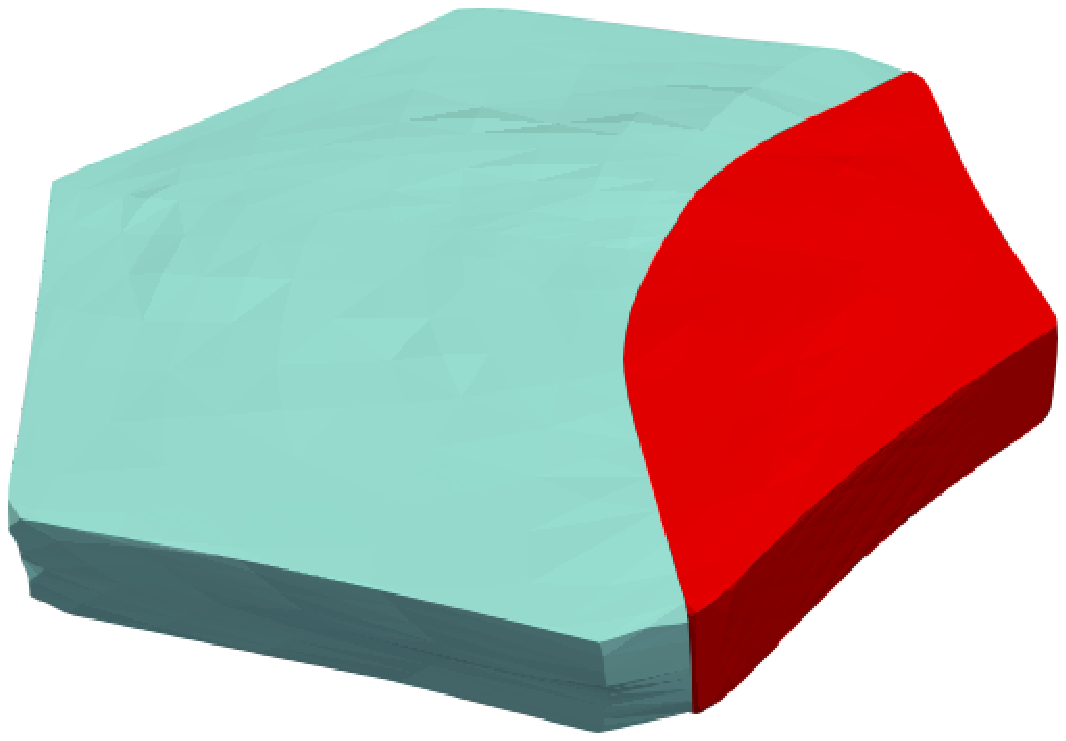}
\includegraphics[angle=-90,width=0.25\textwidth,valign=b]{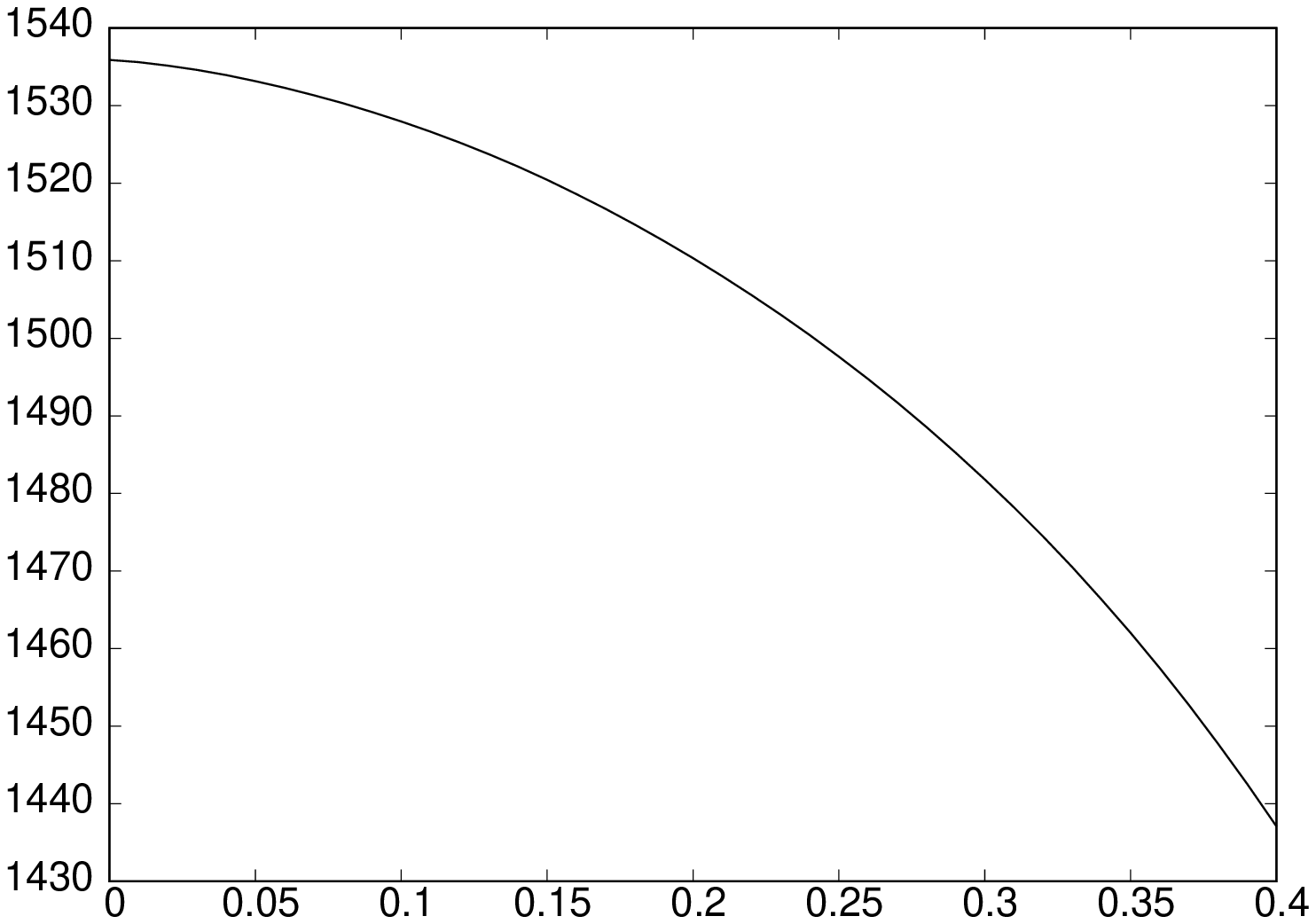}
\caption{($\bv w_D = (2, 1, -3)^\top$, $\rho=\alpha=0.05$, 
$\beta = \beta_{\rm flat,3}$)
The solution at times $t=0, 0.1, 0.2, 0.4$,
and a plot of the discrete energy over time. 
}
\label{fig:3dL4_db_flat}
\end{figure}%

\vspace{0.3cm}
\noindent
{\bf Example 9}:
We use a similar setup to Example~7, but now let $\rho=0.05$
and set $\bv w_D = (2, 0.2, -2.2)^\top$. We also start with a smaller seed. In
fact, the two bubbles of the initial double bubble both enclose a volume of
about $0.017$.
In addition, we choose choose $(\gamma_1,\gamma_2,\gamma_3) = 
\alpha (\gamma_{\rm hex},\gamma_{\rm hex},\gamma_{\rm hex})$, 
$i\in\naturalset{I_S}$, recall \eqref{eq:L44}, with $\alpha=0.05$.
The evolution is shown in Figure~\ref{fig:3dL4_db_202}.
\begin{figure}
\center
\includegraphics[angle=-0,width=0.18\textwidth,valign=b]{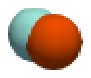}
\includegraphics[angle=-0,width=0.18\textwidth,valign=b]{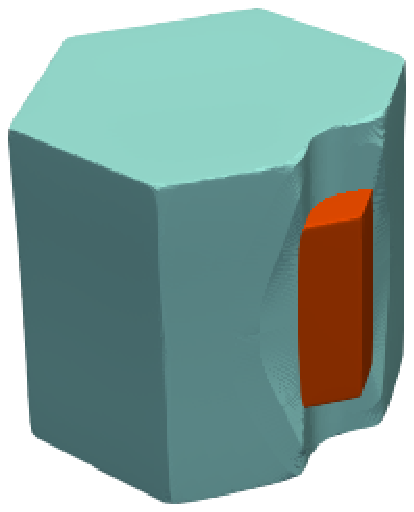}
\includegraphics[angle=-0,width=0.18\textwidth,valign=b]{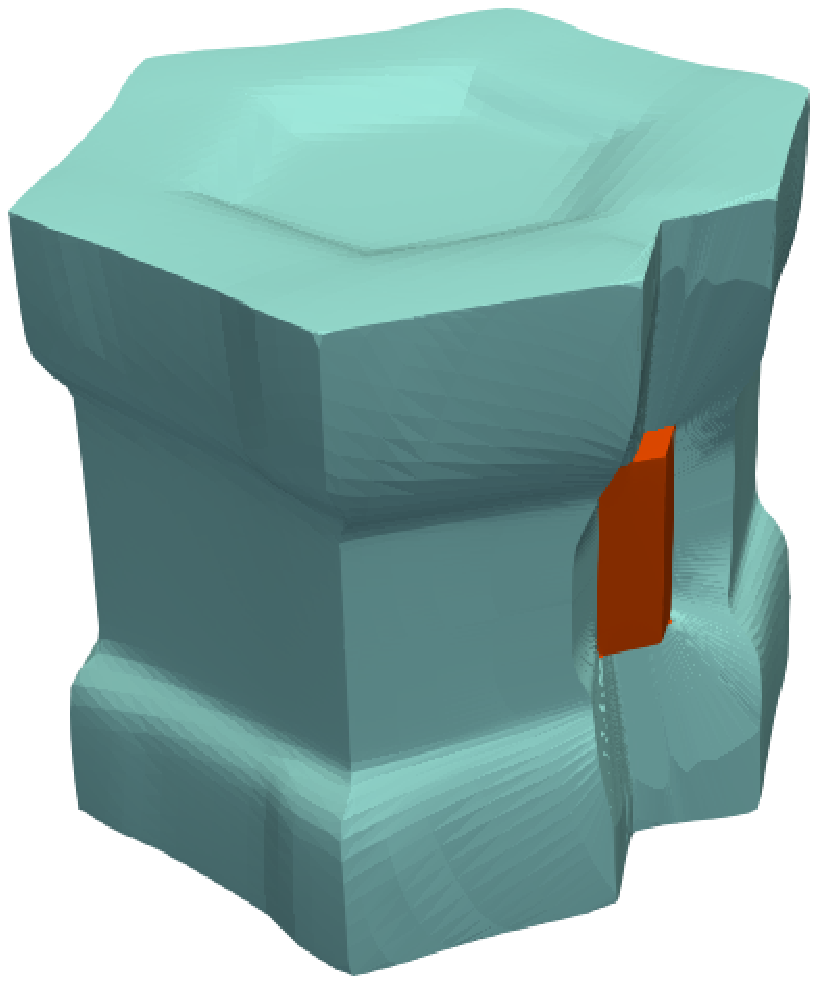}
\includegraphics[angle=-0,width=0.18\textwidth,valign=b]{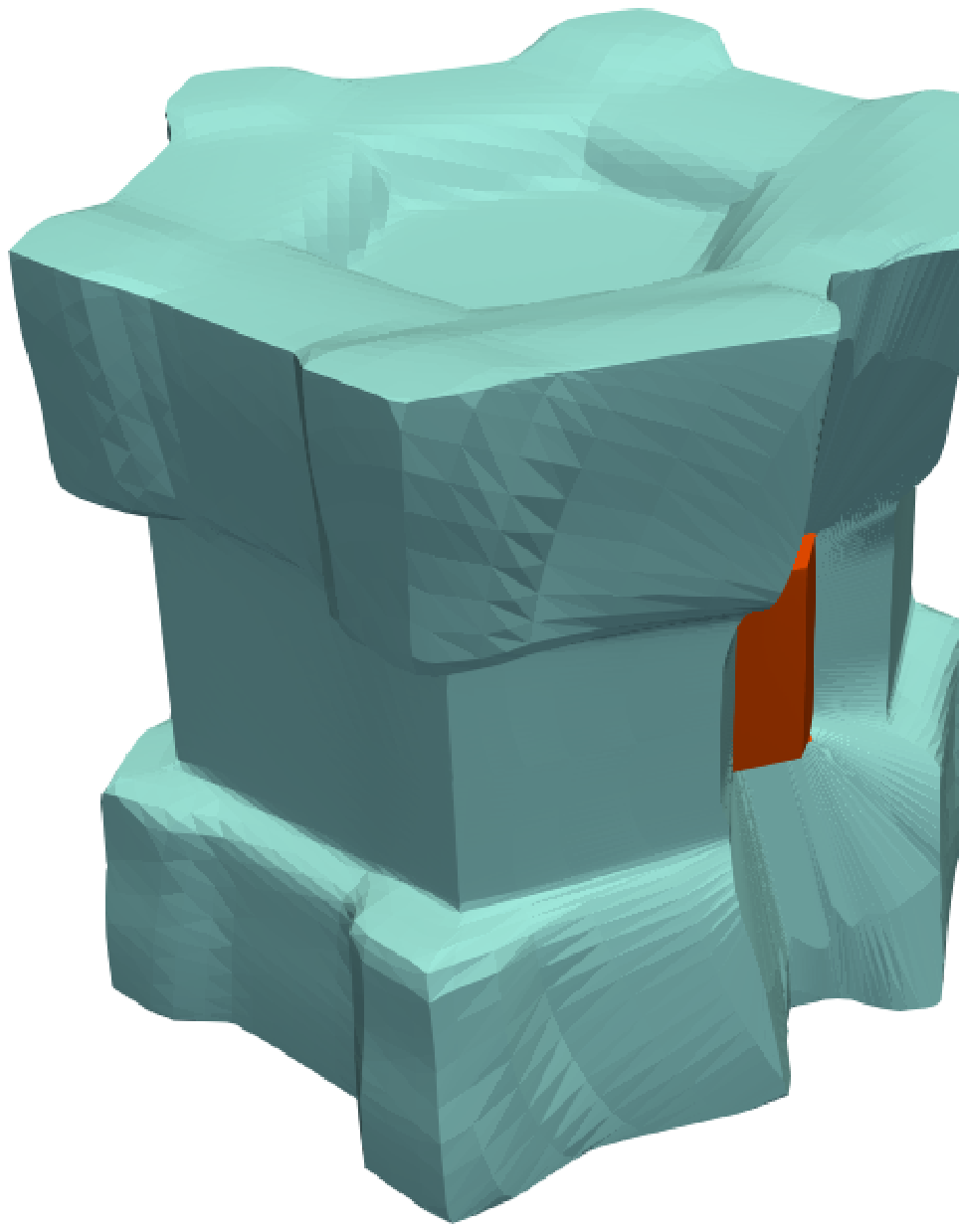}
\includegraphics[angle=-90,width=0.25\textwidth,valign=b]{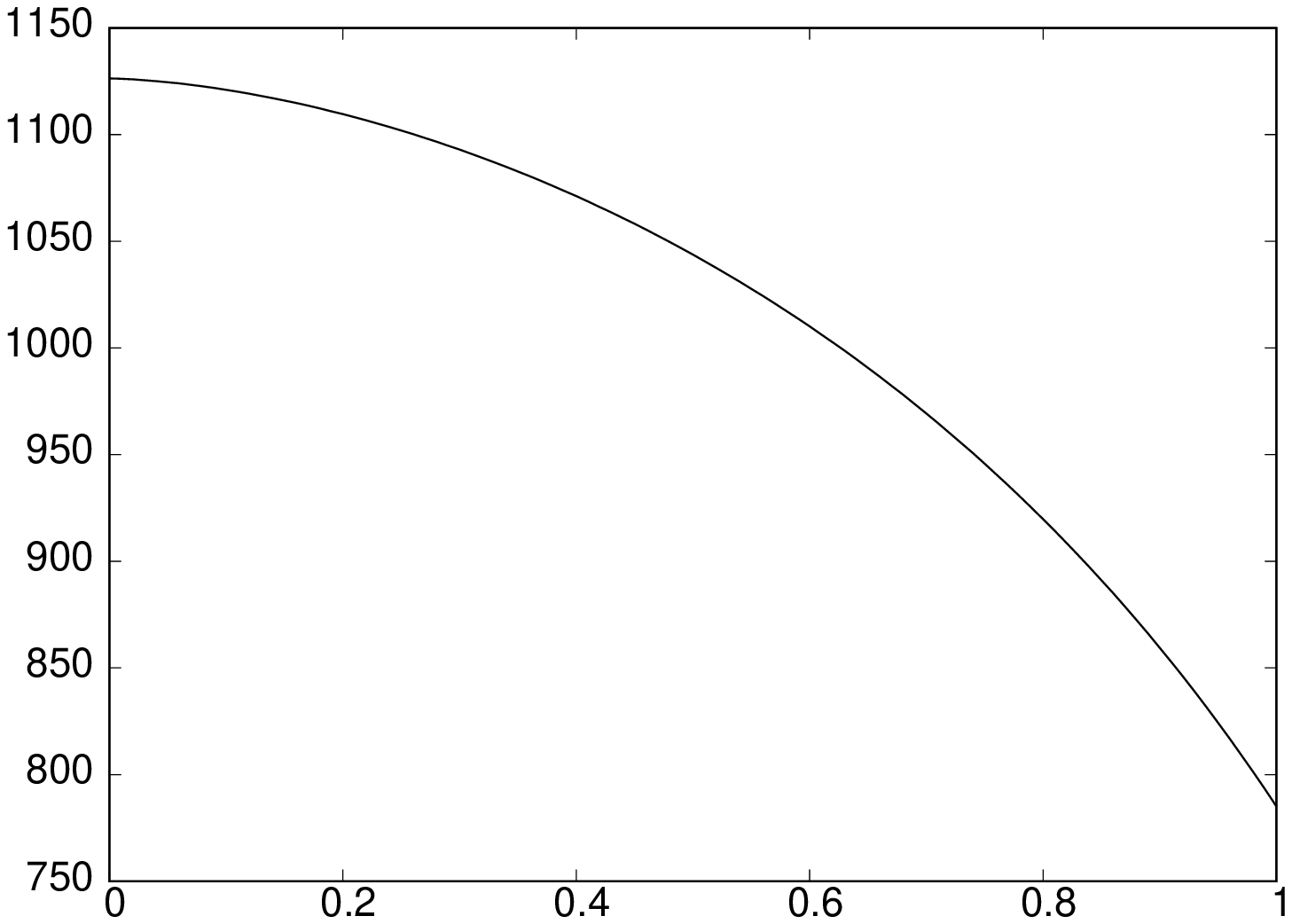}
\caption{($\bv w_D = (2, 0.3, -2.2)^\top$, $\rho=\alpha=0.05$)
The solution at times $t=0, 0.2, 0.8, 1$,
and a plot of the discrete energy over time. 
}
\label{fig:3dL4_db_202}
\end{figure}%

\bibliographystyle{siam}
\bibliography{cite}

\end{document}